\documentclass[11pt]{amsart}
\usepackage{fullpage,xcolor,graphicx}
\usepackage[OT2,T1]{fontenc}
\usepackage{hyperref}
\usepackage{pifont}
\usepackage{amsthm}
\usepackage{amsfonts}
\usepackage{amssymb}
\usepackage[mathscr]{euscript}
\usepackage[all]{xy}
\usepackage{amsmath}
\usepackage{epsfig}
\usepackage{latexsym}
\usepackage{stackengine}
\usepackage{undertilde}
\usepackage{hhline}
\usepackage{bbm}
\usepackage{MnSymbol}
\usepackage{multirow}
\usepackage{rotating}
\usepackage[OT2,OT1]{fontenc} 

\newcommand*\wbar[1]{
  \hbox{ \kern-0.2em%
    \vbox{%
      \hrule height 0.5pt  
      \kern0.25ex
      \hbox{%
        \kern-0.15em
        \ensuremath{#1}%
        \kern-0.05em
      }%
    }%
  \kern0.05em}%
} 

\newcommand*\wbarnew[1]{
  \hbox{ \kern-0.2em%
    \vbox{%
      \hrule height 0.5pt  
      \kern0.25ex
      \hbox{%
        \kern-0.35em
        \ensuremath{#1}%
        \kern-0.05em
      }%
    }%
  \kern0.05em}%
} 

\newtheorem{theorem}{Theorem}[subsection]
\newtheorem{lemma}[theorem]{Lemma}
\newtheorem{proposition}[theorem]{Proposition}

\theoremstyle{definition}
\newtheorem{definition}[theorem]{Definition}
\newtheorem{example}[theorem]{Example}
\newtheorem{notation}[theorem]{Notation}
\theoremstyle{remark}
\newtheorem{remark}[theorem]{Remark}

\newcommand{\minitab}[2][l]{\begin{tabular}{#1}#2\end{tabular}}

\makeatletter
\@namedef{subjclassname@2020}{\textup{2020} Mathematics Subject Classification}
\makeatother

\makeatletter
\newcommand*\rel@kern[1]{\kern#1\dimexpr\macc@kerna}
\newcommand*\widebar[1]{%
  \begingroup
  \def\mathaccent##1##2{%
    \rel@kern{0.8}%
    \overline{\rel@kern{-0.8}\macc@nucleus\rel@kern{0.2}}%
    \rel@kern{-0.2}%
  }%
  \macc@depth\@ne
  \let\math@bgroup\@empty \let\math@egroup\macc@set@skewchar
  \mathsurround\z@ \frozen@everymath{\mathgroup\macc@group\relax}%
  \macc@set@skewchar\relax
  \let\mathaccentV\macc@nested@a
  \macc@nested@a\relax111{#1}%
  \endgroup
}
\makeatother

\title{A geometric foundation of virtual knot theory}

\author[M. Chrisman]{Micah Chrisman}
\address{Department of Mathematics, The Ohio State University, Columbus, Ohio, 43210}
\email{chrisman.76@osu.edu}
\urladdr{https://micah46.wixsite.com/micahknots}


\begin{document}
\begin{abstract} Virtual knots are defined diagrammatically as a collection of figures, called virtual knot diagrams, that are considered equivalent up to finite sequences of extended Reidemeister moves. By contrast, knots in $\mathbb{R}^3$ can be defined geometrically. They are the points of a space $\mathbb{K}$ of knots. The knot space has a topology so that equivalent knots lie in the same path component. The aim of this paper is to use sheaf theory to obtain a fully geometric model for virtual knots. The geometric model formalizes the intuitive notion that a virtual knot is an actual knot residing in a variable ambient space; the usual diagrammatic theory follows as in the classical case. To do this, it is shown that there exists a site $(\textbf{VK}, J_{\textbf{VK}})$ so that its category $\text{Sh}(\textbf{VK},J_{\textbf{VK}})$ of sheaves can be naturally interpreted as the ``space of virtual knots''. A point of this Grothendieck topos, that is a geometric morphism $\textbf{Sets} \to \text{Sh}(\textbf{VK})$, is a virtual knot. The virtual isotopy relation is generated by paths in this space, or more precisely, geometric morphisms $\text{Sh}([0,1]) \to \text{Sh}(\textbf{VK},J_{\textbf{VK}})$. Virtual knot invariants valued in a discrete topological space $\mathbb{G}$ are geometric morphisms $\text{Sh}(\textbf{VK},J_{\textbf{VK}}) \to \text{Sh}(\mathbb{G})$, just as classical knot invariants valued in $\mathbb{G}$ are continuous functions $\mathbb{K} \to \mathbb{G}$. The embedding of classical knots into virtual knots is also realized as a geometric morphism.  
\end{abstract}
\maketitle

\section{Introduction}
\subsection{Motivation} Quantum knot invariants can be derived in three distinct ways. Two of these are geometric and the other is diagrammatic. The first geometric method, using Chern-Simons theory, obtains a quantum invariant from a simply connected compact Lie group $G$ as an integral over the space of $G$-connections on the trivial $G$-bundle over $\mathbb{S}^3$. The integral is the vacuum expectation value of the Wilson loop, which is itself an integral along a framed knot \cite{atiyah,baez_muniain,witten}. The second geometric derivation uses the cohomology of the knot space. The knot space $\mathbb{K}$ is the set of smooth embeddings $\mathbb{S}^1 \to \mathbb{R}^3$ topologized so that its path components are the oriented knot types. One first constructs the finite-type invariants from a spectral sequence on the cohomology of $\mathbb{K}$ with coefficients in an abelian group \cite{vassiliev}. For every semisimple Lie algebra, a quantum invariant is then recovered as a weight system on the algebra of chord diagrams \cite{bn_vass}. The diagrammatic derivation instead uses quantum groups and solutions to the Yang-Baxter equation \cite{kassel}.  
\newline
\newline 
Virtual knot theory, instituted by Kauffman \cite{KaV}, provides a general framework for the study of knot diagrammatics. Beyond the usual planar knot diagrams, virtual knot theory expands to include a set of non-planar knot diagrams. All together, these new \emph{viritual knot diagrams} are considered equivalent up to a set of \emph{extended Reidemeister moves} (see Section \ref{sec_lang_diag} below). The value of this expansion is that virtual knot diagrams function as an ``algebraic closure'' of classical knot diagrams. For example, every Wirtinger group of deficiency one is the group of some virtual knot diagram \cite{kim} and there are non-classical virtual knot diagrams having unit Jones polynomial \cite{KaV}. Not every such Wirtinger group is the fundamental group of a knot in $\mathbb{S}^3$ and the existence of a non-trivial knot with unit Jones polynomial is famously unknown. Other instances of this algebraic realization occur in the theory of quandles \cite{CKS_geom,frs} and the Kashiwara-Vergne problem \cite{bd_II}. 
\newline
\newline 
Quantum knot invariants have many well-known applications to physics \cite{witten}, low-dimensional topology \cite{garoufalidis}, and number theory \cite{gelca,kreimer}. Since virtual knot diagrams are analogous to an ``algebraic closure'' of classical knot diagrams, it is reasonable to hope that advances in virtual knot theory will reflect new insights back into these areas. The coupling of knot theory with physics, topology, and number theory, employs a hybrid of geometric and diagrammatic methods. But contrary to the classical case, the diagrammatic model of virtual knot theory has no known fully geometric counterpart. In particular, there is no ``space of virtual knots''. This imbalance makes it difficult to detect any of the expected relations between virtual knot theory and its potential applications.
\newline
\newline 
A popular workaround is to use the minimal genus model of virtual knot theory (see e.g. \cite{adams_2, im_lee_lee, JKO, matveev_roots, sil_will_sat}). Recall that every virtual knot diagram can be represented (non-uniquely) by a knot diagram on a closed orientable surface. Invariants and geometric properties can then be defined by restricting to surfaces of smallest possible genus. Well-definedness follows from Kuperberg's theorem \cite{kuperberg}, which states that every virtual knot type has a minimal genus representative on such a surface $\Sigma$ that is unique up to isotopy and orientation preserving diffeomorphisms $\psi: \Sigma \to \Sigma$. The minimal genus model, however, lacks several desirable features that are present in the classical geometric model. For starters, the equivalence relation on knots in $\mathbb{S}^3$ is generated by isotopy, but the equivalence relation defined by the minimal genus model is not reducible to isotopy alone. Indeed, the mapping class group of any $\Sigma \ne \mathbb{S}^2$ is non-trivial. Secondly, the minimal genus model of virtual knot theory cannot recover the diagrammatic model as not all virtual knot diagrams are realizable on a surface of minimal genus. Lastly, it is generally a hard problem to determine when a given knot in $\Sigma \times \mathbb{R}$ is of minimal genus in the first place (see \cite{cart_sil_will_covering}). The minimal genus model thus becomes impractical for knots with either large minimal genus or large crossing number. 
\newline
\newline
Here we introduce a model of virtual knot theory that, from a logical point of view, is geometric in the same sense that classical knot theory is geometric. The model formalizes the intuitive notion that a virtual knot is an actual knot inhabiting a variable ambient space. The virtual knot types are recovered without using diagrammatics. Furthermore, the equivalence relation on virtual knots is generated solely by isotopy. All virtual knot invariants are also realized geometrically.  

\subsection{Overview of main results} \label{subsec_results} The geometric model is based upon a generalization of the knot space $\mathbb{K}$ to a ``virtual knot space''. Consider first the classical case and observe that the points of $\mathbb{K}$ are knots, each of which can be specified by a continuous function $\mathbbm{1} \to \mathbb{K}$, where $\mathbbm{1}=\{0\}$ is the one-point space.  The path components of $\mathbb{K}$ are the oriented knot types. Hence, the isotopy relation is determined by paths $\mathbb{I}=[0,1] \to \mathbb{K}$. If $\mathbb{G}$ is a discrete topological space, any continuous function $\mathbb{K} \to \mathbb{G}$ is a knot invariant. Geometric knot theory is thus interpreted in the category $\textbf{Top}$ of topological spaces and continuous functions.
\newline
\newline
To interpret ``geometric'' for virtual knot theory, we first reinterpret ``knot space'', ``knot'', ``isotopy'', and ``knot invariant'' in the 2-category $\mathfrak{Top}$ of topoi and geometric morphisms. For a topological space $\mathbb{X}$, denote the category of open sets of $\mathbb{X}$ by $\textbf{X}$ and the category of sheaves on $\mathbb{X}$ by $\text{Sh}(\textbf{X})$. The knot space $\mathbb{K}$ is reinterpreted as the Grothendieck topos $\text{Sh}(\textbf{K})$ of sheaves on $\mathbb{K}$. Knots are replaced by the points of $\text{Sh}(\textbf{K})$, i.e. geometric morphisms $\text{Sh}(\textbf{1})=\textbf{Sets} \to \text{Sh}(\textbf{K})$. Paths $\mathbb{I} \to \mathbb{K}$ are replaced with geometric morphisms $\text{Sh}(\textbf{I}) \to \text{Sh}(\textbf{K})$ and knot invariants are replaced with geometric morphisms $\text{Sh}(\textbf{K}) \to \text{Sh}(\textbf{G})$. The sheaf-theoretic model is equivalent to the geometric one. Hence, ``geometric'' and ``sheaf-theoretic'' may be used interchangeably. 
\newline
\newline
The role of ``virtual knot space'' will be played by a Grothendieck topos $\text{Sh}(\textbf{VK})=\text{Sh}(\textbf{VK},J_{\textbf{VK}})$ of sheaves on a site $(\textbf{VK},J_{\textbf{VK}})$. Virtual knots are the points of $\text{Sh}(\textbf{VK})$. Geometric morphisms $\text{Sh}(\textbf{I}) \to \text{Sh}(\textbf{VK})$ between the points of the virtual knot space generate the equivalence relation of virtual isotopy. The equivalence classes of virtual isotopy are in one-to-one correspondence with those of diagrammatic virtual knot theory. Geometric morphisms $\text{Sh}(\textbf{VK}) \to \text{Sh}(\textbf{G})$ correspond to virtual knot invariants. See Figure \ref{figure_compare} for a summary. Here  $J_{\textbf{K}}$ denotes open cover topology on $\textbf{K}$.

\begin{figure}[htb]
\begin{tabular}{|c|c||c|c|} \cline{3-4}
\multicolumn{2}{c}{} &  \multicolumn{2}{|c|}{Sites $(\textbf{C},J)$} \\ \cline{3-4}
\multicolumn{2}{c|}{} & \multirow[c]{2}*{$(\textbf{K},J_{\textbf{K}})$} & \multirow[c]{2}*{$(\textbf{VK},J_{\textbf{VK}})$} \\ \multicolumn{2}{c|}{} & & \\  \hhline{--=|=} 
\multirow[c]{3}*{\rotatebox{90}{\small Geometric Morphisms \hspace{.05cm}}}& \multirow[c]{2}*{$\textbf{Sets} \to \text{Sh}(\textbf{C},J)$} & \multirow[c]{2}*{Knots} & \multirow[c]{2}*{Virtual Knots} \\
 & & & \\ \cline{2-4}
& \multirow[c]{2}*{$\text{Sh}(\textbf{I}) \to \text{Sh}(\textbf{C},J)$} & \multirow[c]{2}*{Isotopy} & \multirow[c]{2}*{Virtual Isotopy} \\ & & & \\ \cline{2-4} 
& \multirow[c]{4}*{$\text{Sh}(\textbf{C},J) \to \text{Sh}(\textbf{G})$} & \multirow[c]{4}*{\minitab[c]{$\mathbb{G}$-valued \\ Knot Invariants}} & \multirow[c]{4}*{\minitab[c]{$\mathbb{G}$-valued  \\ Virtual Knot Invariants}} \\ & & & \\ & & & \\ & & & \\ \hline
\end{tabular}
\caption{Comparison between sheaf-theoretic knot theory and virtual knot theory.} \label{figure_compare}
\end{figure}

Alternatively, we also describe virtual knot invariants as cohomology classes. For $G$ an abelian group, the singular cohomology group $H^0(\mathbb{K},G)$ is the group of $G$-valued knot invariants. Here it will be shown that for any discrete abelian group $\mathbb{G}$, the group of $\mathbb{G}$-valued virtual knot invariants is $H^0(\text{Sh}(\textbf{VK}),\Delta_{\mathbb{G}})$, the $0$-th sheaf cohomology group with coefficients in the constant sheaf $\Delta_{\mathbb{G}}$. This result should be contrasted with the minimal genus model, where invariants are defined with constant coefficients on a minimal genus representative. In the sheaf-theoretic perspective, virtual knot invariants instead have locally constant coefficients. Local coefficients essentially keep track of how the value of an invariant behaves under stabilization. Hence, these invariants are defined for any representative rather than just those of minimal genus.
\newline
\newline  
Simultaneously lifting classical and virtual knot theory to the $2$-categorical setting allows for the relations between them to be studied more directly. The canonical example is the ``embedding of classical knots into virtual knots''.  It will be shown that this can be realized as a geometric morphism $\text{Sh}(\textbf{K})\to \text{Sh}(\textbf{VK})$. More generally, if $\mathbb{K}(\Sigma)$ is the space of knots in $\Sigma \times \mathbb{R}$ for $\Sigma$ an oriented surface, there is a geometric morphism $\text{Sh}(\textbf{K}(\Sigma)) \to \text{Sh}(\textbf{VK})$. This formalizes the commonly used expression ``the projection of knots in thickened surfaces to virtual knots''.
\newline
\newline
Formalization is another main objective of this paper. Recall that it is common in the literature to define a ``virtual knot'' as an ``equivalence class of virtual knot diagrams". The latter should be more accurately called a \emph{virtual knot type}. This means that ``virtual knots'' themselves are effectively undefined. Similarly, a ``virtual isotopy'' is often conflated with a ``finite sequence of extended Reidemeister moves''. To connect the geometric and diagrammatic models of virtual knot theory, more precision is required. Hence, we define intermediate notions of \emph{variable-space knots} and \emph{variable-space isotopies}. Intuitively, a variable-space knot is an actual knot in a thickened surface having non-constant ambient space. A variable-space isotopy is an actual isotopy of variable space knots. Variable-space knots and isotopies are then shown to be in one-to-one correspondence with geometric morphisms $\textbf{Sets} \to \text{Sh}(\textbf{VK})$ and $\text{Sh}(\textbf{I}) \to \text{Sh}(\textbf{VK})$, respectively.


\subsection{Organization} The geometric, diagrammatic, and sheaf-theoretic, models of knot theory are discussed in Section \ref{sec_languages}. The virtual knot space and the geometric model of virtual knot theory are introduced in Section \ref{sec_virtual_knot_space}. Variable-space knots, variable-space isotopies, and variable-space knot invariants are formalized in Sections \ref{sec_var_ambient}, \ref{sec_virtual_isotopy_var}, and \ref{sec_cohom}, respectively. These are then shown to be in one-to-one correspondence with our sheaf-theoretically defined virtual knots, virtual isotopies, and virtual knot invariants (Sections \ref{sec_points}, \ref{sec_virtual_isotopy_paths}, \ref{sec_invar_geom}, respectively). For the sheaf cohomology of the virtual knot space, see Section \ref{sec_cohom}. Geometric morphisms between sheaf-theoretic knot spaces are studied in Section \ref{sec_other}. Additional discussion of formalization in virtual knot theory appears in Section \ref{sec_further}, along with some questions and directions for further research. 

\section{The languages of knot theory} \label{sec_languages} 
 
\subsection{Geometric $\&$ sheaf-theoretic models} \label{sec_sheaf_theory_K} A \emph{knot} $K:\mathbb{S}^1 \to \mathbb{R}^3$ is a smooth embedding of the circle. The set $\mathbb{K}$ of knots is topologized by taking the Taylor series coefficients at $K(z)$ for each $z \in \mathbb{S}^1$ and $K \in 
\mathbb{K}$, and requiring the power series to vary continuously in $z$. More precisely, there is an injective map $j^{\infty}:C^{\infty}(\mathbb{S}^1, \mathbb{R}^3) \to C^0(\mathbb{S}^1,J^{\infty}(\mathbb{S}^1,\mathbb{R}^3))$, where $J^{\infty}(\mathbb{S}^1,\mathbb{R}^3)$ is the space of $\infty$-jets. For $A \in  C^{\infty}(\mathbb{S}^1, \mathbb{R}^3)$ and $z \in \mathbb{S}^1$,  $j^{\infty}(A)(z)$ is the $\infty$-jet of $A$ at $z$. Give $C^0(\mathbb{S}^1,J^{\infty}(\mathbb{S}^1,\mathbb{R}^3))$ the compact-open topology and $C^{\infty}(\mathbb{S}^1, \mathbb{R}^3)$, via the injection $j^{\infty}$, the subspace topology. Then $C^{\infty}(\mathbb{S}^1, \mathbb{R}^3)$ is an infinite-dimensional smooth manifold (\cite{kriegl_michor}, Theorem IX.42.1). In particular, it is Hausdorff. The subspace $\mathbb{K}\subseteq C^{\infty}(\mathbb{S}^1, \mathbb{R}^3)$ of knots is open and dense in $C^{\infty}(\mathbb{S}^1, \mathbb{R}^3)$ (see \cite{hirsch}, Theorems 2.1.4 and 2.2.13). An ambient isotopy of knots $K,\widebar{K}:\mathbb{S}^1 \to \mathbb{R}^3$ corresponds to a path $\sigma: \mathbb{I} \to \mathbb{K}$ between the corresponding points $K,\widebar{K}\in \mathbb{K}$. Conversely, if $K,\widebar{K} \in \mathbb{K}$ are in the same path component of $\mathbb{K}$, then $K$ and $\widebar{K}$ are ambient isotopic. Knot invariants are continuous functions $\nu: \mathbb{K} \to \mathbb{G}$, with $\mathbb{G}$ discrete. If $K \in \mathbb{K}$, denote also by $K:\mathbbm{1}=\{0\} \to \mathbb{K}$ the map sending $0$ to the knot $K$. The maps $K$, $\sigma$, $\nu$ will form the  \emph{standard geometric model} of knot theory defined below. 
\newline
\newline
Now, let $\mathbb{X},\mathbb{Y}$ be any topological spaces and $f:\mathbb{X} \to \mathbb{Y}$ a continuous function between them. Recall that $f$ induces the direct image functor $f_*:\text{Sh}(\textbf{X}) \to \text{Sh}(\textbf{Y})$ defined by $f_*(F)(U)=F(f^{-1}(U))$, where $F$ is a sheaf on $\mathbb{X}$ and $U \subseteq \mathbb{Y}$ is open. The left adjoint to $f_*$ is the inverse image functor $f^*:\text{Sh}(\textbf{Y}) \to \text{Sh}(\textbf{X})$. To define $f^*$, each sheaf $P$ on $\mathbb{Y}$ is first identified with an \'{e}tale space $p:\mathbb{E} \to \mathbb{Y}$. Pulling back over $f$ in the category $\textbf{Top}$ gives an \'{e}tale space $q:f^*(\mathbb{E}) \to \mathbb{X}$. This may in turn be identified with a sheaf $Q$ over $\mathbb{X}$. Define $f^*(P)=Q$. The direct and inverse image functors give an adjoint situation $f^*:\text{Sh}(\textbf{Y}) \rightleftarrows \text{Sh}(\textbf{X}):f_*$, where $f^*$ is left exact (i.e. preserves finite limits). This is called a \emph{geometric morphism from} $\text{Sh}(\textbf{X})$ \emph{to} $\text{Sh}(\textbf{Y})$, written $f:\text{Sh}(\textbf{X}) \to \text{Sh}(\textbf{Y})$.
\newline
\newline
Denote by $\mathfrak{Top}$ the $2$-category whose objects are topoi. An arrow $f:\mathscr{E} \to \mathscr{F}$ in $\mathfrak{Top}$ between topoi $\mathscr{E},\mathscr{F}$ is a geometric morphism, which again is any adjoint situation $f^*:\mathscr{F} \rightleftarrows \mathscr{E}:f_*$ with $f^*$ left exact. A $2$-cell $f \implies g$ in $\mathfrak{Top}$ between geometric morphisms $f,g$ is a natural transformation $f^* \to g^*$. The category of geometric morphisms from $\mathscr{E}$ to $\mathscr{F}$ is denoted $\underline{\text{Hom}}(\mathscr{E},\mathscr{F})$. The category of sheaves on a topological space is a Grothendieck topos, so the construction in the previous paragraph gives a functor $\mathscr{T}:\textbf{Top} \to \mathfrak{Top}$ defined by $\mathscr{T}(\mathbb{X})=\text{Sh}(\textbf{X})$ and $\mathscr{T}(f: \mathbb{X} \to \mathbb{Y})=(f:\text{Sh}(\textbf{X}) \to \text{Sh}(\textbf{Y}))$. 
\newline
\newline
Applying $\mathscr{T}(f: \mathbb{X} \to \mathbb{Y})=(f:\text{Sh}(\textbf{X}) \to \text{Sh}(\textbf{Y}))$ to knots $K: \mathbbm{1} \to \mathbb{K}$, isotopies $\sigma:\mathbb{I} \to \mathbb{K}$, and invariants $\nu:\mathbb{K} \to \mathbb{G}$ gives an interpretation of ``knot'', ``isotopy'', and ``knot invariant'' in $\mathfrak{Top}$. Conversely, knots, isotopies, and invariants can be uniquely recovered from their corresponding type of geometric morphism in $\mathfrak{Top}$. Indeed, the functor $\mathscr{T}:\textbf{Top} \to \mathfrak{Top}$, restricts to a 2-category equivalence between the full subcategory of sober spaces $\mathfrak{sob}$ and spatial topoi (see \cite{johnstone}, Corollary 7.27).  In $\mathfrak{sob}$,  there is a $2$-cell $f \implies g$ for $f,g: \mathbb{X} \to \mathbb{Y}$ whenever $f(x) \in \text{cl}(\{g(x)\})$ for all $x \in \mathbb{X}$. This defines a partial ordering on $\text{Hom}_{\mathfrak{sob}}(\mathbb{X},\mathbb{Y})$, where $f \le g$ if and only if there is a $2$-cell $f \implies g$. If $\mathbb{Y}$ is sober, and $f: \text{Sh}(\textbf{X})\to \text{Sh}(\textbf{Y})$ is a geometric morphism, the equivalence takes $f$ to a continuous function $f^{\mathbb{X}}_{\mathbb{Y}}:\mathbb{X} \to \mathbb{Y}$ such that $\mathscr{T}(f^{\mathbb{X}}_{\mathbb{Y}})\cong f$. Then we have the following result:

\begin{proposition} \label{prop_convert_to_2_cat} Let $f,\widebar{f}\in\underline{\text{Hom}}(\text{Sh}(\textbf{X}),\text{Sh}(\textbf{Y}))$ where $(\mathbb{X},\mathbb{Y})=(\mathbbm{1},\mathbb{K})$, $(\mathbb{I},\mathbb{K})$, or $(\mathbb{K},\mathbb{G})$. Let maps $f^{\mathbb{X}}_{\mathbb{Y}},\widebar{f}^{\,\mathbb{X}}_{\mathbb{Y}}:\mathbb{X} \to \mathbb{Y}$ correspond through the equivalence $\mathscr{T}$ to $f,\widebar{f}$. Then $f \cong \widebar{f}$ iff $f^{\mathbb{X}}_{\mathbb{Y}}=\widebar{f}^{\,\mathbb{X}}_{\mathbb{Y}}$.
\end{proposition}
\begin{proof}  Since $\mathbb{Y}=\mathbb{K}$ or $\mathbb{G}$, $\mathbb{Y}$ is Hausdorff and hence one point sets are closed. If $g,h:\mathbb{X} \to \mathbb{Y}$ are continuous, and $g \implies h$ is a $2$-cell, it follows that $g(x)=h(x)$ for all $x \in \mathbb{X}$. Hence, if $f,\widebar{f} \in \underline{\text{Hom}}(\text{Sh}(\textbf{X}),\text{Sh}(\textbf{Y}))$ are naturally isomorphic, their corresponding continuous functions $f^{\mathbb{X}}_{\mathbb{Y}},\widebar{f}^{\,\mathbb{X}}_{\mathbb{Y}}:\mathbb{X} \to \mathbb{Y}$ are equal. Conversely, if $f^{\mathbb{X}}_{\mathbb{Y}}=\wbar{f}^{\mathbb{X}}_{\mathbb{Y}}$, then $f \cong \mathscr{T}(f^{\mathbb{X}}_{\mathbb{Y}})=\mathscr{T}(\wbar{f}^{\mathbb{X}}_{\mathbb{Y}}) \cong \wbar{f}$. 
\end{proof}

Therefore, knot theory may be done in $\mathfrak{Top}$ in the same way it is done in $\textbf{Top}$, as long as naturally isomorphic geometric morphisms are identified. Let $\tau_0:\{0\} \to \mathbb{I}$, $\tau_1:\{1\} \to \mathbb{I}$ be the inclusion maps and set $0=\mathscr{T}(\tau_0)$, $1=\mathscr{T}(\tau_1)$. Two geometric morphisms $K,\widebar{K}:\textbf{Sets} \to \text{Sh}(\textbf{K})$ will be called \emph{isotopic} if there is a geometric morphism $\sigma:\text{Sh}(\textbf{I}) \to \text{Sh}(\textbf{K})$ such that $\sigma \circ 0 \cong K$ and $\sigma \circ 1 \cong \widebar{K}$. Proposition \ref{prop_convert_to_2_cat} implies that $K,\widebar{K}:\textbf{Sets} \to \text{Sh}(\textbf{K})$ are isotopic if and only if $K^{\mathbbm{1}}_{\mathbb{K}},\widebar{K}^{\,\mathbbm{1}}_{\mathbb{K}}:\mathbbm{1} \to \mathbb{K}$ are isotopic. For a geometric morphism $\nu:\text{Sh}(\textbf{K}) \to \text{Sh}(\textbf{G})$, $\nu \circ K\cong \nu \circ \widebar{K}$ if and only if $\nu^{\mathbb{K}}_{\mathbb{G}} \circ K^{\mathbbm{1}}_{\mathbb{K}}=\nu^{\mathbb{K}}_{\mathbb{G}} \circ \widebar{K}^{\,\mathbbm{1}}_{\mathbb{K}}$. Hence, if $K,\widebar{K}:\textbf{Sets} \to \text{Sh}(\textbf{K})$ are isotopic, $\nu \circ K \cong \nu \circ \widebar{K}$. If desired, the common value in $\mathbb{G}$ can be uniquely recovered as the point $\nu^{\mathbb{K}}_{\mathbb{G}} \circ K^{\mathbbm{1}}_{\mathbb{K}}=\nu^{\mathbb{K}}_{\mathbb{G}} \circ \widebar{K}^{\,\mathbbm{1}}_{\mathbb{K}}$. 
\newline
\newline
Hence we have two models of a first-order language $\mathfrak{L}_g$ (see \cite{marker}, Section 1.1 or \cite{mac_moer}, Section X.1) of \emph{geometric knot theory}. The language itself consists of a sort $\mathfrak{K}$ of \emph{knots}, a relation $\mathfrak{I} \subseteq \mathfrak{K} \times \mathfrak{K}$ called \emph{isotopy}, and a set of function symbols $\mathfrak{f}:\mathfrak{K} \to \mathfrak{S}$ called \emph{invariants}, where $\mathfrak{S}$ is a sort of \emph{values}. The standard geometric model $SG$, in $\textbf{Sets}$, assigns $\mathfrak{K}$ to $\mathfrak{K}^{(SG)}=\mathbb{K}$. The model assigns $\mathfrak{I}$ to the equivalence relation $\mathfrak{I}^{(SG)} \subseteq \mathfrak{K}^{(SG)} \times \mathfrak{K}^{(SG)}$ such that $(K,\widebar{K}) \in \mathfrak{I}^{(SG)}$ if and only if $K,\widebar{K}$ are in the same path component of $\mathbb{K}$. Each function symbol $\mathfrak{f}$ is assigned to a continuous function $\mathfrak{f}^{(SG)}:\mathfrak{K}^{(SG)} \to \mathfrak{S}^{(SG)}$, where $\mathfrak{S}^{(SG)}$ is some set with the discrete topology. The sheaf-theoretic model $ST$ is obtained from $SG$ by identifying each $K \in \mathbb{K}$ with a continuous function $K:\mathbbm{1} \to \mathbb{K}$ and applying the functor $\mathscr{T}$. This gives the set $\mathfrak{K}^{(ST)}$. The isotopy relation $\mathfrak{I}^{(ST)}$ and function symbols $\mathfrak{f}^{(ST)}$ are interpreted according to the preceding paragraph. These observations imply:

\begin{theorem} The sheaf-theoretic model $ST$ and standard geometric model $SG$ are isomorphic models of the language $\mathfrak{L}_g$ of geometric knot theory.
\end{theorem}

Lastly, it is worthwhile to remark that the topological structure of $\mathbb{K}$ can also be recovered from $\text{Sh}(\textbf{K})$, again due to the $2$-categorical equivalence between $\mathfrak{sob}$ and spatial topoi. Furthermore, the category of subobjects of the terminal object $1$ of $\text{Sh}(\textbf{K})$ is isomorphic as a Heyting algebra to the lattice of open sets of $\mathbb{K}$ (see \cite{mac_moer}, Proposition II.3.4). 

\begin{figure}[htb]
\begin{tabular}{|cccccc|} \hline 
\multirow[l]{5}*{\rotatebox{90}{$\leftarrow$ Extended Reidemeister Moves $\rightarrow$ \hspace{.15cm}}} & \multicolumn{1}{|c}{\multirow[l]{2}*{\rotatebox{90}{ \small Reidemeister Moves \hspace{.025cm}}}} & & & & \\
\multicolumn{1}{|c|}{} & & \multicolumn{4}{c|}{\begin{tabular}{cccc} \begin{tabular}{c} \def\svgwidth{.6in}
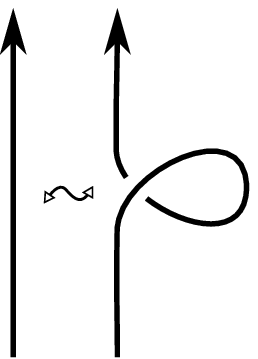 \\ $\Omega 1 a$ \end{tabular} & \begin{tabular}{c} \def\svgwidth{.6in}
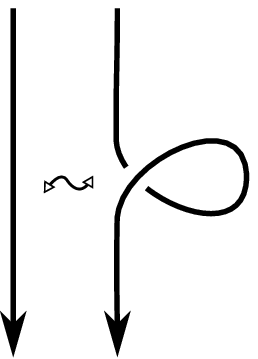 \\ $\Omega 1 b$\end{tabular} & \begin{tabular}{c} \def\svgwidth{.8in}
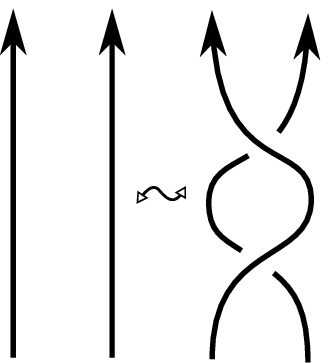 \\ $\Omega 2 a$ \end{tabular} & \begin{tabular}{c} \def\svgwidth{1.6in}
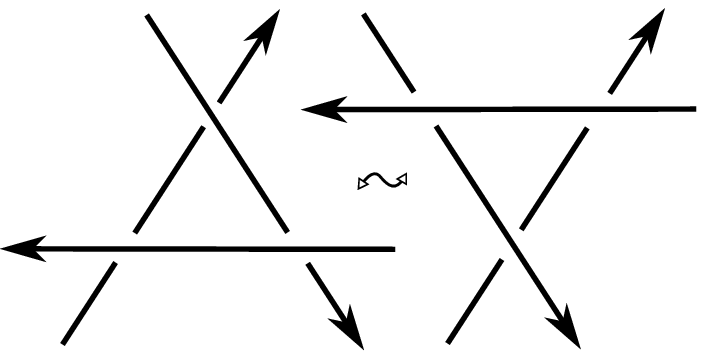 \\ $\Omega 3 a$ \end{tabular} \end{tabular}} \\ \hhline{~-----}
 & & & & & \\
& & \begin{tabular}{c} \def\svgwidth{.6in}
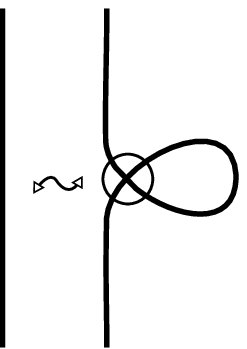 \\ $v\Omega 1$ \end{tabular} & \begin{tabular}{c} \def\svgwidth{.8in}
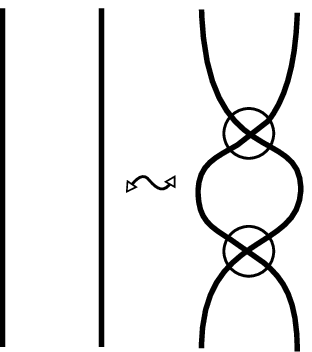 \\ $v\Omega 2$\end{tabular} & \begin{tabular}{c} \def\svgwidth{1.6in}
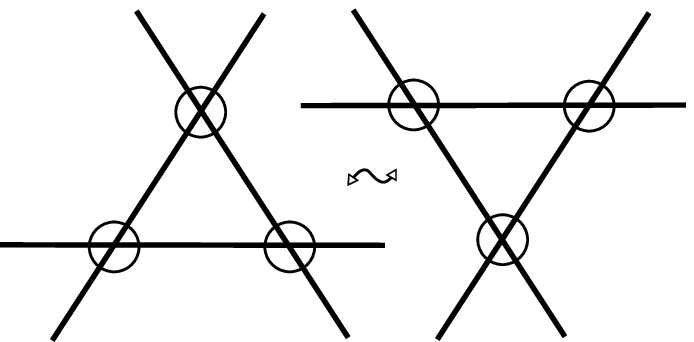 \\ $v\Omega 3$ \end{tabular} & \begin{tabular}{c} \def\svgwidth{1.6in}
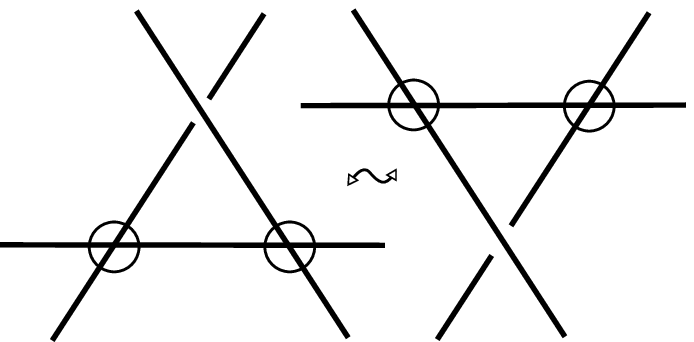 \\ $v\Omega 4$ \end{tabular} \\ \hline
\end{tabular}
\caption{The Reidemeister and extended Reidemeister moves.} \label{fig_moves}
\end{figure}

\subsection{Diagrammatic models} \label{sec_lang_diag} This section reviews several diagrammatic models of knot theory and introduces a modification of the Carter-Kamada-Saito model that will be needed ahead. The language $\mathfrak{L}_d$ of diagrammatic knot theory has a sort $\mathfrak{D}$ of \emph{knot diagrams}, a relation $\mathfrak{R} \subseteq \mathfrak{D} \times \mathfrak{D}$ called \emph{Reidemeister equivalence}, and a set of function symbols $\mathfrak{f}:\mathfrak{D} \to \mathfrak{S}$, called \emph{invariants} where $\mathfrak{S}$ is again a sort of \emph{values}. A \emph{diagrammatic knot theory} has an axiom of the form $(\forall x_1,x_2)(\mathfrak{R}(x_1,x_2) \implies \mathfrak{f}(x_1)=\mathfrak{f}(x_2))$ for all function symbols $\mathfrak{f}$. The standard diagrammatic model $SD$ assigns $\mathfrak{D}^{(SD)}$ to the set of oriented knot diagrams on $\mathbb{S}^2$. The relation $\mathfrak{R}^{(SD)} \subseteq \mathfrak{D}^{(SD)} \times \mathfrak{D}^{(SD)}$ is the smallest equivalence relation containing all the Reidemeister moves and planar isotopies (or equivalently, orientation preserving diffeomorphisms of $\mathbb{S}^2$). Polyak's \cite{polyak_minimal} minimal generating set of Reidemeister moves is depicted at the top of Figure \ref{fig_moves}. An equivalence class of the relation $\mathfrak{R}^{(SD)}$ is called an \emph{(oriented) knot type}. A function symbol $\mathfrak{f}$ is assigned to a function $\mathfrak{f}^{(SD)}:\mathfrak{D}^{(SD)} \to \mathfrak{S}^{(SD)}$ satisfying the axiom for $\mathfrak{f}$. Hence, it is an invariant of knot types.
\newline 


Kauffman's diagrammatic knot theory for virtual knot diagrams will here be called the \emph{Kauffman virtual diagram model ($KVD$)}. It assigns $\mathfrak{D}^{(KVD)}$ to the set of planar knot diagrams, where in addition to the two usual kinds of crossings, \emph{virtual crossings} are also allowed (see Figure \ref{fig_cross}). The relation $\mathfrak{R}^{(KVD)}\subseteq \mathfrak{D}^{(KVD)} \times \mathfrak{D}^{(KVD)}$ is the smallest equivalence relation generated by the extended Reidemeister moves (see Figure \ref{fig_moves}) and planar isotopies. For moves $v\Omega 1-v\Omega 4$, the arcs are oriented arbitrarily. Equivalence classes of $\mathfrak{R}^{(KVD)}$ are the \emph{(oriented) virtual knot types}. Each $\mathfrak{f}:\mathfrak{D} \to \mathfrak{S}$ is assigned a function $f^{(KVD)}:\mathfrak{D}^{(KVD)} \to \mathfrak{S}^{(KVD)}$ that is constant on each virtual knot type. 
\newline

\begin{figure}[htb]
\begin{tabular}{|c|cccc||c|c|} \hline
\multirow[l]{2}*{\rotatebox{90}{classical \hspace{.65cm}}} & & & & & & \multirow[c]{2}*{\rotatebox{-90}{\hspace{.35cm} virtual  }} \\
& \begin{tabular}{c} 
\def\svgwidth{.75in}
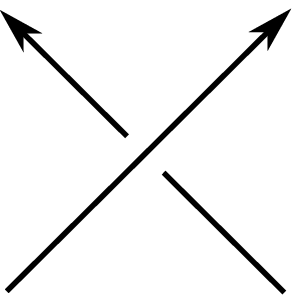 \end{tabular} & & \begin{tabular}{c} \def\svgwidth{.75in}
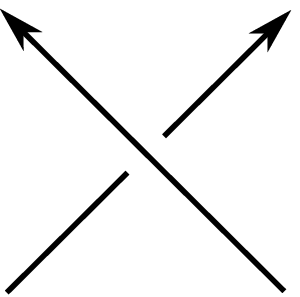 \end{tabular} & & \begin{tabular}{c} \def\svgwidth{.75in}
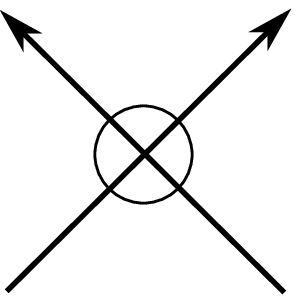 \end{tabular} & \\
& $\oplus$ & & $\ominus$ & & & \\\hline
\end{tabular}
\caption{Crossing types.} \label{fig_cross}
\end{figure}
 
Every virtual knot diagram $X$ can be represented by a diagram $D$ on some compact, connected, and oriented surface $\Xi$. Briefly, a $0$-handle (i.e. a disc $\mathbb{D}^2$) is placed at every classical crossing of $X$ and an untwisted $1$-handle (i.e. a band $\mathbb{D}^1 \times \mathbb{D}^1$) is placed along every arc of $D$.  At a virtual crossing, the bands along each of the arcs pass over and under one another. The union of $0$- and $1$-handles is a compact surface $\Xi$ supporting a knot diagram $D$. See Figure \ref{fig_virt_to_abs}. Such a pair $(\Xi,D)$, where $D$ is a deformation retract of $\Xi$, is called an \emph{abstract knot diagram} (Kamada-Kamada \cite{kamkam}).
\newline
\newline
This can be made into another model of diagrammatic knot theory by enlarging to the set of pairs $(\Xi,D)$ where $\Xi$ is any compact, connected, and oriented surface and $D$ is a knot diagram on $\Xi$. The \emph{Carter-Kamada-Saito model} $(CKS)$ assigns $\mathfrak{D}^{(CKS)}$ to be the set of such pairs. For pairs $(\Xi_1,D_1)$, $(\Xi_2,D_2) \in \mathfrak{D}^{(CKS)}$, write $(\Xi_1,D_1)\stackrel{s}{\squigarrowleftright} (\Xi_2,D_2)$ if there is a compact, connected, and oriented surface $\Xi_{1.5}$, and orientation preserving embeddings $\upsilon_1:\Xi_1 \to \Xi_{1.5},\delta_2:\Xi_2 \to \Xi_{1.5}$ such that $\upsilon_1(D_1)$ and $\delta_2(D_2)$ are related by a Reidemeister move or are equal. This will be called a \emph{stable Reidemeister move}. Assign $\mathfrak{R}^{(CKS)} \subseteq \mathfrak{D}^{(CKS)} \times \mathfrak{D}^{(CKS)}$ to be the smallest equivalence relation containing all stable Reidemeister moves. Function symbols are assigned to functions $\mathfrak{f}^{(CKS)}:\mathfrak{D}^{(CKS)} \to \mathfrak{S}^{(CKS)}$ that are invariant under stable Reidemeister moves. Carter, Kamada, and Saito \cite{CKS} proved that there is a bijection between the equivalence classes of $\mathfrak{R}^{(KVD)}$ and $\mathfrak{R}^{(CKS)}$ .
\newline

\begin{figure}[htb]
\[
\xymatrix{
\begin{array}{c} \def\svgwidth{2in}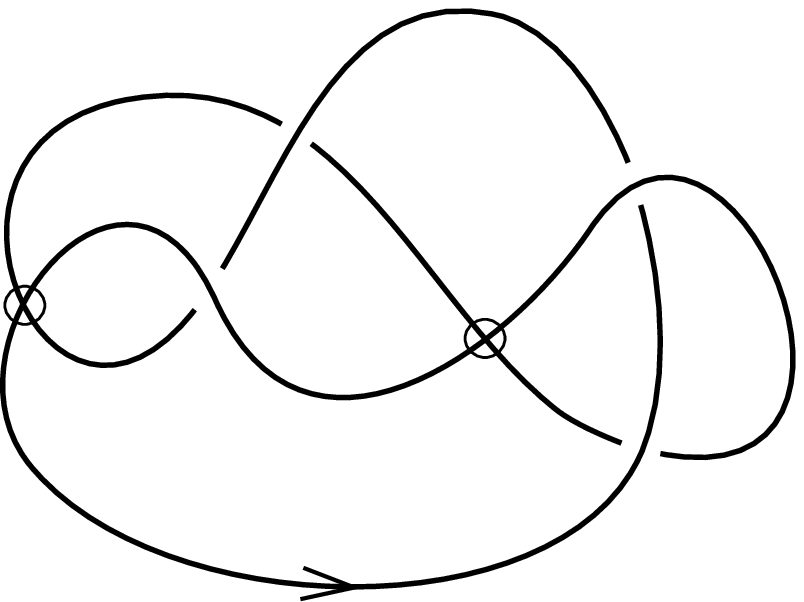  \end{array} \ar[r] & \begin{array}{c} \def\svgwidth{2in}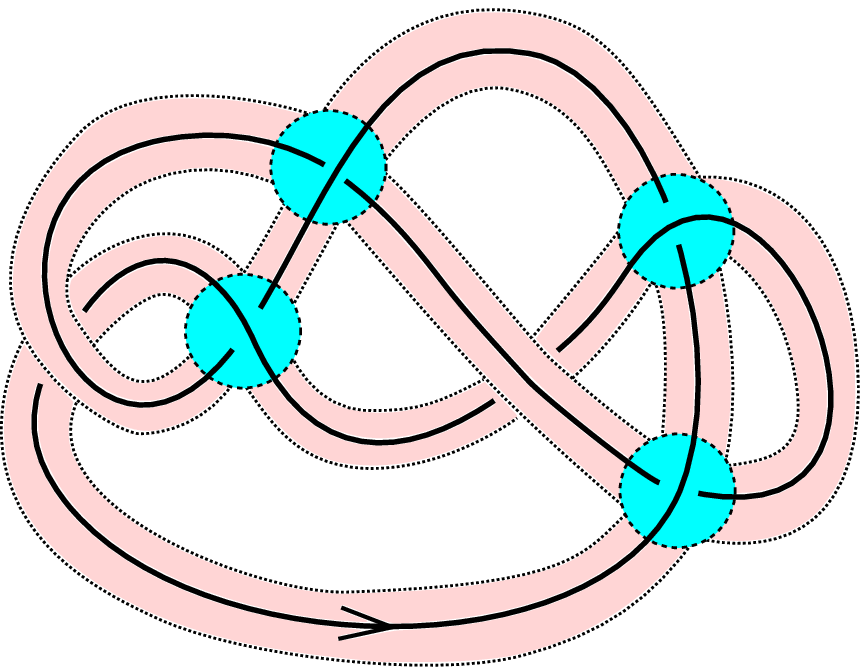 \end{array}}
\]
\caption{A virtual knot diagram $X$ (left) and its abstract knot diagram $(\Xi,D)$ (right).} \label{fig_virt_to_abs}
\end{figure} 

The diagrammatic model recovered from the virtual knot space ahead is close to the $CKS$ model. To ensure that the category $\textbf{VK}$ is well-behaved, a slightly modified collection of supporting surfaces is needed. The \emph{modified CKS model} ($CKS^*$) assigns $\mathfrak{D}^{(CKS^*)}$ to the set of pairs $(\Sigma,D)$ where $\Sigma$ is a smooth oriented surface with $\partial \Sigma=\varnothing$ and $D$ is a knot diagram on $\Sigma$. By surface, we mean a $2$-manifold that is second-countable and Hausdorff. Surfaces are not assumed to be connected. Note that a smooth orientation preserving embedding $\upsilon: \Sigma_1 \to \Sigma_2$ of such surfaces with $\partial \Sigma_1=\partial \Sigma_2=\varnothing$ is open. A stable Reidemeister move $(\Sigma_1,D_1) \stackrel{s}{\squigarrowleftright} (\Sigma_2,D_2)$ is defined in the same way as before, except that all maps are between oriented surfaces without boundary. As before, define $\mathfrak{R}^{(CKS^*)} \subseteq \mathfrak{D}^{(CKS^*)} \times \mathfrak{D}^{(CKS^*)}$ to be the smallest equivalence relation containing all stable Reidemeister moves.

\begin{proposition} \label{prop_no_boundary_OK} There is a bijection between the equivalence classes of $\mathfrak{R}^{(CKS)}$ and $\mathfrak{R}^{(CKS^*)}$.
\end{proposition}

\begin{proof} Suppose that $((\Xi,D),(\widebar{\Xi},\widebar{D})) \in \mathfrak{R}^{(CKS^*)}$. Then there is a finite sequence of stable Reidemeister moves:
\[
(\Xi,D)=(\Xi_1,D_1) \stackrel{s}{\squigarrowleftright} (\Xi_2,D_2)\stackrel{s}{\squigarrowleftright}\cdots \stackrel{s}{\squigarrowleftright} (\Xi_n,D_n)=(\widebar{\Xi},\widebar{D}).
\]
Each intermediate surface $\Xi_i$ can be replaced with $\Sigma_i=\Xi_i\smallsetminus \partial \Xi_i$. This has no effect on Reidemeister moves, so it follows that $((\Xi\smallsetminus \partial \Xi,D),(\widebar{\Xi}\smallsetminus \partial \widebar{\Xi},\widebar{D})) \in \mathfrak{R}^{(CKS^*)}$. Conversely, suppose that $((\Sigma,D),(\widebar{\Sigma},\widebar{D}))\in \mathfrak{R}^{(CKS^*)}$. Then there is a finite sequence of stable Reidemeister moves: 
\[
(\Sigma,D) = (\Sigma_1,D_1) \stackrel{s}{\squigarrowleftright} (\Sigma_2,D_2)\stackrel{s}{\squigarrowleftright} \cdots \stackrel{s}{\squigarrowleftright} (\Sigma_n,D_n)=(\widebar{\Sigma},\widebar{D}).
\]
Every knot diagram on $\Sigma_i$ is contained in a compact connected subsurface $\Xi_i$ on $\Sigma_i$. Every  Reidemeister move also occur in a compact neighborhood of the knot diagram. It follows that $((\Xi_i,D_i),(\widebar{\Xi}_{i+1},\widebar{D}_{i+1})) \in \mathfrak{R}^{(CKS)}$ for $1 \le i \le n-1$ and hence $((\Xi_1,D_1),(\widebar{\Xi}_n,\widebar{D}_n)) \in \mathfrak{R}^{(CKS)}$. This proves the result, since replacing $\Sigma_i$ with $\Xi_i \smallsetminus \partial \Xi_i$ also preserves the $\mathfrak{R}^{(CKS^*)}$-equivalence class.
\end{proof}

\section{The virtual knot space} \label{sec_virtual_knot_space}

\subsection{Preliminaries} \label{sec_neighborhoods} In analogy with the classical case, the objects of the category $\textbf{VK}$ should be open sets. These will be gathered from among all spaces $\mathbb{K}(\Sigma)$ of smooth embeddings $K:\mathbb{S}^1 \to \Sigma \times \mathbb{R}$, where $\Sigma$ an oriented surface with $\partial \Sigma=\varnothing$ (see Section \ref{sec_lang_diag}). Before defining the virtual knot space, we gather some facts about $\mathbb{K}(\Sigma)$. As for $\mathbb{K}$, first define a topology on $C^{\infty}(\mathbb{S}^1, \Sigma \times \mathbb{R})$ using the jet map $j^{\infty}:C^{\infty}(\mathbb{S}^1, \Sigma \times \mathbb{R}) \to C^0(\mathbb{S}^1,J^{\infty}(\mathbb{S}^1,\Sigma \times \mathbb{R}))$, where $C^0(\mathbb{S}^1,J^{\infty}(\mathbb{S}^1,\Sigma \times \mathbb{R}))$ has the compact-open topology. Again, $C^{\infty}(S^1, \Sigma \times \mathbb{R})$ is an infinite-dimensional smooth manifold and $\mathbb{K}(\Sigma)$ is open and dense in $C^{\infty}(\mathbb{S}^1, \Sigma \times \mathbb{R})$. The following result is included for completeness.

\begin{proposition} \label{prop_iso_iff_path} If knots $K, \widebar{K} \in \mathbb{K}(\Sigma)$ are in the same path component, then $K,\widebar{K}$ are ambient isotopic in $\Sigma \times \mathbb{R}$.
\end{proposition}
\begin{proof} Since $\mathbb{K}(\Sigma)$ is a smooth manifold, a continuous path $\sigma:\mathbb{I} \to \mathbb{K}(\Sigma)$ between $K$ and $\widebar{K}$ can be approximated with a smooth one. A smooth path in $\mathbb{K}(\Sigma)$ yields a smooth isotopy $\mathbb{S}^1 \times \mathbb{I} \to \Sigma \times \mathbb{R}$. By the isotopy extension theorem  (\cite{hirsch}, Theorem 8.1.3), $K$ and $\widebar{K}$ are ambient isotopic.
\end{proof}


Now suppose that $\Sigma_1,\Sigma_2$ are oriented surfaces without boundary and $\psi: \Sigma_1 \to \Sigma_2$ is an orientation preserving embedding. Since $\partial \Sigma_1=\varnothing$, $\psi$ is open. Let $\text{id}:\mathbb{R} \to \mathbb{R}$ be the identity map. The map $\psi \times \text{id}: \Sigma_1 \times \mathbb{R} \to \Sigma_2 \times \mathbb{R}$ defines a map $\Psi:\mathbb{K}(\Sigma_1) \to \mathbb{K}(\Sigma_2)$ by $\Psi(K)=(\psi \times \text{id}) \circ K$. It is a straightforward exercise to show that, with the given topologies on $\mathbb{K}(\Sigma_1)$ and $\mathbb{K}(\Sigma_2)$, the map $\Psi$ is continuous, open, and injective. Note the significance of the hypothesis that our surfaces $\Sigma_1,\Sigma_2$ have empty boundary: it is needed so that all such maps $\psi$ and $\Psi$ are open.

\begin{notation}[Capitalized $\&$ decapitalized maps] \label{notate_cap} If $\psi:\Sigma_1 \to \Sigma_2$ is an embedding of surfaces, the capitalized letter $\Psi$ will be used to denote the open map $\Psi:\mathbb{K}(\Sigma_1) \to \mathbb{K}(\Sigma_2)$, where $\Psi(K)=(\psi \times \text{id}) \circ K$.  Conversely, a capital letter $\Psi$ denoting a map $\Psi:\mathbb{K}(\Sigma_1) \to \mathbb{K}(\Sigma_2)$ will be assumed to be of this form. The decapitalized letter $\psi$ will denote the corresponding embedding $\psi: \Sigma_1 \to \Sigma_2$.
\end{notation}

Capitalized maps $\Psi$ will be used in the next subsection to formally define the arrows of $\textbf{VK}$. Since $\Psi$ is open, continuous, and injective, it can be viewed intuitively as a \emph{generalized inclusion map}, analogous with the arrows of category of open sets $\textbf{K}$ of $\mathbb{K}$. If $U_1 \subseteq \mathbb{K}(\Sigma_1)$ and $K \in \Psi(U_1) \subseteq \mathbb{K}(\Sigma_2)$, then $U_1$ can then be viewed as a \emph{generalized neighborhood} of $K$.  
\newline
\newline
We conclude this subsection with a technical lemma. It will be used ahead (see Theorem \ref{thm_recovers}) to prove that the geometric model of virtual knot theory recovers its diagrammatic model $CKS^*$. 

\begin{lemma} \label{prop_no_more_diagrams} Suppose that $(\Sigma,D),(\widebar{\Sigma},\widebar{D}) \in \mathfrak{D}^{(CKS^*)}$ and let $K \in \mathbb{K}(\Sigma)$, $\widebar{K} \in \mathbb{K}(\widebar{\Sigma})$ be knots having diagrams $D,\widebar{D}$, respectively. Then $(\Sigma,D),(\widebar{\Sigma},\widebar{D})$ have the same virtual knot type if and only if there is a diagram in $\textbf{Top}$ of surfaces and capitalized maps of the following form:
\[
\xymatrix{
\mathbb{K}(\Sigma)=\mathbb{K}(\Sigma_1) \ar[d]_-{\Upsilon_1} & \mathbb{K}(\Sigma_2) \ar[dl]_-{\Delta_2} \ar[d]_-{\Upsilon_2} & \ar[dl]_-{\Delta_3} \cdots & \mathbb{K}(\Sigma_{n-1}) \ar[dl]_-{\Delta_{n-1}} \ar[d]^-{\Upsilon_{n-1}} & \mathbb{K}(\Sigma_n)=\mathbb{K}(\widebar{\Sigma}) \ar[dl]^-{\Delta_{n}} \\
\mathbb{K}(\Sigma_{1.5}) & \mathbb{K}(\Sigma_{2.5}) & \cdots & \mathbb{K}(\Sigma_{n-.5}) & 
},
\]
and path components $C_1 \subseteq \mathbb{K}(\Sigma_1)$, $C_{1.5} \subseteq \mathbb{K}(\Sigma_{1.5}),\ldots,C_n \subseteq \mathbb{K}(\Sigma_n)$ such that $K \in C_1$, $\widebar{K} \in C_n$, and for $1 \le i \le n-1$, $\Upsilon_i(C_i) \subseteq C_{i+.5}$, $\Delta_{i+1}(C_{i+1}) \subseteq C_{i+.5}$. 
\end{lemma}

\begin{proof} For the forward direction, suppose there is a sequence $(\Sigma,D)=(\Sigma_1,D_1)  \stackrel{s}{\squigarrowleftright} (\Sigma_2,D_2)  \stackrel{s}{\squigarrowleftright} \cdots  \stackrel{s}{\squigarrowleftright} (\Sigma_n,D_n)=(\widebar{\Sigma},\widebar{D})$ of stable Reidemeister moves. For each integer $i$, $1 \le i \le n-1$, there is a surface $\Sigma_{i+.5}$ and embeddings $\upsilon_i:\Sigma_i \to \Sigma_{i+.5}$, $\delta_{i+1}:\Sigma_{i+1} \to \Sigma_{i+.5}$ such that $\upsilon_i(D_i)$ and $\delta_i(D_{i+1})$ are equal or Reidemeister equivalent. Let $K_i \in \mathbb{K}(\Sigma_i)$ be a knot having diagram $D_i$. If $\upsilon_i(D_i)$ and $\delta_i(D_{i+1})$ are Reidemeister equivalent, then $\Upsilon_i(K_i)$ and $\Delta_i(K_{i+1})$ are ambient isotopic and hence lie in the same path component of $\mathbb{K}(\Sigma_{i+.5})$. Then there are path components $C_i \subseteq \mathbb{K}(\Sigma_i)$, $C_{i+.5} \subseteq \mathbb{K}(\Sigma_{i+.5})$, $C_{i+1} \subseteq \mathbb{K}(\Sigma_{i+1})$ such that $K_i \in C_i$, $K_{i+1} \in C_{i+1}$ and $\Upsilon_i(C_i) \subseteq C_{i+.5}$, $\Delta_{i+1}(C_{i+1}) \subseteq C_{i+.5}$. The reverse implication follows from the well-known fact that two knots in the same thickened surface are ambient isotopic if and only if they have diagrams that are Reidemeister equivalent.
\end{proof}

\subsection{The virtual knot site} \label{sec_site} The category $\textbf{VK}$ can now be defined as the category whose objects are generalized neighborhoods and whose arrows are generalized inclusions. That is, an object of $\textbf{VK}$ is an open set $U \subseteq \mathbb{K}(\Sigma)$, where $\Sigma$ ranges over all oriented surfaces without boundary. An arrow from $U_1 \subseteq \mathbb{K}(\Sigma_1)$ to $U_2 \subseteq \mathbb{K}(\Sigma_2)$ in $\textbf{VK}$ corresponds to a capitalized map $\Psi=\psi \times \text{id}:\mathbb{K}(\Sigma_1) \to \mathbb{K}(\Sigma_2)$ such that $\Psi(U_1) \subseteq U_2$. The arrow $\Psi:U_1 \to U_2$ is to be considered a function from the set $U_1$ to the set $U_2$. Hence, if $\Psi,\Phi: U_1 \to U_2$ are parallel arrows such that $\Psi(K_1)=\Phi(K_1)$ for all $K_1 \in U_1$, then $\Psi$ and $\Phi$ are the same arrow in $\textbf{VK}$. In particular, the empty set $\varnothing$ is an initial object in $\textbf{VK}$. The composition of $\Phi:U \to V$ with $\Psi:V \to W$ is defined by $\Psi \circ \Phi=(\psi \circ \phi) \times \text{id}$. The next two results establish some useful properties of $\textbf{VK}$.

\begin{lemma} \label{lemma_factors} Let $U_1 \subseteq \mathbb{K}(\Sigma_1), U_2 \subseteq \mathbb{K}(\Sigma_2)$ be objects and $\Psi=\psi \times \text{id}:U_1 \to U_2$ an arrow in $\textbf{VK}$. Then $\Psi$ factors as: 
\[
\xymatrix{ \Psi: U_1 \ar[rr]^-{\Psi'=\psi' \times \text{id}} & & U_2' \ar[rr]^-{I=\iota \times \text{id}} & & U_2},
\]
where $U_2' \subseteq \mathbb{K}(\Sigma_2')$, $\Sigma_2' \subseteq \Sigma_2$ is a subsurface, $\psi':\Sigma_1 \to \Sigma_2'$ is an orientation preserving diffeomorphism, $\iota:\Sigma_2' \to \Sigma_2$ is inclusion, $\psi=\iota \circ \psi'$, and $\Psi':U_1 \to U_2'$ is an isomorphism in $\textbf{VK}$.
\end{lemma}

\begin{proof} Since $\psi$ is an embedding and $\Sigma_1$ is a surface without boundary, $\Sigma_2'=\text{im}(\psi)$ is a subsurface of $\Sigma_1$. Define $\iota:\Sigma_2' \to \Sigma_2$ to be the inclusion and define $\psi':\Sigma_1 \to \Sigma_2'$ by $\psi'(z)=\psi(z)$ for all $z \in \Sigma_1$. Then $\psi'$ is an orientation preserving diffeomorphism and $\psi=\iota \circ \psi'$. Also, $\Psi':\mathbb{K}(\Sigma_1) \to \mathbb{K}(\Sigma_2')$ is a diffeomorphism. Set $U_2'=\Psi'(U_1)$. Then there is an arrow $\Psi':U_1 \to U_2'$ and an arrow $(\Psi')^{-1}:U_2' \to U_1$ corresponding the inverse diffeomorphism of $\Psi':\mathbb{K}(\Sigma_1) \to \mathbb{K}(\Sigma_2')$. These arrows satisfy $\Psi' \circ (\Psi')^{-1}=1_{U_2'}$ and $(\Psi')^{-1}\circ\Psi' =1_{U_1}$, so that $\Psi'$ is an isomorphism. 
\end{proof}

\begin{remark} Suppose that $I:U' \to U$ is an arrow where $I$ is capitalized from an inclusion map $\iota:\Sigma' \to \Sigma$ as in Lemma \ref{lemma_factors}. We will often abuse terminology and refer to $I$ itself as an inclusion.
\end{remark}

\begin{theorem} \label{lemma_vk_pullbacks} The category  $\textbf{VK}$ has pullbacks.
\end{theorem}
\begin{proof} First consider the case of a diagram $\xymatrix{U_1 \ar[r]^-{I_1} & U & \ar[l]_-{I_2} U_2}$ in $\textbf{VK}$ where $U_1 \subseteq \mathbb{K}(\Sigma_1)$, $U_2 \subseteq \mathbb{K}(\Sigma_2)$, $U \subseteq \mathbb{K}(\Sigma)$, and $\Sigma_1,\Sigma_2 \subseteq \Sigma$ are subsurfaces. We also assume that $i=1,2$, $I_i=\iota_i \times \text{id}: \Sigma_i \times \mathbb{R} \to \Sigma \times \mathbb{R}$ where $\iota_i$ is the inclusion map for $\Sigma_i \subseteq \Sigma$. Note also that $\Sigma_i$ is an open subset of $\Sigma$. Hence $\Sigma_1 \cap \Sigma_2$ is a subsurface of $\Sigma_1$, $\Sigma_2$, and $\Sigma$. Now, if $I_1(U_1) \cap I_2(U_2) =\varnothing$, we can take $\varnothing$ as the pullback. Otherwise, the knots in $I_1(U_1) \cap I_2(U_2)$ all have image lying over $\Sigma_1 \cap \Sigma_2$. Denote by $W$ the set of these knots lying over $\Sigma_1 \cap \Sigma_2$, so that $W \subseteq \mathbb{K}(\Sigma_1 \cap \Sigma_2)$. Let $J_1:W \to U_1$, $J_2:W \to U_2$ denote the arrows induced by the inclusion maps. We claim that the following square is a pullback. Suppose that $V \subseteq \mathbb{K}(\Gamma)$ and $\Phi=\phi \times \text{id}:V \to U_2,\Psi=\psi \times \text{id}:V \to U_1$ are maps so that the diagram below commutes. 
\[
\xymatrix{
V \ar@{-->}[dr]^-{\Upsilon} \ar@/^2pc/[drr]_{\Phi} \ar@/_2pc/[ddr]^-{\Psi} & & \\
 & W \ar[d]_-{J_1} \ar[r]^-{J_2} & U_2 \ar[d]^-{I_2} \\
 & U_1 \ar[r]_-{I_1} & U
}
\]
Since $I_1$ and $I_2$ are inclusions, this implies $\Psi(K)$ and $\Phi(K)$ are the same knot lying over $\Sigma_1 \cap \Sigma_2$ for all $K \in V$ and that their common image is contained in $W \subseteq \mathbb{K}(\Sigma_1 \cap \Sigma_2)$. Set $\upsilon=\psi:\Gamma \to \Sigma_1 \cap \Sigma_2$. Then there is an arrow $\Upsilon:V \to W$ in $\textbf{VK}$ such that $J_1 \circ \Upsilon(K)=\Psi(K)$ and $J_2 \circ \Upsilon(K)=\Phi(K)$ for all $K \in V$. As any other arrow satisfying the universal property must also agree with $\Psi$ for every $K \in V$, $\Upsilon$ is unique by our convention on parallel arrows.  Thus, the diagram above is a pullback.
\newline
\newline
Now consider the general case of a diagram $\xymatrix{U_1 \ar[r]^-{\Psi_1} & U & \ar[l]_-{\Psi_2} U_2}$. Apply Lemma \ref{lemma_factors} to both $\Psi_1,\Psi_2$. Then $\Psi_i$ factors as $\xymatrix{U_i \ar[r]^-{\Psi_i'} & U_i' \ar[r]^-{I_i} &  U}$ for some open set $U_i' \subseteq \mathbb{K}(\Sigma_i')$ and $\Psi_i'$ is an isomorphism in $\textbf{VK}$. The previous case implies there is a pullback $W'$ of $I_1$ over $I_2$ as below:
\[
\xymatrix{
W' \ar[r]^-{J_2} \ar[d]_-{J_1} & U_2' \ar[ddr]^-{I_2} \ar[r]^-{(\Psi_2')^{-1}} & U_2 \ar[dd]^-{\Psi_2} \\
U_1' \ar[d]_-{(\Psi_1')^{-1}} \ar[drr]_-{I_1} & & \\
U_1 \ar[rr]_-{\Psi_1} & & U
}
\]
The universal property of the upper left pullback square then implies that the outside square is also a pullback. Thus, $\textbf{VK}$ has pullbacks.
\end{proof}

Next we define a Grothendieck topology $J_{\textbf{VK}}$ on $\mathbf{VK}$. Recall that a \emph{sieve} on an object $C$ in a small category $\textbf{C}$ is a set $S$ of arrows with common codomain $C$ that is closed under composition on the right: if $(f:B \to C) \in S$ and $g:A \to B$, then $f \circ g \in S$. A Grothendieck topology $J_{\textbf{C}}$ assigns a set of sieves $J_{\textbf{C}}(C)$, called \emph{covering sieves}, to each object $C$ of $\textbf{C}$ so that certain axioms are satisfied (see \cite{mac_moer}, III.2, Definition 1). The axioms generalize open coverings of topological spaces to categories. In particular, for a topological space one may take the open cover topology as a Grothendieck topology on its lattice of open sets. Hence, for the knot space $\mathbb{K}$, a sieve $S$ is in $J_{\textbf{K}}(U)$ if $U=\bigcup_{(V \subseteq U) \in S} V$. Since arrows $\Psi:V \to U$ of $\textbf{VK}$ are generalized inclusions, it is natural to define a sieve $S$ on an object $U$ of $\textbf{VK}$ to be covering if $U$ is the union of its images. Thus, for an object $U$ of $\mathbf{VK}$, a sieve $S$ on $U$ is in  $J_{\textbf{VK}}(U)$ if:
\[
U=\bigcup_{(\Psi: V \to U) \in S} \Psi(V).
\] 
\begin{lemma} \label{lemma_grothendieck_topology} $J_{\textbf{VK}}$ is a Grothendieck topology on $\mathbf{VK}$.
\end{lemma}
\begin{proof} Suppose that $U \subseteq \mathbb{K}(\Sigma)$ is an object of $\textbf{VK}$. The maximal sieve $\{\Psi:V \to U|\, \text{cod}(\Psi)=U \}$ on $U$ is in $J_{\textbf{VK}}(U)$, since it contains all of the inclusion arrows for any open cover of $U$ in $\mathbb{K}(\Sigma)$. This is the first axiom. For the second axiom (stability), it must be shown that for any arrow $\Phi:V \to U$ and $S \in J_{\textbf{VK}}(U)$, we have that $\Phi^*(S) \in J_{\textbf{VK}}(V)$, where $\Phi^*(S)=\{\Psi| \Phi \circ \Psi \in S\}$. Then we need to show $\bigcup_{\Psi \in \Phi^*(S)} \Psi(\text{dom}(\Psi))=V$. Let $K \in V$. Since $S$ covers $U$, there is an $\Upsilon: O \to U$ in $S$ such that $\Phi(K) \in \Upsilon(O)$. Now pull back $\Upsilon$ back over $\Phi$ to obtain the following commutative diagram:
\[
\xymatrix{ W \ar[r]^{\Upsilon_1} \ar[d]_{\Phi_1} & O \ar[d]^{\Upsilon} \\
           V \ar[r]_{\Phi} & U
           } 
\]
The construction of the pullback in Lemma \ref{lemma_vk_pullbacks} implies that $K \in \Phi_1(W)$. Since $S$ is a sieve, $\Phi \circ \Phi_1 \in S$ and hence $\Phi_1 \in \Phi^*(S)$. Thus, $\bigcup_{\Psi \in \Phi^*(S)} \Psi(\text{dom}(\Psi))=V$ and the stability axiom is satisfied. Lastly, we must prove transitivity: if $S \in J_{\textbf{VK}}(U)$, $R$ is a sieve on $U$, and $\Phi^*(R) \in J_{\textbf{VK}}(V)$ for all $\Phi:V \to U$ in $S$, then $R \in J_{\textbf{VK}}(U)$. Since $\Phi^*(R) \in J_{\textbf{VK}}(V)$ for all $\Phi:V \to U$ in $S$, $V$ is a union of open sets $\Upsilon(W)$ where $\Upsilon:W \to V$ and $\Phi \circ \Upsilon \in R$. Then since $S$ covers $U$, $U=\bigcup_{\Phi \in S} \Phi(V)=\bigcup_{\Phi \in S} \Phi \left( \bigcup_{\Upsilon \in \Phi^*(R)} \Upsilon(W) \right)=\bigcup_{\Phi \in S, \Phi \circ \Upsilon \in R} \Phi \circ \Upsilon(W)$. Thus, $R \in J_{\textbf{VK}}(U)$.
\end{proof}

\subsection{The geometric model of virtual knot theory}  A model $VG$ of the language $\mathfrak{L}_g$ of geometric knot theory is obtained from the sheaf-theoretic model $ST$ of knot theory by simply replacing the sheaves on the site $(\textbf{K},J_{\textbf{K}})$ everywhere with sheaves on the site $(\textbf{VK},J_{\textbf{VK}})$. This yields the following interpretations of ``knot space'', ``knot'', ``isotopy'' and ``knot invariant''. 

\begin{definition}[Virtual knot space] \label{defn_virtual_knot_space} The \emph{virtual knot space} is the Grothendieck topos of sheaves on the site $(\textbf{VK},J_{\textbf{VK}})$. For simplicity, we will abbreviate  $\text{Sh}(\textbf{VK},J_{\textbf{VK}})$ by $\text{Sh}(\textbf{VK})$.
\end{definition}

\begin{definition}[Virtual knot] \label{defn_virtual_knot} A \emph{virtual knot} is a geometric morphism $K:\textbf{Sets} \to \text{Sh}(\textbf{VK})$. Two virtual knots of $K_1,K_0:\textbf{Sets} \to \text{Sh}(\textbf{VK})$ are considered the same if they are isomorphic in the category $\underline{\text{Hom}}(\textbf{Sets},\text{Sh}(\textbf{VK}))$ of such geometric morphisms. 
\end{definition}

\begin{definition}[Virtual isotopy] \label{defn_virtual_isotopy_paths} A \emph{virtual isotopy} is a geometric morphism $\sigma:\text{Sh}(\textbf{I}) \to \text{Sh}(\textbf{VK})$. Two virtual isotopies  $\sigma_1,\sigma_0:\text{Sh}(\textbf{I}) \to \text{Sh}(\textbf{VK})$ are considered the same if they are isomorphic in the category $\underline{\text{Hom}}(\text{Sh}(\textbf{I}),\text{Sh}(\textbf{VK}))$ of such geometric morphisms. If $0,1:\textbf{Sets} \to \text{Sh}(\textbf{I})$ are the geometric morphisms corresponding to the inclusions $t_0:\{0\}\rightarrow \mathbb{I}, t_1:\{1\} \rightarrow \mathbb{I}$, we will say $\sigma$ is a virtual isotopy between the virtual knots $\sigma \circ 0:\textbf{Sets} \to \text{Sh}(\textbf{VK})$ and $\sigma \circ 1:\textbf{Sets} \to \text{Sh}(\textbf{VK})$. 
\end{definition}

\begin{definition}[Virtual knot invariant] \label{defn_virtual_knot_invariant} For $\mathbb{G}$ a set with the discrete topology, a \emph{virtual knot invariant} is a geometric morphism $\nu:\text{Sh}(\textbf{VK}) \to \text{Sh}(\textbf{G})$. Two virtual knot invariants are considered the same if they are isomorphic in the category $\underline{\text{Hom}}(\text{Sh}(\textbf{VK}),\text{Sh}(\textbf{G}))$ of geometric morphisms.
\end{definition}

Then $VG$ assigns $\mathfrak{K}^{(VG)}$ to the set of virtual knots, $\mathfrak{I}^{(VG)} \subseteq \mathfrak{K}^{(VG)} \times \mathfrak{K}^{(VG)}$ to the equivalence relation generated by virtual isotopy, and function symbols are interpreted just as in Section \ref{sec_sheaf_theory_K}.

\section{Virtual knots} \label{sec_knots}

\subsection{Variable-space knots} \label{sec_var_ambient} Consider again the space $\mathbb{K}$ of knots in $\mathbb{R}^3$. A knot is a point in $\mathbb{K}$ and, since $\mathbb{K}$ is Hausdorff, it is the unique point in the intersection of all its neighborhoods. Thus, a knot $K$ in $\mathbb{R}^3$ can be identified as the subcategory of objects in $\textbf{K}$ that are neighborhoods of $K$ in $\mathbb{K}$. This can be conveniently expressed using two definitions from category theory: the category of elements and connected categories. If $F:\textbf{C} \to \textbf{Sets}$ is a covariant functor, the \emph{category of elements} $\smallint_{\textbf{C}} F$ of $F$ is the category whose objects are pairs $(C,c)$, where $c \in F(C)$, and whose arrows $(C_1,c_1) \to (C_2,c_2)$ correspond to arrows $f:C_1 \to C_2$ in $\textbf{C}$ such that $F(f)(c_1)=c_2$. Henceforth, we will often abbreviate $F(f)(c)$ as $f \cdot c$. A category $\textbf{C}$ is said to be \emph{connected} if any two pairs of objects $C,\widebar{C}$ of $\textbf{C}$, can be reached by a path of arrows. By a path of arrows, we mean a finite sequence of objects and arrows as follows: 
\[
\xymatrix{C=C_1 \ar[r]^-{u_1} & C_{1.5} & \ar[l]_-{d_2} C_2 \ar[r]^-{u_2}& \cdots & \ar[l]_-{d_{n-1}} C_{n-1} \ar[r]^-{u_{n-1}}& C_{n-.5} & \ar[l]_-{d_n} C_n=\widebar{C}}
\]
The direction of the arrows is immaterial, as any path can be made alternating by adding in appropriately oriented identity arrows. A \emph{connected component} of $\textbf{C}$ is a maximal connected subcategory.  
\newline
\newline
For the category $\textbf{K}$ of open sets of $\mathbb{K}$, let $F: \textbf{K} \to \textbf{Sets}$ be the forgetful functor, which maps an object $U$ to the set of knots it contains and each arrow $U \to V$ of $\textbf{K}$ to an inclusion of sets.

\begin{proposition} \label{prop_connected_comp} Knots $K \in \mathbb{K}$ correspond exactly with the connected components of $\smallint_{\textbf{K}} F$.
\end{proposition}
\begin{proof} For a knot $K \in \mathbb{K}$, there is a full subcategory of $\smallint_{\textbf{K}} F$ consisting of those objects $(U,K)$ where $U$ is a neighborhood of $K$. This subcategory is connected since the objects $(U,K)$ and $(V,K)$ are connected by the path $\xymatrix{(U,K) \ar@{=}[r] &(U,K) & \ar[l] (U \cap V,K) \ar[r] & (V,K)& (V,K) \ar@{=}[l]}$. Also, the subcategory is maximal, since all arrows in $\textbf{K}$ are inclusions. Thus, it is a connected component of $\smallint_{\textbf{K}} F$. On the other hand, if $(U,K)$ and $(V,J)$ are in the same connected component of $\smallint_{\textbf{K}} F$, then it must be that $K=J$, again since all arrows of $\textbf{K}$ are inclusions.  
\end{proof}
  
Now, apply this construction to $\textbf{VK}$. Let $F:\textbf{VK} \to \textbf{Sets}$ denote the forgetful functor mapping an object of $\textbf{VK}$ to its underlying set of knots and each arrow $\Psi$ to its underlying injective function.

\begin{definition}[Variable-space knot] \label{defn_varible_space_knot} A \emph{variable-space knot} is a connected component of $\smallint_{\textbf{VK}} F$. For a knot $K \in U \subseteq \mathbb{K}(\Sigma)$, let $\textbf{V}_K$ denote the variable-space knot satisfying $(U,K) \in \text{obj}(\textbf{V}_K)$. 
\end{definition}

The observant reader will have noticed a formal similarity between the definition of a connected category and the definition of stable Reidemeister equivalence (see Section \ref{sec_lang_diag}). Suppose that we have a sequence of stable Reidemeister moves: 
\small
\[
\xymatrix{(\Sigma_1,D_1) \ar[r]^-{\upsilon_1} & (\Sigma_{1.5},D_{1.5}) & \ar[l]_-{\delta_2} (\Sigma_2,D_2) \ar[r]^-{\upsilon_2} & \cdots &\ar[l]_-{\delta_{n-1}} (\Sigma_{n-1},D_{n-1}) \ar[r]^-{\upsilon_{n-1}} & (\Sigma_{n-.5},D_{n-.5}) & \ar[l]_-{\delta_n} (\Sigma_n,D_n), }
\]
\normalsize
where $\upsilon_1(D_1)=D_{1.5}=\delta_2(D_2)$, $\upsilon_2(D_2)=D_{2.5}=\delta_3(D_3)$,$\ldots$, $\upsilon_{n-1}(D_{n-1})=D_{n-.5}=\delta_n(D_n)$. Then the only possible change at each step is to vary the underlying surface. Choosing knots $K_i$ having diagrams $D_i$ appropriately and setting $U_i=\mathbb{K}(\Sigma_i)$, we then have a path of arrows in $\smallint_{\textbf{VK}}F$:
\small
\[
\xymatrix{(U_1,K_1) \ar[r]^-{\Upsilon_1} & (U_{1.5},K_{1.5}) & \ar[l]_-{\Delta_2} (U_2,K_2) \ar[r]^-{\Upsilon_2} & \cdots &\ar[l]_-{\Delta_{n-1}} (U_{n-1},K_{n-1}) \ar[r]^-{\Upsilon_{n-1}} & (U_{n-.5},K_{n-.5}) & \ar[l]_-{\Delta_n} (U_n,K_n),}
\]
\normalsize
where $\Upsilon_1(K_1)=K_{1.5}=\Delta_2(K_2)$, $\Upsilon_2(K_2)=K_{2.5}=\Delta_3(K_3)$,$\ldots$, $\Upsilon_{n-1}(K_{n-1})=K_{n-.5}=\Delta_n(K_n)$. A variable-space knot is thus a fixed knot together with all the ways it appears in the various thickened surfaces. Furthermore, when viewed as a component of the category of elements $\smallint_{\textbf{VK}} F$, a variable-space knot has the identical description as a classical knot. Hence, variable-space knots link together three seemingly different concepts: geometric knots in $\mathbb{R}^3$, virtual knot diagrams, and, as will be shown in the next subsection, the points of the virtual knot space.

\subsection{Virtual knots $\&$ variable-space knots} \label{sec_points} Now we prove that variable-space knots (Definition \ref{defn_varible_space_knot}) are in one-to-one correspondence with virtual knots (Definition \ref{defn_virtual_knot}). This will be done using the equivalence between the category $\underline{\text{Hom}}(\textbf{Sets},\text{Sh}(\textbf{VK}))$ of geometric morphisms $K:\textbf{Sets} \to \text{Sh}(\textbf{VK})$ and the category of continuous filtering functors $A:\textbf{VK} \to \textbf{Sets}$ (see \cite{mac_moer}, Sections VII.1, VII.2, VII.5, and VII.6). Recall that a covariant functor $A:\textbf{C} \to \textbf{Sets}$ is \emph{filtering} (\cite{mac_moer}, Section VII.6, Definition 2), if the following conditions are satisfied. 
\begin{enumerate}
\item[$(i)$] $A(C) \ne \varnothing$ for some object $C$ of $\textbf{C}$.
\item[$(ii)$] If $b \in A(B)$ and $c \in A(C)$, then there is a diagram $\xymatrix{B & \ar[l]_-{f} D \ar[r]^-{g} & C}$ in $\textbf{C}$ and a $d \in A(D)$ such that  $f \cdot d=b$ and $g \cdot d=c$.
\item[$(iii)$] If $f,g:C \to D$ satisfy $f \cdot c=g \cdot c$ for some $c \in A(C)$, then there is an arrow $e:B \to C$ and a $b \in A(B)$ such that $f\circ e=g \circ e$ and $e \cdot b=c$. The arrow $e$ is called an \emph{equalizer} of $f,g$.
\end{enumerate}
Furthermore, if $(\textbf{C},J_{\textbf{C}})$ is a site, a filtering functor $A:\textbf{C} \to \textbf{Sets}$ is \emph{continuous} if $A$ maps every covering sieve to a jointly epimorphic family of functions in $\textbf{Sets}$. Hence, if $S \in J_{\textbf{C}}(U)$ and $u \in A(U)$, there is an $(f:V \to U) \in S$ and $v \in A(V)$ such that $A(f)(v)=u$.
\newline
\newline
Suppose that $A:\textbf{VK} \to \textbf{Sets}$ is continuous and filtering. The first task is to show that $A$ determines a variable-space knot. The semantics of the word ``determines'' will be made clear in the following lemma and definition. First we set some notation. Suppose $A(U) \ne \varnothing$ for some object $U\subseteq \mathbb{K}(\Sigma)$ of $\textbf{VK}$. For every $u \in A(U)$, set:
\[
S_u=\{\Psi: V\to U|\Psi\cdot v=A(\Psi)(v)=u \text{ for some } v \in A(V)\}.
\]
Since $A$ maps covering sieves to epimorphic families, $S_u \ne \varnothing$ for every $u \in A(U)$. The next lemma shows that the images of all arrows in $S_u$ intersect in a single knot.

\begin{lemma} \label{lemma_determines} $\bigcap_{(\Psi:V \to U) \in S_u} \Psi(V)=\{K\}$ for some knot $K \in U$.
\end{lemma}
\begin{proof} Suppose that $\bigcap_{\Psi \in S_u} \Psi(V)=\varnothing$. Then for every $K \in U$, there is a map $\Phi_K:V_K \to U$ such that $K \in \Phi_K(V_K)$ and $u$ is not in the image of $A(\Phi_K)$. Let $S$ be the sieve generated by these maps $\Phi_K$, i.e. the set of arrows of the form $\Phi_K \circ \Upsilon$ where $\Upsilon:W \to V_K$. Since $K \in \Phi_K(V_K)$ for all $K \in U$, $S \in J_{\textbf{VK}}(U)$. But $A$ does not map $S$ to an epimorphic family. Hence $\bigcap_{\Psi\in S_u} \Psi(V) \ne \varnothing$.
\newline
\newline
Now suppose that $K_1,K_2 \in \bigcap_{\Psi\in S_u} \Psi(V)$ and $K_1 \ne K_2$. Since $\mathbb{K}(\Sigma)$ is Hausdorff, there are disjoint open neighborhoods $V_1 \subseteq U$ of $K_1$ and $V_2 \subseteq U$ of $K_2$. Furthermore, since one-point sets are closed, $\{V_1,V_2, V_3=U\smallsetminus \{K_1,K_2\}\}$ is an open cover of $U$. Consider the covering sieve $S$ of $U$ generated by the three inclusion arrows $I_j:V_j \to U$ for $j=1,2,3$. Now, $A$ maps covering sieves to epimorphic families, so that there must be an arrow $\Psi:W \to U$ in $S$ such that $\Psi \cdot w=u$ for some $w \in A(W)$. Hence, $\Psi \in S_u$. Since $\Psi \in S$, it must factor through one of $V_1,V_2,V_3$. But $\Psi$ cannot factor through $V_3$, because its image must contain both $K_1$ and $K_2$. Likewise, $\Psi$ cannot factor through $V_1$ because its image must contain $K_2$. Similarly, $\Psi$ cannot factor through $V_2$. This is a contradiction, so we must have that $K_1=K_2$.
\end{proof}

\begin{definition} \label{defn_determines} For a continuous filtering functor $A:\textbf{VK} \to \textbf{Sets}$, $U \subseteq \mathbb{K}(\Sigma)$, $u \in A(U) \ne \varnothing$, the knot in $\bigcap_{\Psi \in S_u} \Psi(V)$ will be called the \emph{knot in $U$ determined by $u \in A(U)$}.
\end{definition}

Therefore, it must be shown that the knots determined by a continuous filtering functor form a connected component of $\smallint_{\textbf{VK}} F$. This is stated below, but the proof will be momentarily delayed. 

\begin{lemma} \label{lemma_meta_determines} Let $A:\textbf{VK} \to \textbf{Sets}$ be continuous and filtering. Let $\textbf{V}$ be the full subcategory of $\smallint_{\textbf{VK}} F$ consisting of objects $(U,K)$ such that $K$ is the knot in $U$ determined by some $u \in A(U)$. Then $\textbf{V}$ is a variable-space knot. 
\end{lemma}

\begin{proof} See Section \ref{sec_proof_meta_determines}
\end{proof}

Since virtual knots are the points of $\text{Sh}(\textbf{VK})$ considered equivalent up to natural isomorphism, it must also be proved that isomorphic continuous filtering functors determine the same variable-space knot. This is the next result, which should be compared with the proof of Proposition \ref{prop_convert_to_2_cat}.

\begin{lemma} \label{lemma_natural_imp_same_knots} If $A,B:\textbf{VK} \to \textbf{Sets}$ are continuous filtering functors and there is a natural transformation $\alpha:A \to B$, then $A$ and $B$ determine the same variable-space knot.
\end{lemma}
\begin{proof} Suppose that $K \in U$ is determined by $u \in A(U)$. By Lemma \ref{lemma_meta_determines}, it suffices to show that $(U,K)$ is an object of the connected component of $\smallint_{\textbf{VK}} F$ determined by $B$. Define: 
\begin{align*}
S_u^A &=\{\Phi:W \to U|\exists a \in A(W) \text{ s.t. } A(\Phi)(a)=u\}, \text{ and} \\
S_{\alpha_U(u)}^B &= \{\Phi:W \to U|\exists b \in B(W) \text{ s.t. } B(\Phi)(b)=\alpha_U(u)\}.
\end{align*}
Now, if $\Phi \in S_u^A$, the naturality of $\alpha$ implies that:
\[
B(\Phi)\circ \alpha_W (a)=\alpha_U \circ A(\Phi)(a)=\alpha_U(u).
\]
Then $\Phi \in S_{\alpha_U(u)}^B$ and it follows that $S_u^A \subseteq S_{\alpha_U(u)}^B$. Hence, $\bigcap_{\Phi \in S_{\alpha_U(u)}^B} \Phi(W) \subseteq \bigcap_{\Phi \in S_u^A} \Phi(W)=\{K\}$. Thus, $K 
\in U$ is determined by $\alpha_U(u) \in B(U)$.
\end{proof}

Conversely, every variable-space knot uniquely (up to isomorphism) determines a continuous filtering functor. We state the result but again delay its proof until after the main theorem is established.

\begin{lemma} \label{lemma_inverse_determines} Let $K \in O \subseteq \mathscr{K}(\Sigma)$ be a knot and $\textbf{V}_K$ its variable-space knot. 
\begin{enumerate}
\item There is a continuous filtering functor $A_K:\textbf{VK} \to \textbf{Sets}$ that determines $\textbf{V}_K$.
\item If $A: \textbf{VK} \to \textbf{Sets}$ determines $\textbf{V}_K$, then $A \cong A_K$.
\end{enumerate}
\end{lemma}

\begin{proof} See Section \ref{sec_proof_inverse_determines}.
\end{proof}

\begin{theorem}\label{thm_points_sh_vk}  Variable-space knots are in one-to-one correspondence with virtual knots.
\end{theorem}

\begin{proof} 
Since the category of geometric morphisms $\underline{\text{Hom}}(\textbf{Sets},\text{Sh}(\textbf{VK}))$ is equivalent to the category of continuous filtering functors $A:\textbf{VK} \to \textbf{Sets}$, it suffices to show such a correspondence between variable-space knots and continuous filtering functors. By Lemma \ref{lemma_natural_imp_same_knots}, isomorphic continuous filtering functors determine the same variable-space knot. By Lemma \ref{lemma_inverse_determines}(1), this correspondence is surjective. To see that it is injective, suppose that $A,B:\textbf{VK} \to \textbf{Sets}$ determine the same variable-space knot $\textbf{V}$. If $(O,K) \in \text{obj}(\textbf{V})$, then Lemma \ref{lemma_same_knots_imp_nat_iso}(2) implies that $A \cong A_K \cong B$.
\end{proof}

\subsection{Proof of Lemma \ref{lemma_meta_determines}} \label{sec_proof_meta_determines} By Lemma \ref{lemma_determines}, each $u \in A(U)$ determines a knot $K \in U$. The next three lemmas describe how the arrows of $\textbf{VK}$ act on $u$ and $K$. Lemma \ref{lemma_meta_determines} is proved afterwards.

\begin{lemma} \label{lemma_determines_1} If $K \in U \subseteq \mathbb{K}(\Sigma)$ is the knot determined by $u \in A(U)$, $\Phi:W \to U$ is an arrow, and $\Phi(J)=K$ for some $J \in W$, then there is a $w \in W$ such that $\Phi\cdot w=u$.
\end{lemma}
\begin{proof} Choose an open cover of $U$ containing $\Phi(W)$ such that $\Phi(W)$ is the only open set in the cover that contains $K$. Consider the sieve $S$ generated by the inclusion maps from elements of the open cover to $U$, where in place of the inclusion $\Phi(W) \to U$, we use the arrow $\Phi:W \to U$. Clearly, $S$ is a covering sieve. Since $A$ maps $S$ to an epimorphic family, there is some arrow $\Upsilon:Y \to U$ such that $\Upsilon \cdot y=u$ for some $y \in A(Y)$ and hence $\Upsilon \in S_u$. Now, $\Upsilon$ must factor through either $\Phi$ or one of the other elements of the open cover. Suppose that $\Upsilon$ factors through some inclusion $\Omega:O \to U$ in the open cover, so that $\Omega \ne \Phi$. Then $K \not\in \Upsilon(Y)$, which contradicts the fact that $\Upsilon \in S_u$ and $\bigcap_{\Psi \in S_u} \Psi(\text{dom}(\Psi))=\{K\}$. It follows that $\Upsilon$ must factor through $\Phi$. Then $\Upsilon=\Phi \circ \Psi$ for some map $\Psi:Y \to W$ and $\Phi \cdot (\Psi \cdot y)=\Upsilon \cdot y=u$. Setting $w=\Psi \cdot y \in A(W)$ proves the lemma. 
\end{proof}

\begin{lemma} \label{lemma_determines_2} Suppose that $\Phi:W \to U$ and $\Phi \cdot w=u$ for some $w \in A(W)$ and $u \in A(U)$. If $K \in U$ is the knot determined by $u$ and $L \in W$ is the knot determined by $w$, then $\Phi(L)=K$.
\end{lemma}
\begin{proof} Write $S_w=\{\Upsilon:\widebar{W} \to W|\exists \widebar{w} \in A(\widebar{W}) \text{ s.t. } \Upsilon \cdot \widebar{w}=w\}$. Thus, $\{L\}=\bigcap_{\Upsilon_i} \Upsilon(\widebar{W})$. By hypothesis, $\Phi \cdot \Upsilon \cdot \widebar{w}=u$. Then $K$ is in the image of $\Phi \circ \Upsilon$ for all $\Upsilon \in S_w$. Hence, the preimage of $K$ (by $\Phi$) is in the image of $\Upsilon$ for all $\Upsilon \in S_w$. Then $\Phi^{-1}(\{K\})=\{L\}$ and the result follows.
\end{proof}

\begin{lemma} \label{lemma_determines_3}Suppose that $\Phi:V \to U$, $v \in A(V)$, and $J \in V$ is the knot determined by $v$. Then $\Phi(J)$ is the knot determined by $\Phi \cdot v$. 
\end{lemma}
\begin{proof} Write $S_{\Phi \cdot v}=\{\Psi:\widebar{U} \to U|\exists \widebar{u} \in A(\widebar{U}) \text{ s.t. } \Psi \cdot \widebar{u}=\Phi \cdot v\}$. Since $A$ is filtered and $A(\widebar{U}),A(V)$ are nonempty, there is for each $\Psi \in S_{\Phi \cdot v}$ a diagram $\xymatrix{V & \ar[l]_-{\Upsilon} W \ar[r]^-{\Gamma} & \wbar{U}}$ such that $\Upsilon \cdot w=v$ and $\Gamma\cdot w=\widebar{u}$ for some $w \in A(W)$. Then for each $\Psi\in S_{\Phi \cdot v}$, we have the parallel arrows $\Phi \circ \Upsilon,\Psi \circ \Gamma:W \to U$. These arrows satisfy $\Phi \cdot \Upsilon \cdot w=\Phi \cdot v=\Psi \cdot \Gamma \cdot w$. By the third filtering property of $A$, there is an arrow $\Xi:X \to W$ such that $\Phi \circ \Upsilon \circ \Xi=\Psi \circ \Gamma \circ \Xi$ and $\Xi \cdot x=w$ for some $x \in A(X)$. Then $\Upsilon \cdot \Xi \cdot x=v$ and hence $J \in \Upsilon \circ \Xi(X)$. This implies $\Phi(J) \in \Phi \circ \Upsilon \circ \Xi(X)=\Psi \circ \Gamma \circ \Xi(X)$. Hence, $\Phi(J) \in \Psi(\widebar{U})$ for all $\Psi \in S_{\Phi \cdot v}$ and $\Phi(J)$ is the knot determined by $\Phi \cdot v$.
\end{proof}

\begin{proof}[Proof of Lemma \ref{lemma_meta_determines}] Choose some object $U$ of $\textbf{VK}$ such that $A(U) \ne \varnothing$. For $u \in A(U)$, let $K\in U$ be the knot determined by $u$. Let $J \in V$ be another knot in the same connected component of $\smallint_{\textbf{VK}} F$. Then there is a path of arrows in $\smallint_{\textbf{VK}} F$ connecting $(U,K)=(U_1,K_1)$ and $(V,J)=(U_{n},K_{n})$:
\small
\[
\xymatrix{
(U,K)=(U_1,K_1) \ar[d]_-{\Upsilon_1} & (U_2,K_2) \ar[dl]_-{\Delta_2} \ar[d]^-{\Upsilon_2} & \ar[dl]_-{\Delta_3} \cdots & (U_{n-1},K_{n-1}) \ar[dl]_-{\Delta_{n-1}} \ar[d]^-{\Upsilon_{n-1}} & (U_n,K_n)=(V,J) \ar[dl]_-{\Delta_{n}} \\
(U_{1.5},K_{1.5}) & (U_{2.5},V_{2.5}) & \cdots & (U_{n-.5},K_{n-.5}) & 
}
\]
\normalsize 
Set $u_1=u$. By Lemma \ref{lemma_determines_3}, $K_{1.5}=\Upsilon_1(K_1)$ is the knot determined by $\Upsilon_1 \cdot u_1$. By Lemma \ref{lemma_determines_1}, there is a $u_2 \in A(U_2)$ such that $\Delta_2 \cdot u_2=\Upsilon_1 \cdot u_1$ and $K_2$ is the knot determined by $u_2 \in U_2$. Continuing inductively, it follows that for all $2 \le i \le n$ that there is a $u_{i} \in A(U_{i})$ such that $\Delta_i\cdot u_i=\Upsilon_{i-1} \cdot u_{i-1}$ and $K_{i} \in U_i$ is the knot determined by $u_i$. Thus, $J \in V$ is a knot determined by $A$.
\newline
\newline
Lastly, it must be shown that if $K \in U$ is determined $u \in A(U)$ and $J \in V$ is determined by $v \in A(V)$, then $(U,K)$ and $(V,J)$ are in the same connected component of $\smallint_{\textbf{VK}} F$. Since $A$ is filtered, there is a diagram $\xymatrix{U & \ar[l]_-{\Phi} W \ar[r]^-{\Psi} & V}$ such that for some $w \in A(W)$, $\Phi \cdot w =u$ and $\Psi \cdot w =v$. Then if $L \in W$ is the knot determined by $w \in A(W)$, Lemma \ref{lemma_determines_3} implies that $\Phi(L)=K$ and $\Psi(L)=J$. Thus, $(U,K),(V,J)$ are in the same connected component of $\smallint_{\textbf{VK}} F$.
\end{proof}

\subsection{Proof of Lemma \ref{lemma_inverse_determines}} \label{sec_proof_inverse_determines} Fix a knot $K \in O \subseteq \mathbb{K}(\Sigma)$ and let $\textbf{V}_K$ be the connected component of $\smallint_{\textbf{VK}} F$ containing $(O,K)$. We will show there is a continuous filtering functor $A_K:\textbf{VK} \to \textbf{Sets}$ that determines $\textbf{V}_K$. Lemmas \ref{lemma_determines} , \ref{lemma_determines_1} , \ref{lemma_determines_2}, and \ref{lemma_determines_3} provide some insight into how $A_K$ must be defined. Suppose that $\Psi_1,\Psi_2:O \to U$ are arrows, $\Psi_1(K)$ is determined by some $u_1 \in A_K(U)$, and $\Psi_2(K)$ determined by some $u_2 \in A_K(U)$. If $u_2=u_1$, then $\Psi_1(K)=\Psi_2(K)$ and there is an $o \in A_K(O)$ such that $\Psi_1 \cdot o=u_1=u_2=\Psi_2 \cdot o$. The third filtering property applied to $A_K$ means that $\Psi_1$ and $\Psi_2$ can be equalized. In other words, they agree on some open set $O' \subseteq O$. Each $u \in A_K(U)$ therefore represents a set of arrows $ \Psi: O \to U$ that agree on some (generalized) neighborhood $O'$ of $K$. Intuitively, $A_K(U)$ is the set of all the ways $K$ gets mapped into $U$, with two maps equivalent if they agree on a neighborhood of $K$. We make this precise with a definition.

\begin{definition}[$(O,K)$-neighborhood] Suppose that $O' \subseteq \mathbb{K}(\Sigma')$ and $\Sigma' \subseteq \Sigma$ is a subsurface. Suppose that $K:\mathbb{S}^1 \to \Sigma \times \mathbb{R}$ factors as $\xymatrix{\mathbb{S}^1 \ar[r]^-{K'} & \Sigma' \times \mathbb{R} \ar[r]^{\iota \times \text{id}} & \Sigma \times \mathbb{R}}$, where $\iota:\Sigma' \to \Sigma$ is inclusion and $K' \in O'$. Then $O'$ is called and \emph{$(O,K)$-neighborhood} of $K$. By abuse of notation, the knot $K' \in O'$ will also be denoted by $K$. Let $\mathscr{O}_K$ be the set of $(O,K)$-neighborhoods of $K$. By Theorem \ref{lemma_vk_pullbacks}, $\mathscr{O}_K$ is closed under pullbacks: if $O_1',O_2' \in \mathscr{O}_K$ with inclusions $I_1:O_1' \to O$ and $I_2:O_2' \to O$, then the pullback $O_1' \times_{O} O_2'$ of the inclusions is also an $(O,K)$-neighborhood. 
\end{definition}

Now, let $U$ be any object of $\textbf{VK}$ and let $O_1',O_2' \in \mathscr{O}_K$, where $O_1' \subseteq \mathbb{K}(\Sigma_1')$, $O_2' \subseteq \mathbb{K}(\Sigma_2')$. Suppose $\Phi_1:O_1' \to U$ and $\Phi_2:O_2' \to U$ are arrows. Write $\Phi_1 \sim \Phi_2$ if there is an $O_3' \in \mathscr{O}_K$ and inclusions $I_1:O_3' \to O_1'$, $I_2:O_3' \to O_2'$ such that $\Phi_1 \circ I_1=\Phi_2 \circ I_2$, where $O_3' \subseteq \mathbb{K}(\Sigma_3')$ and $\Sigma_3' \subseteq \Sigma_1' \cap \Sigma_2'$. Note $\sim$ is an equivalence relation. We will write $\sim$-equivalence class of $\Phi:O' \to U$ for $O' \in \mathscr{O}_K$ by $[\Phi]$. 
\newline
\newline
The functor $A_K: \textbf{VK} \to \textbf{Sets}$ can now be defined as follows. For $U$ an object of $\textbf{VK}$, define $A_K(U)$ to be the set of $\sim$-equivalence classes $[\Phi]$, $\Phi:O'\to U$ for $O' \in \mathscr{O}_K$. For an arrow $\Psi:U \to V$, define $A_K(\Psi):A_K(U) \to A_K(V)$ by $A_K(\Psi)([\Phi])=[\Psi \circ \Phi]$. A simple example is given below.

\begin{example} \label{example_A_K} Let $\Sigma$ be the open unit disc in $\mathbb{C}$ and $K:\mathbb{S}^1\to\Sigma \times\mathbb{R}$ be given by $K(e^{i \theta})=(\tfrac{1}{2} e^{i \theta},0)$. A diagram of $K$ on $\Sigma$ is depicted in Figure \ref{fig_example_A_K}. Take $O$ be the path component of $\mathbb{K}(\Sigma)$ containing $K$. Let $\Xi=\mathbb{S}^1 \times \mathbb{S}^1$ and set $U=\mathbb{K}(\Xi)$. Let $\Sigma'$ be the open annulus $\{re^{i \theta}|\tfrac{1}{4} <r<1\}$ and $O'$ be the path component of $\mathbb{K}(\Sigma')$ containing $K$. Then $O'$ is an $(O,K)$-neighborhood. Consider the maps $\Psi_1:O \to U$ and $\Psi_2 ,\Psi_3,\Psi_4:O' \to U$ depicted schematically in Figure \ref{fig_example_A_K}. Clearly, $[\Psi_1] \neq [\Psi_4]$ in $A_K(U)$, since $\Psi_1(O) \cap \Psi_4(O)=\varnothing$. For $\Psi_2=\psi_2 \times \text{id}:\mathbb{K}(\Sigma') \to \mathbb{K}(\Xi)$, take $\psi_2(re^{i \theta})=(e^{i r\pi/8},e^{i \theta})$. For $\Psi_3=\psi_3 \times \text{id}:\mathbb{K}(\Sigma') \to \mathbb{K}(\Xi)$, let $\psi_3$ be the composition of $\psi_2$ with a Dehn twist in the image of the annulus $\{re^{i \theta}|\tfrac{1}{2} \le r\le \tfrac{3}{4}\}$. This annulus is painted cyan in Figure \ref{fig_example_A_K}. Note that $[\Psi_2],[\Psi_3]$ are both distinct from $[\Psi_1]$ and $[\Psi_4]$ in $A_K(U)$. Furthermore, $[\Psi_2]\ne [\Psi_3]$, even though $\Psi_2(K)=\Psi_3(K)$. Perturbing $K$ an arbitrarily small amount into the cyan region forces $\Psi_2$ and $\Psi_3$ to disagree.
\end{example}

\begin{figure}[htb]
\def\svgwidth{4in}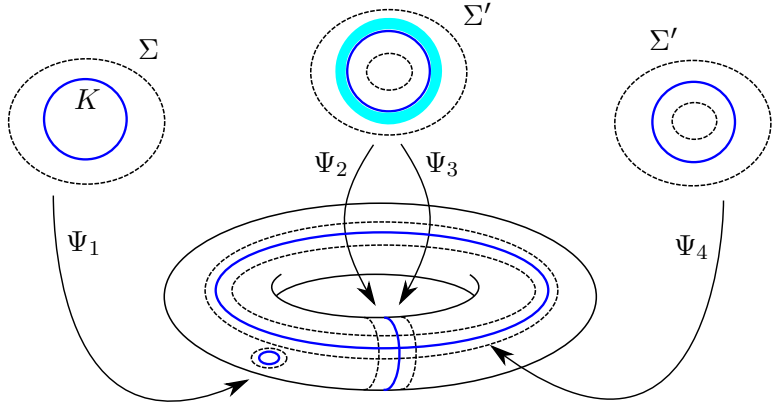
\caption{The arrows $\Psi_1,\Psi_2,\Psi_3,\Psi_4$  are not $\sim$-equivalent. See Example \ref{example_A_K}.} \label{fig_example_A_K}
\end{figure}

\begin{lemma} \label{lemma_A_K_filtering} The functor $A_K: \textbf{VK} \to \textbf{Sets}$ is filtering and continuous.
\end{lemma}
\begin{proof} Property $(i)$ is satisfied because $A_K(W') \ne \varnothing$ for all $W' \in \mathscr{O}_K$. Now suppose there are parallel arrows $\Phi_1,\Phi_2:U \to V$ such that $A(\Phi_1)([\Psi])=A(\Phi_2)([\Psi])$ for some $\Psi:W' \to U$ with $W' \in \mathscr{O}_K$. Then $[\Phi_1 \circ \Psi]=[\Phi_2 \circ \Psi]$ and hence there is $(O,K)$-neighborhood $O'$ such that $\Phi_1 \circ \Psi$ and $\Phi_2 \circ \Psi$ have the same restriction to $O'$. Let $I:O' \to W'$ denote this inclusion. Then $\Phi_1 \circ \Psi \circ I=\Phi_2 \circ \Psi \circ I$, so that $\Psi \circ I$ equalizes $\Phi_1,\Phi_2$. Denote by $\Lambda:O' \to O'$ the identity map. Then $A_K(\Psi\circ I)([\Lambda])=[\Psi \circ I \circ \Lambda]=[\Psi]$, so that $[\Psi]$ is in the image of $A_K(\Psi \circ I)$.  This proves $(ii)$.
\newline
\newline
For $(iii)$, suppose that $A_K(U_1),A_K(U_2)$ are nonempty and let $[\Psi_1] \in A_K(U_1)$, $[\Psi_2] \in A_K(U_2)$ for $\Psi_1:O_1' \to U_1$, $\Psi_2:O_2' \to U_2 $ where $O_1',O_2' \in \mathscr{O}_K$. Let $O_1' \times_O O_2' \in \mathscr{O}_K$ be the pullback of the inclusions  $O_1' \to O$, $O_2' \to O$. Let $I_1: O_1' \times_O O_2' \to O_1'$, $I_2: O_1' \times_O O_2' \to O_2'$ denote the inclusion maps and $\Lambda:O_1' \times_O O_2' \to O_1' \times_O O_2'$ the identity. Then $[\Lambda] \in A_K(O_1' \times_O O_2')$ and we have a diagram:
\[
\xymatrix{U_1 & \ar[l]_-{\Psi_1 \circ I_1} O_1' \times_o O_2' \ar[r]^-{\Psi_2 \circ I_2} & U_2},
\]
such that $A_K(\Psi_1 \circ I_1)([\Lambda])=[\Psi_1]$ and $A_K(\Psi_2 \circ I_2)([\Lambda])=[\Psi_2]$. This proves $(iii)$.
\newline
\newline 
For the continuity property, suppose that $S=\{\Psi_i=\psi_i \times \text{id}:U_i \to U\}$ is a covering sieve for $U$ where $U_i \subseteq \mathbb{K}(\Sigma_i)$. To prove that this is mapped to an epimorphic family, it must be shown that every $[\Phi] \in A_K(U)$ is in the image of some $A_K(\Psi_i)$. Write $\Phi=\phi \times \text{id}:O' \to U$ for $O' \in \mathscr{O}_K$, $O' \subseteq \mathbb{K}(\Sigma')$. Since $S$ is a covering sieve, there is a $\Psi_i\in S$ such that $\Phi(K) \in \Psi_i(U_i)$. Then we may choose a supporting surface $\Sigma'' \subseteq \Sigma$ of $K$ and an $(O,K)$-neighborhood $O'' \subseteq \mathbb{K}(\Sigma'')$ such that $\phi(\Sigma'') \subseteq \psi_i(\Sigma_i)$ and $\Phi(O'') \subseteq \Psi_i(U_i)$. Let $\psi_i^{-1}:\psi_i(\Sigma_i) \to \Sigma_i$ be inverse to the diffeomorphism $\psi_i: \Sigma_i \to \psi_i(\Sigma_i)$. Set $\Upsilon_i=(\psi_i^{-1} \circ \phi|_{\Sigma''}) \times \text{id} : O'' \to U_i$. Then $A(\Psi_i)([\Upsilon_i])=[\phi|_{\Sigma''} \times \text{id}]=[\Phi]$.
\end{proof}

\begin{lemma} \label{lemma_A_K_determines_V_K} The continuous filtering functor $A_K$ determines the variable-space knot $\textbf{V}_K$.
\end{lemma}

\begin{proof} Let $\textbf{V}$ be the variable-space knot determined by $A_K$. Let $\Lambda:O \to O$ be the identity map. We will show that $K$ is the knot determined by $[\Lambda] \in A_K(O)$. Suppose is in $\Psi:V \to O$ in $S_{[\Lambda]}$. Then $A_K(\Psi)([\Phi])=[\Lambda]$ for some $[\Phi] \in A_K(V)$, where $\Phi:W' \to V$ and $W'$ is an $(O,K)$-neighborhood. It must be proved that $K \in \Psi(V)$. Since $A_K(\Psi)([\Phi])=[\Lambda]$, we have $[\Psi \circ \Phi]=[\Lambda]$. Then there is an $(O,K)$-neighborhood $O'$ and inclusions $I_1:O' \to W'$, $I_2:O' \to O$ such that $\Psi \circ \Phi \circ I_1=\Lambda \circ I_2$. Then since $K \in I_2(O')$, we have $K \in \Psi(V)$. It follows that $\{K\}=\bigcap_{\Psi \in S_{[\Lambda]}} \Psi(W')$.
\end{proof}

Lemmas \ref{lemma_A_K_filtering} and \ref{lemma_A_K_determines_V_K} combine to prove Lemma \ref{lemma_inverse_determines} (1). The following result proves Lemma \ref{lemma_inverse_determines} (2). It has been restated here in a more explicit form that will be needed later.

\begin{lemma} \label{lemma_same_knots_imp_nat_iso} Let $A:\textbf{VK} \to \textbf{Sets}$ be a continuous filtering functor determining the variable-space knot $\textbf{V}_K$. If $o \in A(O)$ determines $K$, there is a natural isomorphism $\alpha^o:A \to A_K$ satisfying:
\begin{enumerate}
\item $\alpha^o_O(o)=[\Lambda]$, where $\Lambda:O \to O$ is the identity map, and
\item for any $u \in A(U)$,  $\alpha^o_U(u) \in A_K(U)$ determines the same knot $J$ in $U$ as $u\in A(U)$.
\end{enumerate}
\end{lemma}

\begin{proof} For each object $U$ of $\textbf{VK}$, define $\alpha_U^o:A(U) \to A_K(U)$ as follows. First note that since $A$ and $A_K$ both determine $\textbf{V}_K$, $A(U) \ne \varnothing$ if and only if $A_K(U) \ne \varnothing$. Let $u \in A(U)$ and let $J$ be the knot in $U$ determined by $u$. Since $A$ is filtered, there is an object $W$, a diagram $\xymatrix{U & \ar[l]_-{\Psi} W \ar[r]^-{\Phi} & O}$, and a $w \in A(W)$ such that $\Psi \cdot w=u$ and $\Phi \cdot w=o$. By Lemma \ref{lemma_factors}, $\Phi$ factors as $\xymatrix{W \ar[r]^-{\Phi'} & V' \ar[r]^-{I} & O}$ for some isomorphism $\Phi'$ and inclusion $I$. Let $L \in W$ be the knot determined by $w \in A(W)$. By Lemma \ref{lemma_determines_2}, $\Phi(L)=K$ and $\Psi(L)=J$. Since $I(K)=K$, we have $\Phi'(L)=K$ and hence $V'$ is an $(O,K)$-neighborhood. Define $\Psi_u=\Psi \circ (\Phi')^{-1}:V' \to U$ and set $v'=\Phi'\cdot w$. By Lemma \ref{lemma_determines_3}, $K \in V'$ is the knot determined by $v' \in A(V')$. Also, we have that $\Psi_u \cdot v'=\Psi\cdot w=u$. Hence, by Lemma \ref{lemma_determines_2}, $\Psi_u(K)=J$. Define $\alpha_U^o(u)=[\Psi_u]\in A_K(U)$.
\newline
\newline
It must be proved that $\alpha_U^o(u)$ is well defined. Suppose that for $i=1,2$, there is an object $W_i$, a diagram $\xymatrix{U & \ar[l]_-{\Psi_i} W_i \ar[r]^-{\Phi_i} & O}$, and $w_i \in A(W_i)$ such that $\Psi_i \cdot w_i=u$ and $\Phi_i \cdot w_i=o$. Since $A$ is filtered, $W_1,W_2$ are connected by a diagram $\xymatrix{W_1 & \ar[l]_-{\Upsilon_1} T \ar[r]^-{\Upsilon_2} & W_2}$ such that for some $t \in A(T)$ $\Upsilon_i \cdot t= w_i$ for $i=1,2$. Furthermore, since $\Psi_i \cdot \Upsilon_i \cdot t=u$ for $i=1,2$, the parallel arrows $\Psi_1 \circ \Upsilon _1$ and $\Psi_2 \circ \Upsilon_2$ have an equalizer $\Upsilon:\widebar{T} \to T$ such that $\Upsilon \cdot \bar{t}=t$ for some $\bar{t} \in A(\widebar{T})$. Likewise, the equalizer property for the parallel arrows $\Phi_1 \circ \Upsilon_1 \circ \Upsilon$ and $\Phi_2 \circ \Upsilon_2 \circ \Upsilon$ gives a map $\widebar{\Upsilon}:\wbar{\wbar{T}} \to \wbar{T}$ and a $\bar{\bar{t}} \in A(\wbar{\wbar{T}})$ such that $\widebar{\Upsilon}\cdot \bar{\bar{t}}=\bar{t}$. Set $\Delta_i=\Upsilon_i \circ \Upsilon\circ \widebar{\Upsilon}$ and as before write $\Phi_i=I_i \circ \Phi_i'$. This gives the following commutative diagram, where $V_1' \times_O V_2'$ is a pullback.
\[
\xymatrix{ & \ar[dl]_-{\Psi_1} W_1 \ar[r]^-{\Phi_1'} & V_1'  \ar[dr]^{I_1} & \\ 
          U & \wbar{\wbar{T}} \ar[ur]|-{\Phi_1' \circ \Delta_1} \ar[dr]|-{\Phi_2' \circ \Delta_2} \ar[u]^{\Delta_1} \ar[d]_{\Delta_2} \ar@{-->}[r] & V_1' \times_O V_2' \ar[u] \ar[d] & O  \\
            & \ar[ul]^-{\Psi_2} W_2 \ar[r]_-{\Phi_2'} & V_2' \ar[ur]_{I_2} &
}
\] 
Then there is an arrow $\Delta:\wbar{\wbar{T}} \to V_1' \times_O V_2'$ making the right diamond commute. Tracing the diagram, we that $(\Psi_1)_u$ and $(\Psi_2)_u$ agree on the $(O,K)$-neighborhood $\Delta(\wbar{\wbar{T}})$. Hence, $\alpha_U^o$ is well-defined. 
\newline
\newline
To see that $\alpha_U^o$ is injective, suppose $\alpha_U^o(u_1)=\alpha_U^o(u_2)$ for some $u_1,u_2 \in A(U)$. Since $A$ is filtered, there is a diagram $\xymatrix{U & \ar[l]_-{\Phi_{1}} W \ar[r]^-{\Phi_{2}} & U}$ and a $w \in A(W)$ such that $\Phi_{1} \cdot w=u_1$ and $\Phi_{2}\cdot w=u_2$. Then $\alpha_W^o(w)=[\Psi_w]$ for some arrow $\Psi_w:W' \to W$ with $W' \in \mathscr{O}_K$ and there is a $w' \in A(W')$ such that $\Psi_w \cdot w'=w$. Since $\alpha_U^o$ is well-defined, we may choose $\Psi_{u_1}$ to be $\Phi_{1} \circ \Psi_{w}$ and $\Psi_{u_2}$ to be $\Phi_{2} \circ \Psi_{w}$. Then we have $[\Psi_{u_1}]=\alpha_U^o(u_1)=\alpha_U^o(u_2)=[\Psi_{u_2}]$. Thus, there is some $Y' \in \mathscr{O}_K$ with inclusion map $I:Y' \to W'$ such that $\Psi_{u_1} \circ I=\Psi_{u_2} \circ I$. Since $Y' \in \mathscr{O}_K$, Lemma \ref{lemma_determines_1} implies there is a $y' \in A(Y')$ such that $I \cdot y'=w'$. Then $u_1=\Psi_{u_1} \cdot I \cdot y'=\Psi_{u_2} \cdot I \cdot y'=u_2$ and it follows that $\alpha_U^o$ is injective.
\newline
\newline
Next, it will be shown that $\alpha_U^o$ is surjective. Suppose $[\Psi] \in A_K(U)$ with $\Psi:W' \to U$ for some $W' \in \mathscr{O}_K$. Consider the diagram $\xymatrix{U & \ar[l]_-{\Psi} W' \ar[r]^-{I} & O}$, where $I$ is inclusion. Since $o \in A(O)$ determines $K$ and $I(K)=K$, Lemma \ref{lemma_determines_1} implies that there is a $w' \in W'$ such that $I \cdot w'=o$. Let $u=\Psi \cdot w'$. Then we may take $\Psi_u=\Psi$, so that $\alpha_U(u)=[\Psi]$. Hence, $\alpha_U^o$ is surjective. 
\newline
\newline
To prove that $\alpha^o$ is natural, let $\Upsilon:U \to W$ be an arrow. If $u \in A(U)$, and $\Psi_u:V' \to U$ is as above, then $\Upsilon \cdot u$ determines the knot $\Upsilon(\Psi_u(K))$ (see Lemma \ref{lemma_determines_3}). This implies that $\alpha_W^o$ sends $\Upsilon \cdot u$ to $[\Upsilon \circ \Psi_u]$. Hence, $\alpha_W^o(\Upsilon \cdot u)=[\Upsilon \circ \Psi_u]=A_K(\Upsilon)(\alpha_U^o(u))$.  
\newline
\newline
Lastly we prove claims (1) and (2). For (1), consider the diagram $\xymatrix{O & \ar[l]_-{\Lambda} O \ar[r]^-{\Lambda} & O}$. Since $O$ is an $(O,K)$ neighborhood, and $\Lambda \cdot o=o$, it follows that $\alpha_O^o(o)=[\Lambda]$. For claim (2), if $\alpha_U^o(u)=[\Psi_u]$, then $\Psi_u(J)=K$. The proof that $\{J\}=\bigcap_{\Upsilon \in S_{[\Psi_u]}} \Upsilon(\text{dom}(\Upsilon))$ then follows as in Lemma \ref{lemma_A_K_determines_V_K}. 
\end{proof}

\section{Virtual isotopy} \label{sec_virtual_isotopy}

\subsection{Variable-space isotopy} \label{sec_virtual_isotopy_var} Having defined variable-space knots in Section \ref{sec_var_ambient}, we now define isotopies of variable-space knots.  These will be shown in Section \ref{sec_virtual_isotopy_paths} to be in one-to-one correspondence with virtual isotopies (Definition \ref{defn_virtual_isotopy_paths}). As before, we first analyze the knot space $\mathbb{K}$.
\newline
\newline
Denote the set of connected components of a small category $\textbf{C}$ by $\pi_0(\textbf{C})$. Proposition \ref{prop_connected_comp} states that $\pi_0(\smallint_{\textbf{K}} F) \cong \mathbb{K}$ in $\textbf{Sets}$. Let $\sigma^{\mathbb{I}}_{\mathbb{K}}:\mathbb{I} \to \mathbb{K}$ be a path in $\mathbb{K}$ and let $\textbf{K}_t$ be the connected component of $\smallint_{\textbf{K}} F$ containing the object $(\mathbb{K},\sigma^{\mathbb{I}}_{\mathbb{K}}(t))$. Define a function $\sigma:\mathbb{I} \to \pi_0(\smallint_{\textbf{K}} F)$ by $\sigma(t)=\textbf{K}_t$. Since $\sigma^{\mathbb{I}}_{\mathbb{K}}$ is continuous, $\sigma$ satisfies the following local continuity property (with $\textbf{C}=\textbf{K}$ and $\sigma^N_U=\sigma^{\mathbb{I}}_{\mathbb{K}}|_N$):
\begin{center}$(\star)$ \,\,\, 
\begin{minipage}{6in}
For all $U \in \text{obj}(\textbf{C})$ and $t\in \mathbb{I}$, if $(U,K) \in \text{obj}(\sigma(t))$, there is a neighborhood $N \subseteq \mathbb{I}$ of $t$ and a continuous function $\sigma^N_U:N \to U$ s.t. $\sigma^N_U(t)=K$ and $(U,\sigma^N_U(s)) \in \text{obj}(\sigma(s))$ for all $s\in N$.
\end{minipage}
\end{center}
Conversely, suppose $\sigma:\mathbb{I} \to \pi_0(\smallint_{\textbf{K}} F)$ satisfies $(\star)$. Define $\sigma^{\mathbb{I}}_{\mathbb{K}}: \mathbb{I} \to \mathbb{K}$ by $\sigma^{\mathbb{I}}_{\mathbb{K}}(t)=K_t$, where $(\mathbb{K},K_t) \in \text{obj}(\sigma(t))$. Then $\sigma^{\mathbb{I}}_{\mathbb{K}}$ is a well-defined path in $\mathbb{K}$. Hence, $\sigma^{\mathbb{I}}_{\mathbb{K}}(0),\sigma^{\mathbb{I}}_{\mathbb{K}}(1)$ are isotopic in $\mathbb{R}^3$.

\begin{definition}[Variable-space isotopy] \label{defn_virtual_isotopy_var} Let $\textbf{V}_0, \textbf{V}_1 \in \pi_0(\smallint_{\textbf{VK}} F)$ be variable-space knots. A \emph{variable-space isotopy} from $\textbf{V}_0$ to  $\textbf{V}_1$ is a function $\sigma:\mathbb{I} \to \pi_0(\smallint_{\textbf{VK}} F)$ satisfying property $(\star)$ (with $\textbf{C}=\textbf{VK}$) such that $\sigma(0)=\textbf{V}_0$ and $\sigma(1)=\textbf{V}_1$. 
\end{definition}

The following technical lemma will be used to relate variable-space isotopy to both stable Reidemeister equivalence and virtual isotopy. It describes a variable-space isotopy as a finite sequence of paths in knot spaces $\mathbb{K}(\Sigma)$ that are patched together by arrows in $\textbf{VK}$.  

\begin{lemma} \label{lemma_cont_implies_isotop_var_amb} Let $\sigma:\mathbb{I} \to \pi_0(\smallint_{\textbf{VK}} F)$ be a virtual isotopy from $\textbf{V}_0$ to $\textbf{V}_1$. Then for some $n \ge 1$, there is a partition $l_1=0<l_2<r_1<l_3<r_2< \ldots <l_n <r_{n-1} <r_n=1$ of $\mathbb{I}$, a diagram:
\[
\xymatrix{ & \ar[dl]_-{\Delta_1} C_{1.5} \ar[dr]^-{\Upsilon_2} &  & \ar[dl]_-{\Delta_2} C_{2.5} \ar[dr]^-{\Upsilon_3} & & \cdots & & \ar[dl]_-{\Delta_{n-1}} C_{n-.5} \ar[dr]^-{\Upsilon_n} &  \\
          C_1 & & C_2 & & C_3 & \cdots & C_{n-1} & & C_n,                     
} 
\]
in $\textbf{VK}$ of path components $C_i \subseteq \mathbb{K}(\Sigma_i)$, and continuous functions: 
\[
\sigma_1:\underbrace{[l_1,r_1)}_{N_1} \to C_1,\,\, \sigma_{1.5}:\underbrace{(l_2,r_1)}_{N_{1.5}} \to C_{1.5},\,\, \sigma_2:\underbrace{(l_2,r_2)}_{N_2} \to C_2, \ldots, \sigma_{n}:\underbrace{(l_n,r_n]}_{N_n} \to C_{1.5} 
\]
such that the following properties hold:
\begin{enumerate}
\item $\Delta_i \circ \sigma_{i+.5}=\sigma_i|(l_{i+1},r_i)$ and $\Upsilon_{i+1}\circ \sigma_{i+.5}=\sigma_{i+1}|(l_{i+1},r_i)$ for all $i \in \{1,2,\ldots, n-1\}$, and
\item $(C_j,\sigma_j(t)) \in \text{obj}(\sigma(t))$ for all $j \in \{1,1.5,2,\ldots,n\}$ and $t \in N_j$.
\end{enumerate}
\end{lemma}

\begin{proof} For each $t \in \mathbb{I}$, choose some $(C_t,K_t) \in \text{obj}(\sigma(t))$ where $C_t \subseteq \mathbb{K}(\Sigma_t)$ is a path component. There is a neighborhood $M_t$ of $t$ and a continuous function $\rho_t:M_t \to C_t$ such that $(C_t,\rho_t(s)) \in \text{obj}(\sigma(s))$ for all $s \in M_t$. It may be assumed that each $M_t$ is an interval. The set of these $M_t$ form an open cover of $\mathbb{I}$ and hence there is a finite subcover. It also may be assumed that the finite subcover $\{M_1=M_{t_1},\ldots,M_n=M_{t_n}\}$ that has the fewest number of elements. After relabeling, this gives a partition $a_1=0<a_2<b_1<a_3<b_2< \ldots <a_n <b_{n-1} <b_n=1$ with $M_{1}=[a_1,b_1)$, $M_{2}=(a_2,b_2)$, $\ldots$, $M_{n}=(a_n,b_n]$. For $1 \le i \le n-1$, choose $t_{i+.5} \in M_{i} \cap M_{i+1}$ and set $t_{.5}=0,t_{n+.5}=1$. Set $\rho_i=\rho_{t_i}$ and $C_i=C_{t_i}$. We have $\rho_{i}|[t_{i-.5},t_{i+.5}]$ is a path between $\rho_{i}(t_{i-.5})$ and $\rho_{i}(t_{i+.5})$. Moreover, $(C_i,\rho_{i}(t)),(C_{i+1},\rho_{i+1}(t))$ are in the same connected component of $\smallint_{\textbf{VK}} F$ for all $t \in M_i \cap M_{i+1}$. Let $\textbf{V}_{t_{i+.5}}$ denote this connected component for $t=t_{i+.5}$.
\newline
\newline
Since $(C_i,\rho_{i}(t_{i+.5})),(C_{i+1},\rho_{i+1}(t_{i+.5})) \in \text{obj}(\textbf{V}_{t_{i+.5}})$, they are connected by a path of arrows. Repeatedly pulling back along each subdiagram of this path having the form $\rightarrow \bullet \leftarrow$, it follows that there is a path of arrows $\xymatrix{(C_i,\rho_i(t_{i+.5})) & \ar[l]_-{\Delta_{i}} (C_{i+.5},K_{i+.5}) \ar[r]^-{\Upsilon_{i+1}} &  (C_{i+1},\rho_{i+1}(t_{i+.5}))}$ with $C_{i+.5}$ a path component. Now, since $\Delta_i,\Upsilon_{i+1}$ are open, there are neighborhoods $J_i,\bar{J}_i$ of $t_{i+.5}$ such that $\rho_{i}(J_i) \subseteq \Delta_i(C_{i+.5})$ and $\rho_{i+1}(\bar{J}_i) \subseteq \Upsilon_{i+1}(C_{i+.5})$. For each $1\le i \le n-1$, choose $l_{i+1},r_i$ so that $ t_{i+.5} \in (l_{i+1},r_i) \subseteq J_i \cap \bar{J}_{i+1}$, and set $l_1=0,r_n=1$. Then since $\Delta_i,\Upsilon_{i+1}$ are injective, $\rho_i|(l_{i+1},r_i)$ and $\rho_{i+1}|(l_{i+1},r_i)$ factor uniquely through $\Delta_i$ and $\Upsilon_{i+1}$, respectively. Hence, $\sigma_{i+.5}:(l_{i+1},r_i) \to C_{i+.5}$ can be defined piecewise so that $\rho_i(t)=\Delta_i \circ \sigma_{i+.5}(t)$ for $l_{i+1}<t \le t_{i+.5}$ and $\rho_{i+1}(t)=\Upsilon_{i+1} \circ \sigma_{i+.5}(t)$ for $t_{i+.5} \le t < r_i$. For $i \in \{2,3,\ldots,n-1\}$, define $\sigma_i:(l_i,r_i)\to C_i$ by $\sigma_i|(l_i,r_{i-1})=\Upsilon_i \circ \sigma_{i-.5}$, $\sigma_i|[r_{i-1},l_{i+1}]=\rho_i|[r_{i-1},l_{i+1}]$, and $\sigma_i|(l_{i+1},r_i)=\Delta_i \circ \sigma_{i+.5}$. Define $\sigma_1:[l_1,r_1) \to C_1$ by $\sigma_1|[l_1,l_2]=\rho_1|[l_1,l_2]$ and $\sigma_1|(l_2,r_1)=\Delta_1 \circ \sigma_{1.5}$. Finally, define $\sigma_n:(l_n,r_n]\to C_n$ by $\sigma_n|(l_n,r_{n-1})=\Upsilon_n \circ \sigma_{n-.5}$ and $\sigma_n|[r_{n-1},r_n]=\rho_n|[r_{n-1},r_n]$. Then $\sigma_1,\ldots,\sigma_n$ satisfy $(1)$ and $(2)$. 
\end{proof}

With this lemma, we can now prove that the equivalence relation generated by variable-space isotopy yields the same equivalence classes as the diagrammatic models of virtual knot theory.

\begin{theorem} \label{lemma_v_knot_type_isotopy} Let $(\Sigma_0,D_0),(\Sigma_1,D_1)\in \mathfrak{D}^{(CKS^*)}$ be knot diagrams and let $K_0 \in \mathbb{K}(\Sigma_0)$, $K_1 \in \mathbb{K}(\Sigma_1)$ be knots with diagrams $D_0,D_1$. Set $\textbf{V}_0=\textbf{V}_{K_0},\textbf{V}_1=\textbf{V}_{K_1} \in \pi_0(\smallint_{\textbf{VK}} F)$. Then $(\Sigma_0,D_0),(\Sigma_1,D_0)$ have the same virtual knot type if and only if $\textbf{V}_0,\textbf{V}_1$ are isotopic variable-space knots.
\end{theorem}

\begin{proof} If $\textbf{V}_0$, $\textbf{V}_1$ variable-space isotopic, then $(\Sigma_0,D_0),(\Sigma_1,D_1)$ represent the same virtual knot type by Lemmas \ref{lemma_cont_implies_isotop_var_amb} and \ref{prop_no_more_diagrams}. Conversely, suppose that $((\Sigma_0,D_0),(\Sigma_1,D_1))\in \mathfrak{R}^{(CKS^*)}$, so that the two diagrams have the same virtual knot type. By Lemma \ref{prop_no_more_diagrams}, there are path components $C_1$, $C_{1.5}$, $C_2$ , $\ldots$, $C_{n-.5}$, $C_n$ and arrows $\Upsilon_i:C_i \to C_{i+.5}$, $\Delta_{i+1}:C_{i+1} \to C_{i+.5}$ for $i \in \{1,\ldots,n-1\}$ such that $K_0 \in C_1 \subseteq \mathbb{K}(\Sigma_0)$ and $K_1 \in C_n \subseteq \mathbb{K}(\Sigma_1)$. For each $i \in \{2,3,\ldots,n-1\}$, choose any knot $K_{C_i}\in C_i$ and set $K_{C_1}=K_0,K_{C_n}=K_1$. For $i \in \{1,2,\ldots,n-1\}$, let $\sigma_{i+.5}:[\tfrac{i-1}{n-1},\tfrac{i}{n-1}] \to C_{i+.5}$ be a path from $\Upsilon_{i}(K_{C_i})$ to $\Delta_{i+1}(K_{C_{i+1}})$. For $t \in [\tfrac{i-1}{n-1},\tfrac{i}{n-1}]$, let $\textbf{V}_t \in \pi_0(\smallint_{\textbf{VK}} F)$ be the variable-space knot containing $(C_{i+.5},\sigma_{i+.5}(t))$. This is well-defined for all $t\in \mathbb{I}$ since $(C_{i-.5},\Delta_{i}(K_{C_i}))$ and $(C_{i+.5},\Upsilon_i(K_{C_i}))$ are in the same connected component of $\smallint_{\textbf{VK}} F$. Define $\sigma:\mathbb{I} \to \pi_0(\smallint_{\textbf{VK}} F)$ by $\sigma(t)=\textbf{V}_t$.
\newline
\newline
It remains to show that $\sigma$ satisfies property $(\star)$. Suppose that $t \in \mathbb{I}$ and $(U,L) \in \text{obj}(\sigma(t))$. First suppose there is an $i \in \{1,\ldots,n\}$ such that $t \in (\tfrac{i-1}{n-1},\tfrac{i}{n-1})$. Then $(U,L)$ and $(C_{i+.5},\sigma_{i+.5}(t))$ are in the same connected component of $\smallint_{\textbf{VK}}F$. This means that in $\textbf{VK}$ there is a diagram :
\[
\xymatrix{ & & \ar[dl]_-{\Phi_1} U_{1.5} \ar[dr]^-{\Psi_2} &  & \ar[dl]_-{\Phi_2} U_{2.5} \ar[dr]^-{\Psi_3} & & \cdots & & \ar[dl]_-{\Phi_{m-1}} U_{m-.5} \ar[dr]^-{\Psi_m} &  &\\
          C_{i+.5} \ar@{=}[r]& U_1 & & U_2 & & & \cdots & & & U_m \ar@{=}[r] & U.                     
} 
\]
Since $\sigma_{i+.5}$ is continuous, there is an open set $M \subseteq (\tfrac{i-1}{n-1},\tfrac{i}{n-1})$ such that $t \in M$ and $\sigma_{i+.5}(M)\subseteq \Phi_1(U_{1.5})$. Then $\sigma_{i+.5}$ factors through $\Phi_1$ to give a continuous function $\rho_{1.5}:M \to U_{1.5}$ such that $\Phi_1 \circ \rho_{1.5}=\sigma_{i+.5}|M$. Thus, $(U_{1.5},\rho_{1.5}(s)) \in \text{obj}(\sigma(s))$ for all $s \in M$. Also $\rho_2:=\Psi_2 \circ \rho_{1.5}:M \to U_2$ satisfies $(U_2,\rho_2(s)) \in \text{obj}(\sigma(s))$ for all $s \in M$. Continuing in this manner, we eventually obtain a continuous function $\sigma^N_U=\rho_m:N \to U_m=U$ such that $t \in N$ and $(U,\sigma^N_U(s)) \in \text{obj}(\sigma(s))$ for all $s \in N$.  
\newline
\newline
The same argument works for $t=0$ and $t=1$, so suppose $t=\tfrac{i}{n-1}$ for some $i \in \{1,\ldots,n-2\}$. Since $\sigma_{i+.5},\sigma_{i+1.5}$ are continuous, there is an $r_i<\tfrac{i+1}{n-1}$ and an $l_i> \tfrac{i-1}{n-1}$ such that $\sigma_{i+1.5}([t,r_i))\subseteq \Upsilon_{i+1}(C_{i+1})$ and $\sigma_{i+.5}((l_i,t]) \subseteq \Delta_{i+1}(C_{i+1})$. Then define $\sigma_{i+1}:(l_i,r_i) \to C_{i+1}$ piecewise so that $\sigma_{i+1.5}|[t,r_i)=\Upsilon_{i+1} \circ \sigma_{i+1}|[t,r_i)$ and $\sigma_{i+.5}|(l_i,t]=\Delta_{i+1} \circ \sigma_{i+1}|(l_i,t]$. This is continuous since $\sigma_{i+1.5}(t)=\Upsilon_{i+1}(K_{C_{i+1}})$, $\sigma_{i+.5}(t)=\Delta_{i+1}(K_{C_{i+1}})$, so that $\sigma_{i+1}(t)=K_{C_{i+1}}$. Then $(U,L)$ and $(C_{i+1},K_{C_{i+1}})$ are in the same connected component of $\smallint_{\textbf{VK}} F$. Thus, we may apply the same argument as in the preceding paragraph, using $(C_{i+1},K_{i+1})$ in place of $(C_{i+.5},\sigma_{i+.5}(t))$ and $\sigma_{i+1}$ in place of $\sigma_{i+.5}$. This implies that $\sigma$ satisfies property $(\star)$ and we conclude that $\textbf{V}_0$ and $\textbf{V}_1$ are variable-space isotopic.
\end{proof}

\subsection{Virtual isotopy $\&$ variable-space isotopy} \label{sec_virtual_isotopy_paths} Here we prove both the one-to-one correspondence between variable-space isotopies and virtual isotopies (Theorem \ref{thm_paths_sh_vk}) and the one-to-one correspondence between virtual knot types and virtual isotopy equivalence classes (Theorem \ref{thm_recovers}). The proof of Theorem \ref{thm_paths_sh_vk} follows the strategy of Section \ref{sec_points}. The category of geometric morphisms $\text{Sh}(\textbf{I}) \to \text{Sh}(\mathbf{VK})$ is equivalent to the category of continuous filtering functors $A:\textbf{VK} \to \text{Sh}(\textbf{I})$ (see \cite{mac_moer}, Chapter VII, Section 9). Each such $A$ corresponds to a variable-space isotopy $\sigma_A: \mathbb{I} \to \pi_0(\smallint_{\textbf{VK}} F)$ and any variable-space isotopy $\sigma: \mathbb{I} \to \pi_0(\smallint_{\textbf{VK}} F)$ corresponds a continuous filtering functor $A_{\sigma}:\textbf{VK} \to \text{Sh}(\textbf{I})$. Furthermore, $\sigma_{A_{\sigma}}=\sigma$ and $A_{\sigma_A} \cong A$.
\newline
\newline
Let $A: \textbf{VK} \to \text{Sh}(\textbf{I})$ be a continuous filtering functor. The reader is referred to \cite{mac_moer}, Theorem VII.10.1, for the list of equivalent conditions for continuous filtering functors that will be used below. Define $\sigma_A: \mathbb{I} \to \pi_0(\smallint_{\textbf{VK}} F)$ as follows. For each $t \in \mathbb{I}$, the continuous function $\tau_t:\mathbbm{1} \to \mathbb{I}$ given by $\tau_t(0)=t$ corresponds to a geometric morphism $t=\mathscr{T}(\tau_t):\textbf{Sets} \to \text{Sh}(\textbf{I})$. Composing $A$ with the left adjoint $t^*$ gives a continuous filtering functor $t^* \circ A:\textbf{VK} \to \text{Sh}(\textbf{I}) \to \textbf{Sets}$, as in Section \ref{sec_points}. By Theorem \ref{thm_points_sh_vk}, this determines a variable-space knot $\textbf{V}_t \in \pi_0(\smallint_{\textbf{VK}} F)$. Define $\sigma_{A}(t)=\textbf{V}_t$. 

\begin{lemma} \label{lemma_cff_is_vi} For $A: \textbf{VK} \to \text{Sh}(\textbf{I})$ as above,  $\sigma_A: \mathbb{I} \to \pi_0(\smallint_{\textbf{VK}} F)$ is a variable-space isotopy.
\end{lemma}
\begin{proof} Let $t \in \mathbb{I}$ and $(U,K) \in \text{obj}(\sigma_A(t))$. It must be shown that there is a neighborhood $N$ of $t$ and a continuous function $\sigma^N_U:N \to U$ such that $\sigma^N_U(t)=K$ and $(U,\sigma^N_U(s)) \in \text{obj}(\sigma_A(s))$ for all $s \in N$. Since $(U,K) \in \text{obj}(\sigma_A(t))$, $K$ is the knot determined by some $u_t \in t^* \circ A(U)$. Now, $A(U)$ is a sheaf on $\mathbb{I}$. Every sheaf on a topological space can be identified with its sheaf of germs. Furthermore, for $P \in \text{Sh}(\textbf{I})$, applying the left adjoint $t^*:\text{Sh}(\textbf{I}) \to \textbf{Sets}$ gives the stalk of $P$ at $t$. This implies that there is a neighborhood $N$ of $t$ such that $u_t=\text{germ}_t u$ for some $u \in A(U)(N)$. For any $s \in N$, define $u_s \in s^* \circ A(U)$ by $u_s=\text{germ}_s u$ and let $\sigma^N_U(s)$ be the knot in $U$ determined by $u_s$ (see Definition \ref{defn_determines}). 
\newline
\newline
By construction, $(U,\sigma^N_U(s)) \in \sigma_A(s)$. Now we show that $\sigma^N_U:N \to U$ is continuous. Let $V \subseteq U, V \ne U$ be an open set such that $\sigma^N_U(s) \in V$ for some $s \in N$. We will prove that there is a neighborhood $M \subseteq N$ of $s$ such that $\sigma^N_U(M) \subseteq V$. First take an open cover $\{U_i\}$ of $U$ such that $U_1=V$ is the only open set containing $\sigma^N_U(s)$. Let $S$ be the covering sieve on $U$ generated by $\{U_i\}$. The continuity property of $A$ applied to $u \in A(U)(N)$ means that $N$ is covered by arrows in $\textbf{I}$ of the form $j:M \to N$ such that there is a $\Psi: W \to U$ of $S$ and a $w \in A(W)(M)$ satisfying $A(\Psi)_{M}(w)=A(U)(j)(u)$. Then we may choose some such $j$ ,$\Psi$, and $w$ so that $s\in M$. Taking the stalk over $s$ for the equation $A(\Psi)_{M}(w)=A(U)(j)(u)$ gives $(s^* \circ A)(\Psi)(\text{germ}_s w)=u_s$. Then by Lemma \ref{lemma_determines}, the knot $\sigma^N_U(s)$ determined by $u_s$ must be in $\Psi(W)$. The choice of $S$ then implies that $\Psi$ factors through the inclusion $I:V \subseteq U$. Write $\Psi=I \circ \Phi$ for some $\Phi:W \to V$ and set $v=A(\Phi)_M(w) \in A(V)(M)$. Thus we have that for all $r \in M$, $\text{germ}_r v \in r^* \circ A(V)$. Also note that $A(I)_M(v)=A(U)(j)(u)$. Hence, $(r^* \circ A)(I)(\text{germ}_r v)=u_r$. By Lemma \ref{lemma_determines_2}, the knot in $V$ determined by $\text{germ}_r v$ is mapped by $I$ to the knot in $U$ determined by $u_r$, i.e. $\sigma^N_U(r)$. But $I$ is merely the subset map $V \subseteq U$, so we have $\sigma^N_U(r) \in V$ for all $r \in M$ and it follows that $\sigma^N_U$ is continuous.
\end{proof}

\begin{lemma} \label{lemma_nat_cont_filter_determine_isotopy} If $A_1,A_2:\textbf{VK} \to \text{Sh}(\textbf{I})$ are continuous filtering functors and $\beta:A_1 \to A_2$ is a natural transformation, then $\sigma_{A_1}=\sigma_{A_2}$.
\end{lemma}
\begin{proof} For each object $U$ of $\textbf{VK}$, there is a natural transformation $\beta_U:A_1(U) \to A_2(U)$ of sheaves on $\mathbb{I}$. Then for each $N \subseteq \mathbb{I}$, there is a function $\beta_{U,N}: A_1(U)(N) \to A_2(U)(N)$ that is natural in both $U$ and $N$. These functions will be used to define a natural transformation $\beta^t:t^*\circ A_1 \to t^*\circ A_2$. Recall that every element of $t^*\circ A_1(U)$ can be represented as $\text{germ}_t u$ for some $u \in A_1(U)(N)$. Define $\beta^t_U\left(\text{germ}_t u\right)=\text{germ}_t \beta_{U,N}(u)$. It must be shown that $\beta^t_U$ is natural in $U$. Let $\Psi:U \to V$ be an arrow in $\textbf{VK}$. First note that $A_2(\Psi)_N \circ \beta_{U,N}=\beta_{V,N} \circ A_1(\Psi)_N$ because $\beta$ is natural in $U$. Then:
\begin{align*}
\beta^t_V\left(t^* \circ A_1(\Psi)\left(\text{germ}_t u\right) \right) &= \beta^t_V \left( \text{germ}_t A_1(\Psi)_N(u)\right) \\
&= \text{germ}_t \beta_{V,N} \circ A_1(\Psi)_N(u) \\
&= \text{germ}_t A_2(\Psi)_N \circ \beta_{U,N} (u) \\
&=t^*\circ A_2(\Psi)\left(\text{germ}_t \beta_{U,N}(u)\right) \\
&=(t^* \circ A_2)(\Psi) \circ \beta_U^t\left(\text{germ}_t u \right)
\end{align*}
Thus, $\beta^t$ is a natural transformation. By Lemma \ref{lemma_natural_imp_same_knots}, $t^* \circ A_1$ and $t^* \circ A_2$ determine the same variable-space knot. Hence, $\sigma_{A_1}(t)=\sigma_{A_2}(t)$ for all $t \in \mathbb{I}$ and the proof is complete. 
\end{proof}

\begin{lemma} \label{lemma_meta_isotopy_inverse_determines} Let $\sigma:\mathbb{I} \to \pi_0(\smallint_{\textbf{VK}}F)$ be a variable-space isotopy.
\begin{enumerate}
\item There is a continuous filtering functor $A_{\sigma}:\textbf{VK} \to \text{Sh}(\textbf{I})$ such that $\sigma_{A_{\sigma}}=\sigma$. 
\item If $A: \textbf{VK} \to \text{Sh}(\textbf{I})$ is a continuous filtering functor and $\sigma_A=\sigma$, then $A \cong A_{\sigma}$.
\end{enumerate}
\end{lemma}

\begin{proof} The proof is delayed until Section \ref{sec_proof_meta_isotopy_inverse_determines} below.
\end{proof}
 
\begin{theorem} \label{thm_paths_sh_vk} Variable-space and virtual isotopies are in one-to-one correspondence.
\end{theorem}

\begin{proof} Since $\underline{\text{Hom}}(\text{Sh}(\textbf{I}),\text{Sh}(\textbf{VK}))$ is equivalent to the category of continuous filtering functors $A:\textbf{VK} \to \text{Sh}(\textbf{I})$ (\cite{mac_moer}, Corollary VII.10.2), it is enough to prove that the theorem holds for continuous filtering functors. By Lemma \ref{lemma_nat_cont_filter_determine_isotopy}, isomorphic continuous filtering functors determine the same variable space isotopy. Lemma \ref{lemma_meta_isotopy_inverse_determines}$(1)$ implies the correspondence is surjective. It is injective by Lemma \ref{lemma_meta_isotopy_inverse_determines}$(2)$: if $A,B:\textbf{VK} \to \text{Sh}(\textbf{I})$ satisfy $\sigma_A=\sigma_B$, then $A \cong A_{\sigma_A}=A_{\sigma_B}\cong B$. \end{proof}

Every knot diagram on a surface $\Sigma$ is the projection of the image of some knot $K \in \mathbb{K}(\Sigma)$. By Theorem \ref{thm_points_sh_vk}, there is a virtual knot $\text{pt}(K):\textbf{Sets} \to \text{Sh}(\textbf{VK})$ corresponding to the variable-space knot $\textbf{V}_K$. In contrast to the minimal genus model, all knot diagrams on surfaces are thus represented in our geometric model $VG$. Combining Theorems \ref{lemma_v_knot_type_isotopy} and \ref{thm_paths_sh_vk}, we also have that the equivalence classes of the relation on virtual knots generated by virtual isotopy are in one-to-one correspondence with virtual knot types. In this sense, the geometric model $VG$ of virtual knot theory fully recovers any of the diagrammatic models. We record these observations below for future reference.

\begin{theorem} \label{thm_recovers} Let $(\Sigma_0,D_0),(\Sigma_1,D_1)\in \mathfrak{D}^{(CKS^*)}$ be knot diagrams. Let $K_0 \in \mathbb{K}(\Sigma_0)$, $K_1 \in \mathbb{K}(\Sigma_1)$ be knots with diagrams $D_0,D_1$ and let $\text{pt}(K_0),\text{pt}(K_1):\textbf{Sets} \to \text{Sh}(\textbf{VK})$ denote the virtual knots corresponding to the variable-space knots $\textbf{V}_{K_0},\textbf{V}_{K_1}$. Then $(\Sigma_0,D_0),(\Sigma_1,D_1)$ have the same virtual knot type if and only if there is a virtual isotopy $\sigma:\text{Sh}(\textbf{I}) \to \text{Sh}(\textbf{VK})$ from $\text{pt}(K_0)$ to $\text{pt}(K_1)$.
\end{theorem}

\subsection{Proof of Lemma \ref{lemma_meta_isotopy_inverse_determines}} \label{sec_proof_meta_isotopy_inverse_determines} Suppose that $\sigma:\mathbb{I} \to \pi_0(\smallint_{\textbf{VK}} F)$ is a variable-space isotopy. By the proof of Lemma \ref{lemma_cff_is_vi}, $A_{\sigma}: \mathbf{VK} \to \text{Sh}(\textbf{I})$ should be defined so that, for all $t \in \mathbb{I}$, $t^* \circ A_{\sigma} \cong A_{K_t}$, where $(O,K_t) \in \text{obj}(\sigma(t))$ and $A_{K_t}:\textbf{VK} \to \textbf{Sets}$ is the continuous filtering functor from Section \ref{sec_proof_inverse_determines}. Thus, for each object $U$, we will define an \'{e}tale space $p^U:\Lambda^U \to \mathbb{I}$ such that the fiber over $t$ is in one-to-one correspondence with the set $A_{K_t}(U)$. Then we set $A_{\sigma}(U)$ to be the sheaf of sections of $p^U$. The topology of $\Lambda^U$ is obtained from the local continuity property of $\sigma$. By $(\star)$, there is a continuous function $\sigma^N_U:N \to O$ such that $\sigma^N_O(t)=K_t$ and $\sigma^N_O(s) \in \text{obj}(\sigma(s))$ for all $s \in N$. Suppose that $[\Psi] \in A_{K_t}(U)$ where $\Psi:O' \to U$ and $O'$ is an $(O,K_t)$-neighborhood. Since $\sigma^N_U$ is continuous, there is an $M \subseteq N$ such that $O'$ is an $(O,\sigma^N_U(s))$-neighborhood for all $s \in M$. Hence, $\Psi$ represents an equivalence class $[\Psi]_s \in A_{\sigma^N_U(s)}(U)$ for all $s \in M$. The sets $\dot{\Psi}=\{[\Psi]_s|s \in M\}$ form a basis for the topology on $\Lambda^U$ so that the projection $p^U:\Lambda^U \to \mathbb{I}$, $p([\Psi]_s)=s$ is \'{e}tale.
\newline
\newline
To make this precise, first fix objects $(C_t,K_t) \in \text{obj}(\sigma(t))$ for each $t \in \mathbb{I}$. This is done using Lemma \ref{lemma_cont_implies_isotop_var_amb}. Then there are continuous functions $\sigma_1:N_1 \to C_1$, $\sigma_{1.5}:N_{1.5} \to C_{1.5}$, $\ldots$, $\sigma_n:N_n \to C_n$ that match on overlaps and have the property that $(C_j,\sigma_j(t)) \in \text{obj}(\sigma(t))$ for all $t \in N_j$. For $t \in N_i$, $i \in \{1,2,\ldots,n\}$, fix $(C_i,\sigma_i(t))$. The two knots fixed for $t \in N_{i+.5}$ will be managed with the intermediate knot $(C_{i+.5},\sigma_{i+.5}(t))$. 
\newline
\newline
Next, following Section \ref{sec_proof_inverse_determines}, we define an equivalence relation $\stackrel{t}{\sim}$ on the arrows from the fixed representatives at $t \in \mathbb{I}$ to $U$. Let $H_t^U$ be the (possibly empty) set of arrows in $\textbf{VK}$ to $U$ from all objects in the sets $\mathscr{O}_{\sigma_i(t)}$ with $t \in N_i$. In particular, if $t \in N_i \cap N_{i+1}$ for some $i$, $H_t^U$ is the set of arrows from objects in $\mathscr{O}_{\sigma_i(t)} \cup \mathscr{O}_{\sigma_{i+1}(t)}$. Suppose $\Psi:O_i' \to U$ and $\Phi:O_j' \to U$ for $O_i' \in \mathscr{O}_{\sigma_i(t)}$ and $O_j' \in \mathscr{O}_{\sigma_j(t)}$. If $i=j$, write $\Phi \stackrel{t}{\sim} \Psi$ if there is an $O' \in \mathscr{O}_{\sigma_i(t)}$ and inclusions $I_i:O' \to O_i'$, $I_j:O' \to O_j'$ such that $\Psi \circ I_i=\Phi \circ I_j$. In other words, $\Phi \stackrel{t}{\sim} \Psi$ if they represent the same element of $A_{\sigma_i(t)}(U)$. For $j=i+1$, write $\Phi \stackrel{t}{\sim} \Psi$ if there is an $O' \in \mathscr{O}_{\sigma_{i+.5}(t)}$ and arrows $\Xi_i:O' \to O_i'$, $\Xi_{i+1}: O' \to O_{i+1}'$ such that $\Psi \circ \Xi_i=\Phi \circ \Xi_{i+1}$ and the following diagram commutes, where all vertical arrows are inclusions.
\[
\xymatrix{O_{i}' \ar[d] & \ar[l]^-{\Xi_i} O' \ar[d] \ar[r]_{\Xi_{i+1}}& O_{i+1}' \ar[d]\\ 
C_i & \ar[l]_-{\Delta_i} C_{i+.5} \ar[r]^-{\Upsilon_{i+1}} \ar[r] & C_{i+1}
}
\]
It can be quickly checked that $\stackrel{t}{\sim}$ is reflexive, symmetric, and transitive. Denote by $F_t^U$ the set of $\stackrel{t}{\sim}$-equivalence classes of $H_t^U$. These sets will serve as the fibers of the bundle $p^U:\Lambda^U \to \mathbb{I}$. If $[\Psi] \in F_t^U$, we write $[\Psi]_t$ to denote the fiber in which it lies. Let $\Lambda^U=\bigsqcup_{t \in \mathbb{I}} F_t^U$ and let $p^U:\Lambda^U \to \mathbb{I}$ be the projection $p^U([\Psi]_t)=t$. As discussed above, the basic open sets of $\Lambda^U$ are of the form $\dot{\Psi}=\{[\Psi]_t \in F_t^U| t \in M \}$ where $M \subseteq \mathbb{I}$ is open, $\Psi:O' \to U$, and for all $t \in M$, $O' \in \mathscr{O}_{\sigma_i(t)}$ for some $i$. The verification that $p^U:\Lambda^U \to \mathbb{I}$ is \'{e}tale is left as an exercise.
\newline
\newline
Set $A_{\sigma}(U) \in \text{Sh}(\textbf{I})$ to be the sheaf of sections of $p^U$. If $\Phi: U \to V$ is an arrow of $\textbf{VK}$, there is an \'{e}tale space map $\Lambda(\Phi): \Lambda^U \to \Lambda^V$ defined by $\Lambda(\Phi)([\Psi]_t)=[\Phi \circ \Psi]_t$. The sections functor $\Gamma$ sends $\Lambda(\Phi)$ to an arrow $\Gamma (\Lambda(\Phi))$ in $\text{Sh}(\textbf{I})$, which is a natural transformation from $A_{\sigma}(U)$ to $A_{\sigma}(V)$. Define $A_{\sigma}(\Phi):A_{\sigma}(U)\to A_{\sigma}(V)$ by $A_{\sigma}(\Phi)=\Gamma(\Lambda(\Phi))$. This defines a functor $A_{\sigma}:\textbf{VK} \to \text{Sh}(\textbf{I})$.

\begin{lemma} \label{lemma_pts_of_A_sigma} For a variable-space isotopy $\sigma:\mathbb{I} \to \pi_0(\smallint_{\textbf{VK}} F)$ and $(U,K) \in \text{obj}(\sigma(t))$, $t^* \circ A_{\sigma}\cong A_K$. 
\end{lemma}
\begin{proof} For any object $V$ of $\textbf{VK}$, $A_{\sigma}(V)$ is the sheaf of sections of the \'{e}tale space $\Lambda^V$. Applying the inverse image functor $t^*:\text{Sh}(\textbf{I}) \to \textbf{Sets}$ gives the stalk of $\Lambda^V$ over $t$. But the stalk over $t$ is an equivalence class of arrows to $V$ from objects in $\mathscr{O}_{\sigma_i(t)}$ where $t \in N_i$. Hence, $t^* \circ A_{\sigma}(V)$ can be identified with $A_{\sigma_i(t)}(V)$. This is natural in $V$, so that $t^* \circ A_{\sigma}$ is naturally isomorphic to $A_{\sigma_i(t)}$. Since $K$ and $\sigma_i(t)$ are in the same connected component of $\smallint_{\textbf{VK}} F$, it follows from Lemma \ref{lemma_same_knots_imp_nat_iso} that $A_K$ and $A_{\sigma_i(t)}$ are naturally isomorphic. Hence, $A_K$ and $t^* \circ A_{\sigma}$ are naturally isomorphic.   
\end{proof}

Now we are ready to prove Lemma \ref{lemma_meta_isotopy_inverse_determines}. The two parts are proved separately below.

\begin{proof}[Proof of Lemma \ref{lemma_meta_isotopy_inverse_determines}(1)] If $A_{\sigma}$ is continuous and filtering, $\sigma_{A_{\sigma}}(t)$ is the connected component $\textbf{V}_t$ of $\smallint_{\textbf{VK}} F$ corresponding to the continuous filtering functor $t^* \circ A_{\sigma}:\textbf{VK} \to \textbf{Sets}$. By Lemma \ref{lemma_pts_of_A_sigma}, this is naturally isomorphic to $A_K:\textbf{VK} \to \textbf{Sets}$ for any $(U,K) \in \text{obj}(\sigma(t))$. Since $\textbf{V}_K=\textbf{V}_t$, it follows that $\sigma_{A_{\sigma}}(t)=\sigma(t)$ for all $t \in \mathbb{I}$. Hence, the second part of the lemma follows from the first.
\newline
\newline
To prove the first claim, it suffices to show that $A_{\sigma}$ satisfies properties $(i)-(iv)$ of \cite{mac_moer}, Theorem VII.10.1. Briefly, they are: $(i)$ $A_{\sigma}$ is nonempty, $(ii)$ pairs of objects of $\textbf{VK}$ are suitably connected, $(iii)$ parallel arrows have equalizers, and $(iv)$ $A_{\sigma}$ maps covering sieves to epimorphic families. This is now a routine task since $t^* \circ A_{\sigma}:\textbf{VK} \to \textbf{Sets}$ is continuous and filtering by Lemma \ref{lemma_pts_of_A_sigma}. As an illustration, we verify $(iii)$. Suppose $\Phi_1,\Phi_2:U \to V$, $N \subseteq \mathbb{I}$, $u \in A_{\sigma}(U)(N)$, and $A_{\sigma}(\Phi_1)_N(u)=A_{\sigma}(\Phi_2)_N(u)$. It must be proved that there is an open cover of $N$ by arrows $j:M \to N$ (in $\textbf{I}$) such that there exists an arrow $\Upsilon:W \to U$ (in $\textbf{VK}$) and an $w \in A_{\sigma}(W)(M)$ such that $\Phi_1 \circ \Upsilon=\Phi_2 \circ \Upsilon$ and $A_{\sigma}(\Upsilon)_{M}(w)=A_{\sigma}(U)(j)(u)$. Let $t \in N$. Then $(t^* \circ A_{\sigma})(\Phi_1)(\text{germ}_t u)=(t^* \circ A_{\sigma})(\Phi_2)(\text{germ}_t u)$. Since $t^* \circ A_{\sigma}$ is filtering and hence parallel arrows have equalizers, it follows that there is an arrow $\Upsilon:W \to U$ and a $w_t \in (t^* \circ A_{\sigma})(W)$ such that $\Phi_1 \circ \Upsilon=\Phi_2 \circ \Upsilon$ and $(t^*\circ A_{\sigma})(\Upsilon)(w_t)=\text{germ}_t u$. By choosing a sufficiently small neighborhood $M \subseteq N$ of $t$, it follows there is $w \in A_{\sigma}(W)(M)$ such that $w_t=\text{germ}_t w$ and $A_{\sigma}(\Upsilon)_M(w)$ is a section of $p^U$ on $M$ that agrees with the restriction of $u$ to $M$. Hence, $A_{\sigma}(\Upsilon)_{M}(w)=A_{\sigma}(U)(j)(u)$, where $j:M \to N$ is inclusion. 
\end{proof}

\begin{proof}[Proof of Lemma \ref{lemma_meta_isotopy_inverse_determines}(2)] Since $\sigma=\sigma_A$, $A_{\sigma} =A_{\sigma_A}$ and it suffices to prove that $A \cong A_{\sigma_A}$. For $U$ an object of $\textbf{VK}$, let  $q^U:E^U \to \mathbb{I}$ be the \'{e}tale space of germs of the sheaf $A(U)$ on $\mathbb{I}$. It will be shown that $q^U:E^U \to \mathbb{I}$ is equivalent as an \'{e}tale space to the bundle $p^U:\Lambda^U \to \mathbb{I}$ constructed above from the variable-space isotopy $\sigma_A$.
\newline
\newline
Recall that $\sigma_A(t)$ is the variable-space knot determined by $t^* \circ A:\textbf{VK} \to \textbf{Sets}$. By Lemma \ref{lemma_cff_is_vi}, if $u \in A(U)(N)$ for some $N \subseteq \mathbb{I}$, the knots in $U$ determined by $\text{germ}_t u \in t^* \circ A(U)$ form a path $\sigma^N_U:N \to U$. The \'{e}tale space $\Lambda^U$ is built from a fixed set of paths $\sigma_1:N_1 \to C_1,\sigma_{1.5}:N_{1.5} \to C_{1.5},\ldots,\sigma_n:N_n \to C_n$ obtained from applying Lemma \ref{lemma_cont_implies_isotop_var_amb} to $\sigma_A$. Combining these two results, it may be assumed that for each $j=1,1.5,2,\ldots,n$ there is a $c_j \in A(C_j)(N_j)$ such that $\sigma_j(t)$ is the knot determined by $\text{germ}_t c_j \in t^* \circ A(C_j)$ for all $t \in N_j$. Furthermore, it may be assumed that for $i=1,2,\ldots,n-1$ and $t \in N_i \cap N_{i+1}$, $\Upsilon_{i+1} \cdot \text{germ}_t c_{i+.5}= \text{germ}_t c_{i+1}$ and $\Delta_{i} \cdot \text{germ}_t c_{i+.5}=\text{germ}_t c_i$. This follows from the fact that $t^* \circ A$ is filtering and hence the path component $C_{i+.5}$ can be chosen to fit into a diagram $\xymatrix{C_{i} & \ar[l]_-{\Delta_i} C_{i+.5} \ar[r]^{\Upsilon_{i+1}} & C_{i+1}}$ for which there is a $c_t \in t^* \circ A (C_{i+.5})$ such that $\Upsilon_{i+1} \cdot c_t=\text{germ}_t c_{i+1}$ and $\Delta_i \cdot c_t=\text{germ}_t c_i$. Then $c_t=\text{germ}_t c_{i+.5}$ for some $c_{i+.5} \in A(C_{i+.5})(M)$ and hence $\sigma_{i+.5}$ can be defined as the path determined in this way by the fixed elements $c_{i+.5}$.
\newline
\newline
By Lemma \ref{lemma_same_knots_imp_nat_iso}, for each $i=1,2,\ldots,n$ and $t \in N_i$, there is a natural isomorphism $\alpha^{\text{germ}_t c_i}:t^* \circ A \to A_{\sigma_i(t)}$ given by our fixed choice of element $\text{germ}_t c_i \in t^*\circ A(C_i)$. For each object $U$ of $\textbf{VK}$ and $u_t \in t^* \circ A (U)$, $\alpha^{\text{germ}_t c_i}(u_t)=[\Psi_{u_t}]$, where $\Psi_{u_t}:O_t' \to U$ and $O_t' \in \mathscr{O}_{\sigma_i(t)}$. Since $t^*\circ A(U)$ is the stalk at $t$ of the sheaf $A(U)$, the maps $\alpha^{\text{germ}_t c_i}$ combine to yield a fiber-preserving bijection $\eta^U:E^U \to \Lambda^U$ of \'{e}tale spaces. Assuming for the moment that $\eta^U$ is well-defined, we observe that $\eta^U$ is continuous with continuous inverse. When $A(U)$ is identified with an \'{e}tale space, the basic open sets are of the form $\{\text{germ}_t u| t \in N\}$ where $N \subseteq \mathbb{I}$ and $u \in A(U)(N)$. But the basic open sets of $\Lambda^U$ are of the form $\{[\Psi]_t| t\in N\}$. For $N$ a sufficiently small neighborhood of $t$, $\eta^U([\text{germ}_s u])=[\Psi_{\text{germ}_t u}]_s$ for all $s \in N$. Thus, $\eta^U$ maps basic open sets bijectively from $E^U$ to $\Lambda^U$ and we conclude that $\eta^U$ is continuous with continuous inverse. Similarly, $\eta^U$ is natural in $U$, since $\alpha^{\text{germ}_t c_i}$ is natural in $U$ for all $t$.
\newline
\newline
It now must be shown $\eta^U$ is well-defined. It may be assumed that $t \in N_i \cap N_{i+1}=N_{i+.5}$ for some $i$ as all other cases are trivial. Suppose that $\alpha^{\text{germ}_t c_{i+.5}}(u_t)=[\Psi_{u_t}]$ for some $u_t \in t^* \circ A(U)$ and $\Psi:O' \to U$ with $O' \in \mathscr{O}_{\sigma_{i+.5}(t)}$. Recall that $[\Psi_{u_t}]$ is defined by a diagram $\xymatrix{U & \ar[l]_-{\Psi_{u_t}} O' \ar[r]^-{I_{i+.5}} & C_{i+.5}}$ such that there is an $o_t' \in t^* \circ A(O')$ satisfying $\Psi_{u_t} \cdot o_t'=u_t$ and $I_{i+.5}\cdot o_t'=\text{germ}_t c_{i+.5}$. Set $\Omega_i=\Delta_i \circ I_{i+.5}$, $\Omega_{i+1}=\Upsilon_{i+1} \circ I_{i+.5}$. By Lemma \ref{lemma_factors}, $\Omega_j$ factors as $\xymatrix{O' \ar[r]^-{\Omega_j'} & O_j' \ar[r]^-{I_j} & C_j }$ for $j=i,i+1$, where $\Omega_j'$ is an isomorphism and $I_j$ is inclusion. Hence, $\Psi_{u_t}$ factors through $\Omega_j'$ for $j=i,i+1$. Write $\Psi_{u_t}=\Theta_{u_t} \circ \Omega_{i+1}'$ and $\Psi_{u_t}=\Phi_{u_t} \circ \Omega_i'$ for some arrows $\Theta_{u_t}$, $\Phi_{u_t}$. This gives the commutative diagram below:
\[
\xymatrix{
C_{i+1} & \ar[l]_-{I_{i+1}} O_{i+1}' \ar@{-->}[dr]^{\Theta_{u_t}} & \\
C_{i+.5} \ar[u]^-{\Upsilon_{i+1}} \ar[d]_-{\Delta_i} & \ar[l]_-{I_{i+.5}} O' \ar[r]^-{\Psi_{u_t}}\ar[u]^-{\Omega_{i+1}'} \ar[d]_-{\Omega_i'} & U \\
C_i & \ar[l]^-{I_i} O_i' \ar@{-->}[ur]_{\Phi_{u_t}} &
}
\]
Then $(I_j \circ \Omega_j') \cdot o_t'=\text{germ}_t c_j$ for $j=i,i+1$. Hence, $\alpha^{\text{germ}_t c_i}(u_t)=[\Phi_{u_t}]$ and $\alpha^{\text{germ}_t c_{i+1}}(u_t)=[\Theta_{u_t}]$. But the above diagram means exactly that $\Phi_{u_t} \stackrel{t}{\sim} \Theta_{u_t}$. Thus, $\eta^U$ is a well-defined equivalence of \'{e}tale spaces. This implies that $A$ and $A_{\sigma_A}$ are naturally isomorphic.   
\end{proof}

\section{Virtual knot invariants} \label{sec_invar}

\subsection{Variable-space knot invariants} \label{sec_cohom} Let $\mathbb{G}$ be a nonempty set with the discrete topology. A knot invariant valued in $\mathbb{G}$ is a continuous function $\nu:\mathbb{K} \to \mathbb{G}$, so that $\nu$ is constant on the path components of $\mathbb{K}$. This requirement can be translated into $\pi_0(\smallint_{\textbf{K}} F)$. A function $\nu:\pi_0(\smallint_{\textbf{K}} F) \to \mathbb{G}$ is a knot invariant if for every function $\sigma:\mathbb{I} \to \pi_0(\smallint_{\textbf{K}} F)$ satisfying the local continuity property $(\star)$, we have $\nu(\sigma(0))=\nu(\sigma(1))$. This motivates the following.

\begin{definition}[Variable-space knot invariant] \label{defn_var_space_knot_invar} A \emph{variable-space knot invariant valued in }$\mathbb{G}$ is a function $\nu:\pi_0(\smallint_{\textbf{VK}} F) \to \mathbb{G}$ such that for every variable-space isotopy $\sigma:\mathbb{I} \to \pi_0(\smallint_{\textbf{VK}} F)$, we have $\nu(\sigma(0))=\nu(\sigma(1))$. 
\end{definition}

This notion of invariant for virtual knots is closely related to sheaf cohomology. Recall that the constant presheaf $\Delta_{\mathbb{G}}:\textbf{VK}^{\text{op}} \to \textbf{Sets}$ is defined by $\Delta_{\mathbb{G}}(U)=\mathbb{G}$ for all objects $U$ and $\Delta_{\mathbb{G}}(\Psi)=1_{\mathbb{G}}$ for arrows $\Psi:U \to V$. Applying the associated sheaf functor $\textbf{a}:\textbf{Sets}^{\textbf{VK}^{\text{op}}} \to \text{Sh}(\textbf{VK})$ (see \cite{mac_moer}, Section III.5) to $\Delta_{\mathbb{G}}$ gives a sheaf $\Gamma_{\mathbb{G}}$ that can be described explicitly as follows. For each surface $\Sigma$, consider the covering space $\zeta^{\mathbb{G} \times \Sigma}_{\Sigma}:\mathbb{G} \times \Sigma \to \Sigma$ defined by $\zeta^{\mathbb{G}\times \Sigma}_{\Sigma}(g,z)=z$. This induces a covering space $Z_{\mathbb{K}(\Sigma)}^{\mathbb{G} \times \mathbb{K}(\Sigma)}:\mathbb{G} \times \mathbb{K}(\Sigma) \to \mathbb{K}(\Sigma)$. For $U \subseteq \mathbb{K}(\Sigma)$, let $\Gamma_{\mathbb{G}}(U)$ be the set of sections $s:U \to \mathbb{G} \times \mathbb{K}(\Sigma)$ of $Z^{\mathbb{G}\times \mathbb{K}(\Sigma)}_{\mathbb{K}(\Sigma)}$. For an arrow $\Psi:\widebar{U} \to U$, with $U \subseteq \mathbb{K}(\Sigma),\widebar{U} \subseteq \mathbb{K}(\widebar{\Sigma})$, and any $s \in \Gamma_{\mathbb{G}}(U)$, observe that there is a unique section $\widebar{s} \in \Gamma_{\mathbb{G}}(\widebar{U})$ such that the following diagram commutes:
\[
\xymatrix{
\mathbb{G} \times \mathbb{K}(\widebar{\Sigma}) \ar[r]^-{1_{\mathbb{G}} \times \Psi} & \mathbb{G} \times \mathbb{K}(\Sigma) \\
\wbar{U} \ar[r]^-{\Psi} \ar[u]^-{\bar{s}} & U \ar[u]_-{s}}.
\]
A presheaf $\Gamma_{\mathbb{G}}:\textbf{VK}^{\text{op}} \to \textbf{Sets}$ is obtained by setting $\Gamma_{\mathbb{G}}(\Psi)(s)=\widebar{s}$. Moreover, $\Gamma_{\mathbb{G}}$ is a sheaf. To see this, let $\{s_{\Psi}:\widebar{U} \to \mathbb{G} \times \mathbb{K}(\Sigma)|\Psi \in S\}$ be a matching family for a covering sieve $S=\{\Psi:\widebar{U} \to U\}$ on $U$. For each $s_{\Psi}$, there is a section $(s|\Psi(\widebar{U})):\Psi(\widebar{U}) \to \mathbb{G} \times \mathbb{K}(\Sigma)$ defined by $(s|\Psi(\widebar{U}))(\Psi(x))=(p_1(s_{\Psi}(x)),\Psi(x))$, where $p_1:\mathbb{G} \times \mathbb{K}(\Sigma)\to \mathbb{G}$ is projection onto the first factor. Since $\{s_{\Psi}\}$ is a matching family, if $\Psi,\Phi\in S$ have overlapping images, $(s|\Psi(\text{dom}(\Psi)))(\Psi(x))=(s|\Phi(\text{dom}(\Phi))(\Phi(y))$ whenever $\Psi(x)=\Phi(y)$. Since $\bigcup_{\Psi \in S} \Psi(\text{dom}(\Psi))=U$, this implies that the restrictions patch together to give a section $s:U \to \mathbb{G} \times \mathbb{K}(\Sigma)$ such that $\Gamma_{\mathbb{G}}(\Psi)(s)=s_{\Psi}$ for all $\Psi \in S$. Hence, $\Gamma_{\mathbb{G}}$ is a sheaf. We will follow the usual convention and denote $\Gamma_{\mathbb{G}} \cong \textbf{a} (\Delta_{\mathbb{G}})$ by $\Delta_{\mathbb{G}}$. 

\begin{lemma} \label{thm_hom_v_knot_invar} Variable-space knot invariants valued in $\mathbb{G}$ are in one-to-one correspondence with $\text{Hom}_{\text{Sh}(\textbf{VK})}(1,\Delta_{\mathbb{G}})$, where $1$ is a terminal object of $\text{Sh}(\textbf{VK})$.
\end{lemma}
\begin{proof} Take as terminal object the sheaf which sends every object to the one-point set $\mathbbm{1}=\{0\}$ and every arrow to the identity. A natural transformation $\eta:1 \to \Delta_{\mathbb{G}}$ gives for each pair $U,V$ of objects and each arrow $\Psi:U \to V$ a commutative triangle:
\[
\xymatrix{ & \Delta_{\mathbb{G}}(V) \ar[d]^-{\Delta_{\mathbb{G}}(\Psi)} \\
1 \ar[ur]^-{\eta_V} \ar[r]_-{\eta_U} & \Delta_{\mathbb{G}}(U)
}
\] 
For each object $U$, $\eta$ chooses a section $\eta_U(0):U \to \mathbb{G} \times \mathbb{K}(\Sigma)$ of $Z^{\mathbb{G} \times \mathbb{K}(\Sigma)}_{\mathbb{K}(\Sigma)}$. Then if $C \subseteq \mathbb{K}(\Sigma)$ is a path component, there is a $g_C \in \mathbb{G}$ such that $\eta_C(0)(K)=(g_C,K)$ for all $K \in C$. Suppose that two path components $C,\widebar{C}$ are connected by a path of arrows in $\textbf{VK}$:
\begin{equation*} \label{eqn_thm_hom_v_knot_invar}
\xymatrix{C=C_1 & \ar[l] C_{1.5} \ar[r] & C_2 & \ar[l] \cdots \ar[r] & C_{n-1} & \ar[l] C_{n-.5} \ar[r] & C_n=\widebar{C},}
\end{equation*}
where each $C_j \subseteq \mathbb{K}(\Sigma_j)$ is a path component. Apply $\Delta_{\mathbb{G}}$ and $\eta$ to this path of arrows. Then the commutative triangle implies that $g_{C_i}=g_{C_j}$ for all $i,j$. Thus, there is a well-defined function $\nu:\pi_0(\smallint_{\textbf{VK}} F) \to \mathbb{G}$ given by $\nu(\textbf{V})=g_C$ if there is a path component $C \subseteq \mathbb{K}(\Sigma)$ such that $(C,K) \in \text{obj}(\textbf{V})$ for some $K\in C$. Now, let $\sigma:\mathbb{I} \to \pi_0(\smallint_{\textbf{VK}} F)$ be a variable-space isotopy. By Lemma \ref{lemma_cont_implies_isotop_var_amb}, there are path components $C_1,C_{1.5},\ldots,C_n$ as above and a sequence of continuous functions $\sigma_1:N_1 \to C_1, \sigma_{1.5} \to C_{1.5},\ldots,\sigma_n:N_n \to C_n$ such that $\sigma_j(s) \in \text{obj}(\sigma(s))$ whenever $s \in N_j$. Hence, $\nu(\sigma(0))=g_{C_1}=g_{C_n}=\nu(\sigma(1))$ and $\eta$ uniquely determines a variable-space knot invariant.
\newline
\newline
Conversely, suppose that $\nu:\pi_0(\smallint_{\textbf{VK}} F) \to \mathbb{G}$ is a variable-space knot invariant. For $U \subseteq \mathbb{K}(\Sigma)$, define $s_U:U \to \mathbb{G} \times \mathbb{K}(\Sigma)$ by $s_U(K)=(\nu(\textbf{V}_K),K)$. Then $Z_{\mathbb{G}}^{\Sigma} \circ s_U=1_U$. But $s_U$ is also continuous: if $K_1,K_2$ are in the same path component of $U$, $\textbf{V}_{K_1}$ and $\textbf{V}_{K_2}$ are variable-space isotopic and hence $\nu(\textbf{V}_{K_1})=\nu(\textbf{V}_{K_2})$. Let $\eta:1 \to \Delta_{\mathbb{G}}$ be defined on components by $\eta_U(0)=s_U$. Now, if $\Psi:\widebar{U} \to U$ is an arrow and $K \in \widebar{U}$, we have $\textbf{V}_{K}=\textbf{V}_{\Psi(K)}$ and hence $\nu(\textbf{V}_K)=\nu(\textbf{V}_{\Psi(K)})$. Thus, $(1_{\mathbb{G}} \times \Psi) \circ s_{\widebar{U}}(K)=(\nu(\textbf{V}_{K}),\Psi(K))=(\nu(\textbf{V}_{\Psi(K)}),\Psi(K))=s_U(\Psi(K))$. This implies $s_{\widebar{U}}=\Delta_{\mathbb{G}}(\Psi)(s_U)$ and we conclude that $\eta:1 \to \Delta_{\mathbb{G}}$ is a natural transformation.   
\end{proof}

Now let $\mathbb{G}$ be an abelian group with the discrete topology. Then $H^0(\mathbb{K},\mathbb{G})$ is the group of knot invariants valued in $\mathbb{G}$. Here $H^0$ denotes the usual singular cohomology group. Translating this to sheaf cohomology, this means that classical knot invariants in $\mathbb{G}$ are identified with the group $H^0(\text{Sh}(\textbf{K}),\Delta_{\mathbb{G}})$, where $\Delta_{\mathbb{G}}$ is the constant sheaf in $\text{Sh}(\textbf{K})$ (\cite{johnstone}, Example 8.15.(i)). We now show that the same result holds for variable-space knot invariants.

\begin{theorem} $H^0(\text{Sh}(\textbf{VK}),\Delta_{\mathbb{G}})$ is the group of variable-space knot invariants valued in $\mathbb{G}$. 
\end{theorem}
\begin{proof} Since $\mathbb{G}$ is an abelian group, $\Delta_{\mathbb{G}}$ is an abelian sheaf. By definition, $H^j(\text{Sh}(\textbf{VK}),\Delta_{\mathbb{G}})$ is the $j$-th right-derived functor of $\text{Hom}_{\text{Sh}(\textbf{VK})}(1,\Delta_{\mathbb{G}})$ (see \cite{johnstone}, Definition 8.14). Lemma \ref{thm_hom_v_knot_invar} then implies that $H^0(\text{Sh}(\textbf{VK}),\Delta_{\mathbb{G}})$ is the group of variable-space knot invariants valued in $\mathbb{G}$. 
\end{proof}

\subsection{Virtual $\&$ variable space knot invariants} \label{sec_invar_geom} Here we prove that variable-space knot invariants (Definition \ref{defn_var_space_knot_invar}) are in one-to-one correspondence with virtual knot invariants. The argument proceeds as in Sections \ref{sec_points} and \ref{sec_virtual_isotopy_paths}. First it is shown that every continuous filtering functor $A: \textbf{G} \to \text{Sh}(\textbf{VK})$ determines a variable-space knot invariant $\nu_A:\pi_0(\smallint_{\textbf{VK}} F) \to \mathbb{G}$ that is unique up to isomorphism. Then it is shown that every variable-space knot invariant $\nu:\pi_0(\smallint_{\textbf{VK}} F) \to \mathbb{G}$ determines a continuous filtering functor $A_{\nu}:\textbf{G} \to \text{Sh}(\textbf{VK})$ such that $\nu_{A_{\nu}}=\nu$. We begin by calculating $A(H)$ for $H \subseteq \mathbb{G}$ and $A:\textbf{G} \to \text{Sh}(\textbf{VK})$ continuous and filtering.

\begin{lemma}\label{prop_calculate_A} For $H \subseteq \mathbb{G}$ and $C_1,C_2\ldots,C_n \subseteq \mathbb{K}(\Sigma)$ pairwise disjoint path components,
\[
A(H)\left(\bigsqcup_{j=1}^n C_j \right) \cong \left\{\begin{array}{cl} \varnothing & \text{if } (\exists j) (\forall h \in H) A(\{h\})(C_j)=\varnothing \\ \mathbbm{1} & \text{else}\end{array} \right.
\]
Furthermore, for every path component $C$ in $\textbf{VK}$, there is a unique $h \in \mathbb{G}$ such that $A(\{h\})(C) \ne \varnothing$.
\end{lemma}

\begin{proof} The proof is delayed to Section \ref{sec_prop_calculate_A}.
\end{proof}

Then $\nu_A:\pi_0(\smallint_{\textbf{VK}} F) \to \mathbb{G}$ may be defined as follows. Let $\textbf{V} \in \pi_0(\smallint_{\textbf{VK}} F)$ and suppose that $(C,K) \in \text{obj}(\textbf{V})$ where $C \subseteq \mathbb{K}(\Sigma)$ is a path component. By Proposition \ref{prop_calculate_A}, there is a unique $h \in \mathbb{G}$ such that $A(\{h\})(C) \ne \varnothing$. Define $\nu_A(\textbf{V})=h$. To see that this is well-defined, suppose that $\Psi:U \to C$ is an arrow in $\textbf{VK}$. If $A(\{h\})(C) \ne \varnothing$, then is must be that $A(\{h\})(U) \ne \varnothing$. Moreover, if $(C,K)$, $(\widebar{C},\widebar{K}) \in \text{obj}(\textbf{V})$, then they are connected by a path of arrows. By repeatedly applying the preceding observation, it follows that $A(\{h\})(\widebar{C}) \ne \varnothing$. That $\nu_A$ is a variable-space knot invariant follows as in the proof of Lemma \ref{thm_hom_v_knot_invar}. Uniqueness follows from the next pair of lemmas.

\begin{lemma} \label{lemma_nat_iso_implies_same_invariant} If $A_1,A_2: \textbf{G} \to \text{Sh}(\textbf{VK})$ are continuous filtering functors and $\eta:A_1 \to A_2$ is a natural isomorphism, then $\nu_{A_1}=\nu_{A_2}$.
\end{lemma}
\begin{proof} For every $H \subseteq \mathbb{G}$, there is a natural transformation $\eta_H:A_1(H) \to A_2(H)$ such that for each object $U$ of $\textbf{VK}$, there is an isomorphism $\eta_{H,U}:A_1(H)(U) \to A_2(H)(U)$. By Lemma \ref{prop_calculate_A}, if $C \subseteq \mathbb{K}(\Sigma)$ is a path component, $A_i(H)(C)$ is either empty or a one-point set. This implies that $A_1(H)(C)$ and $A_2(H)(C)$ are either both empty or nonempty. Hence, $\nu_{A_1}=\nu_{A_2}$.
\end{proof}

\begin{lemma} \label{lemma_same_invariant_implies_nat_iso} If $A_1,A_2: \textbf{G} \to \text{Sh}(\textbf{VK})$ are continuous filtering functors and $\nu_{A_1}=\nu_{A_2}$, then $A_1,A_2$ are naturally isomorphic.
\end{lemma}
\begin{proof} Let $C\subseteq \mathbb{K}(\Sigma)$ be a path component, $K \in C$ and $\textbf{V}$ the connected component of $\smallint_{\textbf{VK}} F$ containing $(C,K)$. Since $\nu_{A_1}=\nu_{A_2}$, we have $A_1(\{\nu_{A_1}(\textbf{V})\})(C) \ne \varnothing$ if and only if $A_2(\{\nu_{A_2}(\textbf{V}) \})(C) \ne \varnothing$. Breaking any object $U$ of $\mathbf{VK}$ into its path components, Lemma \ref{prop_calculate_A} implies that for any $H \subseteq \mathbb{G}$, $A_1(H)(U) \ne \varnothing$ if and only if $A_2(H)(U) \ne \varnothing$. Since $A_i(H)(U)$ is either a one point set or the empty set, there is a unique choice of maps $\eta_{H,U}:A_1(H)(U) \to A_2(H)(U)$ that are natural in both $H$ and $U$. Hence, $A_1$ and $A_2$ are naturally isomorphic.
\end{proof}

Conversely, if $\nu:\pi_0(\smallint_{\textbf{VK}} F) \to \mathbb{G}$ is a virtual knot invariant, define a continuous filtering functor $A_{\nu}:\textbf{G} \to \text{Sh}(\textbf{VK})$ as follows. For $U$ an object of $\textbf{VK}$, write:
\[
\nu(U)=\{g \in \mathbb{G}| \exists \textbf{V}\in \pi_0(\smallint\hspace{0cm}_{\small\textbf{VK}} F) \text{ and } K \in U \text{ s.t. } (U,K) \in \text{obj}(\textbf{V}) \text{ and }  \nu(\textbf{V})=g \}.
\]
In other words, $\nu(U)$ is the set of values of the invariant taken over all variable-space knots having a representative in $U$. Then $A_{\nu}:\textbf{G} \to \text{Sh}(\textbf{VK})$ is defined for objects $U$ of $\textbf{VK}$ and $H \subseteq \mathbb{G}$ by:
\[
A_{\nu}(H)(U)=\left\{\begin{array}{cl} \varnothing & \text{if } \nu(U) \not\subseteq H \\ \mathbbm{1} & \text{else}\end{array} \right.
\]
For an arrow $\Psi:U \to V$ of $\textbf{VK}$, note that $\nu(V) \subseteq H$ implies $\nu(U) \subseteq H$. Hence, $A_{\nu}(H)(U)=\varnothing$ implies that $A_{\nu}(H)(V)=\varnothing$. Then there is a unique function $A_{\nu}(H)(\Psi):A_{\nu}(H)(V) \to A_{\nu}(H)(U)$, which is either $1_{\mathbbm{1}}:\mathbbm{1} \to \mathbbm{1}$, $1_{\varnothing}:\varnothing \to \varnothing$, or $!:\varnothing \to \mathbbm{1}$. Likewise, if $i:\widebar{H} \to H$ is an arrow in $\textbf{G}$, define a natural transformation $A_{\nu}(i):A_{\nu}(\widebar{H}) \to A_{\nu}(H)$ on components by $A_{\nu}(i)_U=\text{id}$ if $\nu(U) \subseteq \widebar{H} \subseteq H$ and otherwise $A_{\nu}(i)_U$ is the unique map $!:\varnothing \to A_{\nu}(H)(U)$. 

\begin{lemma} \label{lemma_nu_A_nu_eq_nu} $A_{\nu}:\textbf{G} \to \text{Sh}(\textbf{VK})$ is a continuous filtering functor such that $\nu_{A_{\nu}}=\nu$.
\end{lemma}
\begin{proof} Observe that $A_{\nu}(H)$ is indeed a sheaf for all $H \subseteq \mathbb{G}$. If $\{\Psi_i:U_i \to U\}$ is a covering sieve, a matching family consists of the choice of the single element in the one point set $A_{\nu}(H)(U_i)$ for all $i$. Since $\bigcup_i \Psi_i(U_i)=U$, this implies that $\nu(U) \subseteq H$. Then $A_{\nu}(H)(U) \ne \varnothing$ and hence every matching family has a unique amalgamation. An easy diagram chase shows that $A_{\nu}$ is a functor. For the second claim, note that if $C\subseteq \mathbb{K}(\Sigma)$ is a path component, then $\nu(C)=\{g\}$ for some $g \in \mathbb{G}$ and $A_{\nu}(\nu(C))(C)\ne \varnothing$. Hence, $\nu_{A_{\nu}}(\textbf{V})=g=\nu(\textbf{V})$ whenever $(C,K) \in \text{obj}(\textbf{V})$ for some $K\in C$.
\newline
\newline
To show that $A_{\nu}$ is continuous and filtering, we again use \cite{mac_moer}, Theorem VII.10.1. The first filtering property follows from the fact that $A_{\nu}(\nu(U))(U) \ne \varnothing$. For the second property, suppose that $A_{\nu}(H)(U) \ne \varnothing$ and $A_{\nu}(\widebar{H})(U) \ne \varnothing$. Then $\nu(U) \subseteq H \cap \widebar{H}$ and $A_{\nu}(H \cap \widebar{H})(U) \ne \varnothing$. Hence, the identity arrow $I:U \to U$ satisfies the property that there is a diagram $\xymatrix{H & \ar[l]_{i} H \cap \wbar{H} \ar[r]^{\bar{i}} & \wbar{H}}$ and an element $0 \in A_{\nu}(H \cap \widebar{H})(U)$ such that $A_{\nu}(i)_U(0)=A_{\nu}(H)(I)(0)$ and $A_{\nu}(\widebar{i})_U(0)=A_{\nu}(\widebar{H})(I)(0)$. Thus, the maximal sieve is a covering of $U$ with the required property. The third filtering condition is trivially satisfied since parallel arrows in $\textbf{G}$ are identical. 
\newline
\newline
For continuity, suppose that $S$ is a covering sieve of $H \subseteq \mathbb{G}$ and $A_{\nu}(H)(U) \ne \varnothing$. Since $S$ is a sieve, it contains all the one-point subsets of $H$. Since $A_{\nu}(H)(U) \ne \varnothing$, there is for every path component $\widebar{U} \subseteq U$ an $h \in H$ such that $A_{\nu}(\{h\})(\widebar{U}) \ne \varnothing$. Then the sieve generated by the inclusions of path components $j:\widebar{U} \to U$ is a covering of $U$ such that $A_{\nu}(i)_{\widebar{U}}(0)=A_{\nu}(H)(j)(0)$, where $i:\{h\} \to H$ is also inclusion. Thus, $A_{\nu}$ is continuous and the proof is complete.
\end{proof}

\begin{lemma} \label{lemma_A_nu_A_equals_A} For any continuous filtering functor $A:\textbf{G} \to \text{Sh}(\textbf{VK})$, $A_{\nu_{A}} \cong A$.
\end{lemma}
\begin{proof} Applying Lemma \ref{lemma_nu_A_nu_eq_nu} to $\nu_A$ gives $\nu_{A_{\nu_A}}=\nu_A$. Then use Lemma \ref{lemma_same_invariant_implies_nat_iso}.
\end{proof}

\begin{theorem} \label{thm_invariants} Variable-space and virtual knot invariants are in one-to-one correspondence.
\end{theorem}
\begin{proof} Again by \cite{mac_moer}, Corollary VII.10.2, the category $\underline{\text{Hom}}(\text{Sh}(\textbf{VK}),\text{Sh}(\textbf{G}))$ of geometric morphisms is equivalent to the category of continuous filtering functors $A:\textbf{G} \to \text{Sh}(\textbf{VK})$. The result then follows from Lemmas \ref{lemma_nat_iso_implies_same_invariant}, \ref{lemma_same_invariant_implies_nat_iso}, \ref{lemma_nu_A_nu_eq_nu}, and \ref{lemma_A_nu_A_equals_A}.
\end{proof}

\subsection{Proof of Lemma \ref{prop_calculate_A}} \label{sec_prop_calculate_A} Suppose that $A:\textbf{G} \to \text{Sh}(\textbf{VK})$ is continuous and filtering. First note that $A(\mathbb{G})$ and $A(\varnothing)$ are, respectively, terminal and initial objects of $\text{Sh}(\textbf{VK})$. This is due to the fact that the category of continuous filtering functors is equivalent to the category of geometric morphisms $N:\text{Sh}(\textbf{VK}) \to \text{Sh}(\textbf{G})$, so that $A$ may be identified with $N^* \circ \textbf{a} \circ \textbf{y}$ for some geometric morphism $N$, where $\textbf{y}$ denotes the Yoneda embedding, and $\textbf{a}:\textbf{Sets}^{\textbf{G}^{\text{op}}} \to \text{Sh}(\textbf{G})$ is the sheafification functor (see e.g. \cite{mac_moer}, Section VII.7). The claim then follows from the fact that $N^*$ and $\textbf{a}$ preserve finite limits and arbitrary colimits, since they are both left exact and left adjoints. 

\begin{lemma} \label{lemma_unique_u} For $U$ an object of $\textbf{VK}$ and $H \subseteq \mathbb{G}$, if $u_1,u_2 \in A(H)(U)$, then $u_1=u_2$
\end{lemma}
\begin{proof} Since $A$ is filtered, Theorem VII.10.1.(ii) of \cite{mac_moer} implies that $U$ is covered by a sieve $S$ consisting of arrows of the form $\Psi:V \to U$ for which there is a diagram $\xymatrix{H & \ar[l]_{i} M \ar[r]^{i} & H}$ and an $m \in A(M)(V)$ such that $A(i)_{V}(m)=A(H)(\Psi)(u_1)$ and $A(i)_{V}(m)=A(H)(\Psi)(u_2)$. Then $A(H)(\Psi)(u_1)=A(H)(\Psi)(u_2)$. The restrictions of $u_1$ and $u_2$ to $A(H)(V)$ thus define the same matching family for $S$. Since $A(H)$ is a sheaf, it follows that $u_1=u_2$.
\end{proof}

\begin{lemma} \label{lemma_at_most_one_h} If $U \ne \varnothing$ is an object of $\textbf{VK}$ and there are $h_1,h_2 \in \mathbb{G}$ such that $A(\{h_1\})(U)\ne \varnothing$ and $A(\{h_2\})(U) \ne \varnothing$, then $h_1=h_2$.
\end{lemma}
\begin{proof} As in the preceding lemma, this also uses the second filtering property of $A$. Suppose $h_1 \ne h_2$. The second filtering property implies that $U$ is covered by arrows of the form $\Psi:V \to U$ such that there is a diagram  $\xymatrix{\{h_1\} & \ar[l] M \ar[r] & \{h_2\}}$ in $\textbf{G}$ with $A(M)(V) \ne \varnothing$. Since $h_1 \ne h_2$, we must have that $M=\varnothing$. But $A(\varnothing)(V)=\varnothing$ whenever $V \ne \varnothing$. Thus, no such covering of $U$ is possible unless $U=\varnothing$. This contradiction implies that $h_1=h_2$.
\end{proof}

\begin{lemma} \label{lemma_unique_h} For $C \subseteq \mathbb{K}(\Sigma)$ a path component, there is a unique $h \in \mathbb{G}$ such that $A(\{h\})(C) \ne \varnothing$.
\end{lemma}
\begin{proof} Let $S$ be the covering sieve of the object $\mathbb{G}$ of $\textbf{G}$ consisting of its one point sets and the empty set. Since $A(\mathbb{G})(C) \ne \varnothing$, the continuity property of $A$ (\cite{mac_moer}, Theorem VII.10.1.(iv)) implies that for each $c \in A(\mathbb{G})(C)$,  there is a covering of $C$ by arrows $\Psi: U \to C$ such that for some $i:\{h\} \to \mathbb{G}$ in $S$ and $u \in A(\{h\})(U)$, $A(i)_U(u)=A(\mathbb{G})(\Psi)(c)$. In particular, $A(\{h\})(U) \ne \varnothing$ for some $h \in \mathbb{G}$. First suppose that $A(\{h_1\})(U_1) \ne \varnothing$ and $A(\{h_2\})(U_2)\ne \varnothing$ for some arrows $\Psi_1:U_1 \to C$, $\Psi_2:U_2 \to C$ of this cover of $C$ and $\Psi_1(U_1) \cap \Psi_2(U_2) \ne \varnothing$. Then the pullback $U$ of $\Psi_1$ over $\Psi_2$ is nonempty. Restricting to $U$ then implies that $A(\{h_1\})(U) \ne \varnothing$ and $A(\{h_2\})(U) \ne \varnothing$.  By Lemma \ref{lemma_at_most_one_h}, $h_1=h_2$.
\newline
\newline
Now consider two arbitrary arrows $\Psi:U \to C$, $\widebar{\Psi}: \widebar{U} \to C$ of the cover. Then for some $h,\widebar{h} \in \mathbb{G}$, $A(\{h\})(U) \ne \varnothing$ and $A(\{\widebar{h}\})(\widebar{U}) \ne \varnothing$. Any point of $\Psi(U)$ is connected to any point in $\widebar{\Psi}(\widebar{U})$ by a path $\sigma:\mathbb{I} \to C$. By compactness, the image of $\sigma$ is covered by finitely many arrows $\Psi_1:U_1 \to C,\ldots,\Psi_n:U_n \to C$ in the given cover of $C$, where $\Psi_1=\Psi$ and $\Psi_n=\widebar{\Psi}$. Each of these arrows must have an image with a nonempty intersection with at least one other arrow in the list. By hypothesis, for each $i=1,\ldots n$, there is an $h_i \in \mathbb{G}$ such that $A(\{h_i\})(U_i) \ne \varnothing$. From these two observations, it follows as in the previous paragraph that $h_i=h_j$ for all $i,j$. In particular, $h=\widebar{h}$.
\newline
\newline
Together with Lemma \ref{lemma_at_most_one_h}, this implies that there is a unique $h \in \mathbb{G}$ such that $A(\{h\})(U) \ne \varnothing$ for all arrows $\Psi:U \to C$ of the given covering of $C$. Since $A(\{h\})(U)$ is a one point set, choosing the unique element of each $A(\{h\})(U)$ defines a matching family on the given covering sieve. Since $A(\{h\})$ is a sheaf, this matching family has a unique amalgamation and hence $A(\{h\})(C) \ne \varnothing$.
\end{proof}

The proof of Lemma \ref{prop_calculate_A} will now be completed with the following lemma.

\begin{lemma} For $H \subseteq \mathbb{G}$, $A(H)(\bigsqcup_{j=1}^n C_j) \ne \varnothing$ if and only if for all $j=1,\ldots, n$, there is an $h \in H$ such that $A(\{h\})(C_j) \ne \varnothing$.
\end{lemma}
\begin{proof} The forward direction follows from the continuity of $A$. Let $S$ be the covering sieve of $H$ consisting of its one point sets and $\varnothing$. If $A(H)(\bigsqcup_{j=1}^n C_j) \ne \varnothing$, then there is a cover of $\bigsqcup_{j=1}^n C_j$ by arrows $\Psi: U \to \bigsqcup_{j=1}^n C_j$ such that $A(\{h\})(U) \ne \varnothing$ for some $h \in H$. In particular, there must be a cover of each of the path components $C_j$ consisting of arrows $V \to C_j$ such that $A(\{h\})(V) \ne \varnothing$ for some $h \in H$. Arguing as in the proof of Lemma \ref{lemma_unique_h}, it follows that $A(\{h\})(C_j) \ne \varnothing$ for some $h \in H$. Conversely, suppose that for each $j$ there is an $h \in H$ such that $A(\{h\})(C_j) \ne \varnothing$. The inclusion map $i:\{h\} \to H$ gives a natural transformation $A(i):A(\{h\}) \to A(H)$. Hence,  if $A(\{h\})(C_j) \ne \varnothing$, we must have $A(H)(C_j) \ne \varnothing$. Since $A(H)$ is a sheaf, the matching family given by the one-point sets $A(H)(C_j)$ for the cover generated by the inclusions $C_j \to \bigsqcup_{j=1}^n C_j$   has a unique amalgamation. Hence $A(H)(\bigsqcup_{j=1}^n C_j) \ne \varnothing$.
\end{proof}

\section{Geometric morphisms between sheaf-theoretic knot spaces} \label{sec_other}

\subsection{Projection of knots to virtual knots} \label{sec_project} For the space of knots $\mathbb{K}(\Sigma)$, let $J_{\textbf{K}(\Sigma)}$ denote the open cover topology on the category $\textbf{K}(\Sigma)$ of open sets of $\mathbb{K}(\Sigma)$. Let $\text{Sh}(\textbf{K}(\Sigma))=\text{Sh}(\textbf{K}(\Sigma),J_{\textbf{K}(\Sigma)})$ be the category of sheaves on the site $(\textbf{K}(\Sigma),J_{\textbf{K}(\Sigma)})$. Here we realize the ``projection of knots in $\Sigma \times \mathbb{R}$ to virtual knots'' as a geometric morphism $I_{\Sigma}:\text{Sh}(\textbf{K}(\Sigma)) \to \text{Sh}(\textbf{VK})$. The left adjoint $I^*_{\Sigma}: \text{Sh}(\textbf{VK}) \to \text{Sh}(\textbf{K}(\Sigma))$ can be described as follows.  For $P$ a sheaf on $(\textbf{VK},J_{\textbf{VK}})$, and $O \subseteq \mathbb{K}(\Sigma)$ open, define $I^*_{\Sigma}(P)(O)=P(O)$. For any natural transformation $\eta:P \to Q$ of sheaves $P,Q$ on $(\textbf{VK},J_{\textbf{VK}})$, define $I_{\Sigma}^*(\eta)$ to be restriction of $\eta$ to the objects of $\textbf{K}(\Sigma)$. 

\begin{lemma} If $P$ is a sheaf in $\text{Sh}(\textbf{VK})$, then $I^*_{\Sigma}(P)$ is a sheaf in $\text{Sh}(\textbf{K}(\Sigma))$.
\end{lemma} 
\begin{proof} Let $S=\{f:\widebar{O} \to O\} \in J_{\textbf{K}(\Sigma)}(O)$ be a covering sieve of $O$ and let $\{x_f \in P(\widebar{O})| (f:\widebar{O} \to O) \in S \}$ be a matching family. Denote by $(S)$ the sieve on $O$ in $\textbf{VK}$ generated by $S$. Then $(S)$ covers $O$ in $\textbf{VK}$ since $S$ covers $O$ in $\mathbb{K}(\Sigma)$. The matching family may likewise be extended to $(S)$. If $g=f \circ \Psi \in (S)$, where $f\in S$ and $\Psi:U \to \widebar{O}$ is an arrow in $\textbf{VK}$, define $x_g=P(\Psi)(x_f)$. Since all the arrows in $S$ are inclusions, this is independent of the factorization of $g$ and hence $x_g$ is well-defined. That $x_g$ is a matching family for $(S)$ follows from the following calculation:
\[
P(\Phi)(x_g)=P(\Phi)(P(\Psi)(x_f))=P( \Psi\circ \Phi)(x_f)=x_{f \circ \Psi \circ \Phi}=x_{g \circ \Phi}.
\]
This matching family must have a unique amalgamation $x \in P(O)$ since $P$ is a sheaf of $(\textbf{VK},J_{\textbf{VK}})$. Thus, the matching family $\{x_f\}$ for $S$ also has a unique amalgamation. 
\end{proof}

To show that $I_{\Sigma}^*$ is the inverse image of a geometric morphism $I_{\Sigma}$, we will again use the equivalence between the category of geometric morphisms $\text{Sh}(\textbf{K}(\Sigma)) \to \text{Sh}(\textbf{VK})$ and the category of continuous filtering functors $\textbf{VK} \to \text{Sh}(\textbf{K}(\Sigma))$. First note that the inclusion functor $\text{Sh}(\textbf{VK}) \rightarrowtail \textbf{Sets}^{\textbf{VK}^{\text{op}}}$ has as left adjoint the sheafification functor $\textbf{a}:\textbf{Sets}^{\textbf{VK}^{\text{op}}} \to \text{Sh}(\textbf{VK})$. Since sheafification preserves finite limits, this adjoint situation is a geometric morphism. Composing this with $I_{\Sigma}$ gives a geometric morphism $\text{Sh}(\textbf{K}(\Sigma)) \to \textbf{Sets}^{\textbf{VK}^{\text{op}}}$. The equivalence with the category of continuous filtering functors is obtained by composing the left adjoint of this geometric morphism with the Yoneda embedding $\textbf{y}:\textbf{VK} \to \textbf{Sets}^{\textbf{VK}^{\text{op}}}$ (see \cite{mac_moer}, Lemma VII.7.3, Corollary VII.7.4, and Corollary VII.10.2). Therefore, to show that the functor $I_{\Sigma}^*$ defined above is indeed the left adjoint of some geometric morphism $I_{\Sigma}$, it suffices to show that $I_{\Sigma}^* \circ \textbf{a} \circ \textbf{y}$ is continuous and filtering.

\begin{lemma} \label{lemma_restrict_cont_filtering} $I_{\Sigma}^* \circ \textbf{a} \circ \textbf{y}: \textbf{VK} \to \text{Sh}(\textbf{K}(\Sigma))$ is continuous and filtering.
\end{lemma}
\begin{proof} The composition $\textbf{a} \circ \textbf{y}:\textbf{VK} \to \text{Sh}(\textbf{VK})$ is  continuous and filtering, since the inclusion $\text{Sh}(\textbf{VK}) \to \textbf{Sets}^{\textbf{VK}^{\text{op}}}$ is a geometric morphism. It must be shown that the continuity and filtering properties (\cite{mac_moer}, Theorem VII.10.1) are preserved under restriction to the category of open sets of $\mathbb{K}(\Sigma)$. The ``connected'' property (i.e. \cite{mac_moer}, Theorem VII.10.1.ii)  will be proved explicitly; the other three properties follow from similar arguments. Let $U,\widebar{U} \in \text{obj}(\textbf{VK})$, $O \in \text{obj}(\textbf{VK})$, and $u \in I_{\Sigma}^* \circ \textbf{a} \circ \textbf{y}(U)(O), \widebar{u} \in I_{\Sigma}^* \circ \textbf{a} \circ \textbf{y}(\widebar{U})(O)$. Since $\textbf{a} \circ \textbf{y}$ is filtering, $O$ is covered by arrows $\Psi:W \to O$ in $\textbf{VK}$ for which there is a diagram:
\[
\xymatrix{U & \ar[l]_-{\Phi} V \ar[r]^-{\bar{\Phi}} & \wbar{U}},
\] 
and a $v \in \textbf{a} \circ \textbf{y}(V)(W)$ such that $\textbf{a} \circ \textbf{y}(\Phi)_W(v)=\textbf{a} \circ \textbf{y}(U)(\Psi)(u)$ and $\textbf{a} \circ \textbf{y}(\widebar{\Phi})_W(v)=\textbf{a} \circ \textbf{y}(U)(\Psi)(\widebar{u})$. By Lemma \ref{lemma_factors}, each $\Psi:W \to O$ factors as $\xymatrix{W \ar[r]^-{\Psi'}& O' \ar[r]^{I} & O}$, where $O' \subseteq \mathbb{K}(\Sigma')$, $\Sigma' \subseteq \Sigma$ is a subsurface, $\Psi'$ is an isomorphism and $I$ is induced by inclusion. Set $\widebar{O}=I(O')$ and let $g:\widebar{O} \to O$ be the inclusion arrow in $\textbf{K}(\Sigma)$. Let $\widebar{I}:O' \to \widebar{O}$ be the inclusion arrow in $\textbf{VK}$ which changes the background surface from $\Sigma'$ to $\Sigma$. Hence, $I=g \circ \widebar{I}$ in $\textbf{VK}$. Since the maps $\Psi$ cover $O$ in $\textbf{VK}$, so also do the maps $g$ in $\textbf{K}(\Sigma)$. Then it only needs to be shown that for some $\widebar{o} \in I_{\Sigma}^* \circ\textbf{a}\circ \textbf{y}(V)(\widebar{O})$: 
\[
I_{\Sigma}^* \circ\textbf{a} \circ \textbf{y}(\Phi)_{\widebar{O}}(\widebar{o})=I_{\Sigma}^* \circ \textbf{a} \circ \textbf{y}(U)(g)(u), \text{ and } I_{\Sigma}^* \circ \textbf{a} \circ \textbf{y}(\widebar{\Phi})_{\widebar{O}}(\widebar{o})=I_{\Sigma}^* \circ\textbf{a} \circ \textbf{y}(U)(g)(\widebar{u}).
\]
Recall that $\textbf{a}\circ \textbf{y}(V)=(\text{Hom}_{\textbf{VK}}(-,V)^+)^+$ where ``$+$'' indicates the Grothendieck plus-construction (see \cite{mac_moer}, Section III.5). Note that $\text{Hom}_{\textbf{VK}}(-,V)$ is already a separated presheaf, since arrows $X \to V$ in $\textbf{VK}$ that agree pointwise on $X$ are identified as the same arrow (see Section \ref{sec_site}). Thus, only one application of ``$+$'' is necessary to sheafify $\text{Hom}_{\textbf{VK}}(-,V)$. Then $v \in \textbf{a} \circ \textbf{y}(V)(W)$ can be represented by a matching family $v=\{\Upsilon_{\Xi}:X \to V| (\Xi:X \to W) \in Q\}$ where $Q=\{\Xi:X \to W\}$ is a covering sieve of $W$. Now, $\widebar{I} \circ \Psi'(Q)=\{\widebar{I} \circ \Psi' \circ \Xi:X\to \widebar{O}| \Xi \in Q\}$ is a covering sieve in $J_{\textbf{VK}}(\widebar{O})$.  Set $\widebar{o} \in I_{\Sigma}^* \circ\textbf{a}\circ \textbf{y}(V)(\widebar{O})$ to be the element represented by the matching family $\{\Upsilon_{\Xi}| \widebar{I} \circ \Psi' \circ \Xi \in\widebar{I} \circ \Psi'(Q)\}$. That this is a matching family follows immediately from the fact that $v$ is a matching family.
\newline
\newline
Let $u \in \textbf{a} \circ \textbf{y} (U)(O)$ be represented by the matching family $\{\Gamma_{\Delta}:Y \to U|(\Delta:Y \to O) \in R\}$ where $R \in J_{\textbf{VK}}(O)$. Consider the equation $\textbf{a} \circ \textbf{y}(\Phi)_W(v)=\textbf{a} \circ \textbf{y}(U)(\Psi)(u)$. The left hand side evaluates to $\{\Phi \circ \Upsilon_{\Xi}: \Xi \in Q\}$. The right hand side evaluates to the restriction along $\Psi$: $\{\Gamma_{\Psi \circ \Omega}| \Omega \in \Psi^*(R)\}$, where $\Psi^*(R)=\{\Omega:Z \to W| \Psi \circ \Omega \in R\}$. The equality means that the two matching families have a common refinement, i.e. there is a cover $T \subseteq Q \cap \Psi^*(R)$ such that $\Phi \circ \Upsilon_{\Lambda}=\Gamma_{\Psi \circ \Lambda}$ for all $\Lambda \in T$. 
\newline
\newline  
Finally, it is time to show that $I_{\Sigma}^* \circ\textbf{a} \circ \textbf{y}(\Phi)_{\widebar{O}}(\widebar{o})=I_{\Sigma}^* \circ \textbf{a} \circ \textbf{y}(U)(g)(u)$. First note that $\widebar{I} \circ \Psi'(T) \subseteq g^*(R) \cap \widebar{I} \circ \Psi'(Q)$. Indeed, if $\Theta \in \Psi^*(R)$, then $I \circ \Psi'\circ \Theta=\Psi \circ \Theta \in R$. Hence, $g \circ \widebar{I} \circ \Psi' \circ \Theta \in R$ and $\widebar{I} \circ \Psi' \circ \Theta \in g^*(R)$. Then since $T \subseteq \Psi^*(R)$, $\widebar{I} \circ \Psi'(T) \subseteq g^*(R)$ and $\widebar{I} \circ \Psi'(T) \subseteq g^*(R) \cap \widebar{I} \circ \Psi'(Q)$ as claimed. Now, for any $\Theta \in T$, $\Gamma_{g \circ \widebar{I} \circ \Psi' \circ \Theta}=\Gamma_{\Psi \circ \Theta}=\Phi \circ \Upsilon_{\Theta}$. Since $\widebar{I} \circ \Psi'(T) \subseteq g^*(R) \cap \widebar{I} \circ \Psi'(Q)$, this implies that the matching families $\{\Phi \circ \Upsilon_{\Xi}| \widebar{I} \circ \Psi' \circ \Xi \in \widebar{I} \circ \Psi'(Q)\}$ and $\{\Gamma_{g \circ \Omega}| \Omega \in g^*(R) \}$ have a common refinement to the cover $\widebar{I} \circ \Psi'(T)$. In other words, $I_{\Sigma}^* \circ\textbf{a} \circ \textbf{y}(\Phi)_{\widebar{O}}(\widebar{o})=I_{\Sigma}^* \circ \textbf{a} \circ \textbf{y}(U)(g)(u)$. Similarly, $I_{\Sigma}^* \circ \textbf{a} \circ \textbf{y}(\widebar{\Phi})_{\widebar{O}}(\widebar{o})=I_{\Sigma}^* \circ\textbf{a} \circ \textbf{y}(U)(g)(\widebar{u})$ and hence the second filtering property holds.
\end{proof}

\begin{theorem} \label{thm_restrict_is_geometric} For every $\Sigma$, there is a geometric morphism $I_{\Sigma}: \text{Sh}(\textbf{K}(\Sigma)) \to \text{Sh}(\textbf{VK})$ whose left adjoint is the restriction functor $I_{\Sigma}^*:\text{Sh}(\textbf{VK}) \to \text{Sh}(\textbf{K}(\Sigma))$.
\end{theorem}

\subsection{Embedding classical into virtual knots} \label{sec_embed} Goussarov-Polyak-Viro \cite{GPV} proved that if two classical knot diagrams are related by extended Reidemeister moves, then they are related by Reidemeister moves alone. This result can be lifted to the $2$-category $\mathfrak{Top}$ using the results of Section \ref{sec_project}. First note that by identifying $\mathbb{R}^3$ with $\mathbb{R}^2 \times \mathbb{R}$, we obtain a geometric morphism $I_{\mathbb{R}^2}:\text{Sh}(\textbf{K}) \to \text{Sh}(\textbf{VK})$. Next we describe the group of a virtual knot and its peripheral system as geometric morphisms. 
\newline
\newline
For a knot $K \in \mathbb{K}(\Sigma)$, let $\widebar{\Sigma} \subseteq \Sigma$ be the component of $\Sigma$ over which $K$ lives. Define $\pi_1(K)=\pi_1((\widebar{\Sigma} \times \mathbb{R} \smallsetminus K) / (\widebar{\Sigma} \times [N,\infty)),x_0)$, where $N \in \mathbb{R}$ is chosen so large that the image of $K$ is below $\widebar{\Sigma} \times N$ and $x_0 \in \widebar{\Sigma} \times N$. The group $\pi_1(K)$ is a Wirtinger group of deficiency $0$ or $1$ (see e.g. \cite{kim}). Denote by $\mathbb{WG}$ the set of isomorphism classes of such groups with the discrete topology. Define a function $\pi_1:\pi_0(\smallint_{\textbf{VK}} F) \to \mathbb{WG}$ by $\pi_1(\textbf{V})=\pi_1(K)$ if $(U,K)$ is in $\textbf{V}$. This is well-defined because an embedding $\psi: \Sigma \to \Xi$ of surfaces induces an isomorphism $\pi_1(K) \cong \pi_1(\Psi(K))$ on fundamental groups. Hence, $\pi_1$ is a constant isomorphism class on each object in a connected component of $\smallint_{\textbf{VK}} F$. Furthermore, $\pi_1$ is a variable-space isotopy invariant. To see this, note that $\pi_1(K)$ is an isotopy invariant of knots in a fixed $\Sigma \times \mathbb{R}$. Then since $\pi_1:\pi_0(\smallint_{\textbf{VK}} F) \to \mathbb{WG}$ is well-defined, Lemma \ref{lemma_cont_implies_isotop_var_amb} implies that $\pi_1:\pi_0(\smallint_{\textbf{VK}} F) \to \mathbb{WG}$ preserves the isomorphism class of isotopic variable-space knots. By Theorem \ref{thm_invariants}, there is a geometric morphism $\pi_1: \text{Sh}(\textbf{VK}) \to \text{Sh}(\textbf{WG})$. 
\newline
\newline
For $K \in \mathbb{K}(\Sigma)$, choose a meridian-longitude pair $(m,\ell)$ with $m,\ell \in \pi_1(K)$. The triple $(\pi_1(K),m,\ell)$ is a \emph{peripheral system} of $K$ (see \cite{bz}, Chapter 3.C). The elements $m$ and $\ell$ commute and are uniquely determined up to conjugacy. Let $(G,a,b)$ be a triple with $G$ a Wirtinger group, $a, b \in G$, $a \ne 1$, and $ab=ba$. An isomorphism of triples $(G_1,a_1,b_1)$ and $(G_2,a_2,b_2)$ is a group isomorphism $f:G_1 \to G_2$ such that $f(a_1)=a_2$ and $f(b_1)=b_2$. Let $\mathbb{PS}$ be the set of isomorphism classes of these triples with the discrete topology. The isomorphism class is an isotopy invariant of knots in $\Sigma \times \mathbb{R}$. If $\psi:\Sigma \to \Xi$ is a map, then $(\pi_1(\Psi(K)),\Psi_*(m),\Psi_*(\ell))$ is a peripheral system on $\Psi(K)$ that is isomorphic to $(\pi_1(K),m,\ell)$. Thus, the peripheral system also defines a variable-space knot invariant $\varpi_1:\pi_0(\smallint_{\textbf{VK}} F) \to \mathbb{WG}$ with corresponding geometric morphism $\varpi_1:\text{Sh}(\textbf{VK}) \to \text{Sh}(\textbf{PS})$.

\begin{theorem} If $K, \widebar{K}: \textbf{Sets} \to \text{Sh}(\textbf{K})$ are knots and $I_{\mathbb{R}^2} \circ K$, $I_{\mathbb{R}^2} \circ \widebar{K}: \textbf{Sets} \to \text{Sh}(\textbf{VK})$ are virtually isotopic, then $K,\widebar{K}$ are isotopic.
\end{theorem}

\begin{proof} Since $\mathbb{PS}$ is Hausdorff, $\varpi_1 \circ I_{\mathbb{R}^2} \circ K,\varpi_1 \circ I_{\mathbb{R}^2} \circ \widebar{K}: \textbf{Sets} \to \text{Sh}(\textbf{PS})$ yield the same continuous function $\mathbbm{1} \to \mathbb{PS}$ (see Proposition \ref{prop_convert_to_2_cat}). Then the corresponding knots $K^{\mathbbm{1}}_{\mathbb{K}},\widebar{K}^{\mathbbm{1}}_{\mathbb{K}}$ in $\mathbb{R}^3$ have isomorphic peripheral systems. Knots in $\mathbb{R}^3$ with isomorphic peripheral systems are isotopic (\cite{bz}, Theorem 3.15). Thus, there is a geometric morphism $\sigma:\text{Sh}(\textbf{I}) \to \text{Sh}(\textbf{K})$ from $K$ to $\widebar{K}$. 
\end{proof}

\section{Further discussion} \label{sec_further}

\subsection{Expanded languages for virtual knot theory} The language $\mathfrak{L}_g$ of geometric knot theory considered here has only ``knots'', ``isotopies'', and ``knot invariants''. The construction of $\text{Sh}(\textbf{VK})$ is sufficiently general that it can be applied to other knot-like things, such as links, framed knots, Seifert surfaces, spanning surfaces, etc. Some details are outlined below.
\newline
\newline
For the case of virtual links, let $\mathbb{L}_n(\Sigma)$ be the space of $n$-component links in $\Sigma \times \mathbb{R}$ and define a category $\textbf{VL}_n$ whose objects are open sets $U \subseteq \mathbb{L}_n(\Sigma)$ and whose arrows are capitalized maps $\Psi:U \to V$ such that $\Psi(U) \subseteq V$. A Grothendieck topology $J_{\textbf{VL}_n}$ is defined in the same way as for knots. The points of $\text{Sh}(\textbf{VL}_n,J_{\textbf{VL}_n})$ are virtual links and virtual isotopies are again geometric morphisms $\text{Sh}(\textbf{I}) \to \text{Sh}(\textbf{VL}_n,J_{\textbf{VL}_n})$. Likewise, there is a site $(\textbf{VK}^{fr},J_{\textbf{VK}^{fr}})$ for framed knots. Rather than using the space of embeddings of $\mathbb{S}^1 \to \Sigma \times \mathbb{R}$, this is constructed from the space of embeddings of ribbons $\mathbb{S}^1 \times \mathbb{I} \to \Sigma \times \mathbb{R}$. Again, the points of $\text{Sh}(\textbf{VK}^{fr})=\text{Sh}(\textbf{VK}^{fr},J_{\textbf{VK}^{fr}})$ are framed virtual knots, framed virtual isotopies are geometric morphisms $\text{Sh}(\textbf{I}) \to \text{Sh}(\textbf{VK}^{fr})$, and framed virtual knot invariants are geometric morphisms $\text{Sh}(\textbf{VK}^{fr}) \to \text{Sh}(\textbf{G})$. Sites for Seifert surfaces of homologically trivial knots in $\Sigma \times \mathbb{R}$, and spanning surfaces of $\mathbb{Z}_2$-homologically trivial knots in $\Sigma \times \mathbb{R}$ may be likewise constructed.
\newline
\newline
Furthermore, if $\mathbb{L}_n(\Sigma)$ is the space of $n$-component links in $\Sigma \times \mathbb{R}$, the result of Section \ref{sec_project} can be modified to obtain a geometric morphism $\text{Sh}(\textbf{L}_n(\Sigma)) \to \text{Sh}(\textbf{VL}_n)$. For $1 \le k \le n$, there is also a geometric morphism $C_k:\text{Sh}(\textbf{VL}_n) \to \text{Sh}(\textbf{VK})$ which takes the $k$-th component of an $n$-component virtual link.  Its left adjoint $C_k^*$ can be described as follows. An object $U$ of $\textbf{VL}_n$, is an open set of links $L:\bigsqcup_{i=1}^n \mathbb{S}^1 \to \Sigma\times \mathbb{R}$ for some $\Sigma$. Let $\kappa_k:\mathbb{S}^1 \to \bigsqcup_{i=1}^n \mathbb{S}^1$ be the inclusion into the $k$-th copy of $\mathbb{S}^1$. Let $U_k$ denote the set of compositions $L \circ \kappa_k$ for $L \in U$. Then $U_k$ is open in $\mathbb{K}(\Sigma)$. For $P \in \text{Sh}(\textbf{VK})$, define $C_k^*(P) \in \text{Sh}(\textbf{VL}_n)$ on objects $U$ by $C_k^*(P)(U)=P(U_k)$. For an arrow $\eta:P \to Q$, $C_k^*(\eta)$ is defined on components by $C_k^*(\eta)_U=\eta_{U_k}$. An argument similar to the proof of Lemma \ref{lemma_restrict_cont_filtering} shows that $C_k^* \circ \textbf{a} \circ \textbf{y}:\textbf{VK} \to \text{Sh}(\textbf{VL}_n)$ is continuous and filtering. Hence, $C_k^*$ is the left adjoint of a geometric morphism $C_k:\text{Sh}(\textbf{VL}_n) \to \text{Sh}(\textbf{VK})$. 
\newline
\newline
Another example of this kind is the geometric morphism $\partial:\text{Sh}(\textbf{VK}^{fr}) \to \text{Sh}(\textbf{VL}_2)$ that takes the boundary of a ribbon. Composing this with $C_1:\text{Sh}(\textbf{VL}_2) \to \text{Sh}(\textbf{VK})$ gives a geometric morphism $C_1 \circ \partial: \text{Sh}(\textbf{VK}^{fr}) \to \text{Sh}(\textbf{VK})$. If $K:\textbf{Sets} \to \text{Sh}(\textbf{VK}^{fr})$ is some framed virtual knot, the geometric morphism $C_1 \circ \partial \circ K$ is the virtual knot obtained by deleting the framing of $K$.

\subsection{Long virtual knots, welded knots, etc.} \label{sec_long_welded_etc} There are many diagrammatic knot theories. Some examples are long virtual knots, welded knots, free knots, and flat virtual knots. It is desirable to have sites for these theories as well. The case of welded knots and long virtual knots are of particular interest due to their applications to other areas of low-dimensional topology. 
\newline
\newline
In the welded case, Satoh \cite{satoh} showed that there is a map $\text{Tube}$ from the set of welded knot diagrams to broken-surface diagrams of ribbon torus knots in $\mathbb{R}^4$. The map $\text{Tube}$ is surjective and welded equivalent diagrams are mapped to isotopic ribbon torus knots. Thus, if $\mathbb{RTK}$ is the space of ribbon torus knots in $\mathbb{R}^4$, there is a Grothendieck topos $\text{Sh}(\textbf{RTK})$ whose points $\textbf{Sets} \to \text{Sh}(\textbf{RTK})$ are ribbon torus knots, and whose paths $\text{Sh}(\textbf{I}) \to \text{Sh}(\textbf{RTK})$ generate the isotopy relation for ribbon torus knots. It is therefore reasonable to expect that Tube itself can be realized as a geometric morphism $\text{Sh}(\textbf{VK}) \to \text{Sh}(\textbf{RTK})$ in $\mathfrak{Top}$. This is an open problem.
\newline
\newline
Long virtual knots appear in the study of finite-type invariants of classical knots \cite{GPV}. For a long virtual knot diagram $D$, there is also a well-defined virtual knot diagram $\text{Cl}(D)$, called its \emph{closure}. As is well-known, there are inequivalent long virtual knots that have equivalent closures (see e.g \cite{bart_fenn_kam_kam_long,sil_wil_long}). The problem in this case is to find a site $(\textbf{VLK},J_{\textbf{LVK}})$ for long virtual knots such that the closure map is realized as a geometric morphism $\text{Cl}:\text{Sh}(\textbf{LVK}) \to \text{Sh}(\textbf{VK})$.
  
\subsection{Sheaf cohomology of the virtual knot space} In \cite{vassiliev}, Vassiliev defined the finite-type knot invariants from a spectral sequence on the cohomology of the knot space. Can the finite-type invariants of virtual knots \cite{GPV,KaV} valued in an abelian group $G$ likewise be obtained by applying the analogous spectral sequence argument to the sheaf cohomology groups $H^i(\text{Sh}(\textbf{VK}),\Delta_{\mathbb{G}})$? Some parts of the construction carry over without difficulty. For example, the spaces $\mathbb{K}^{\bullet}_n(\Sigma)$ of singular knots with $n$ singular crossings can be used to obtain a site $(\textbf{VK}^{\bullet}_n,J_{\textbf{VK}^{\bullet}_n})$ exactly as in Section \ref{sec_site}. The Grothendieck topos $\text{Sh}(\textbf{VK}^{\bullet}_n)$ of sheaves on this site can then be taken as the space of singular virtual knots. Again, geometric morphisms $\textbf{Sets} \to \text{Sh}(\textbf{VK}^{\bullet}_n)$ correspond to singular virtual knots, geometric morphisms $\text{Sh}(\textbf{I}) \to \text{Sh}(\textbf{VK}^{\bullet}_n)$ generate the isotopy relation for singular virtual knots, and geometric morphisms $\text{Sh}(\textbf{VK}^{\bullet}_n) \to \text{Sh}(\textbf{G})$ are invariants of singular virtual knots.
\newline
\newline
The main problem is to identify these sheaf-theoretic singular virtual knots in the spectral sequence calculation. It is important to point out that the calculation in \cite{vassiliev} is that long knots are used in place of closed knots. In the classical case, long and closed knots are in one-to-one correspondence. As mentioned in Section \ref{sec_long_welded_etc}, this is not true in general for non-classical long virtual knots. The conjectured closure arrow $\text{Cl}:\text{Sh}(\textbf{LVK}) \to \text{Sh}(\textbf{VK})$ will thus play a central role in recovering the finite-type invariants of virtual knots from the sheaf cohomology of $\text{Sh}(\textbf{VK})$.

\subsection{Independence theorems} Are there theorems about classical knots that can only be proved using virtual knot theory? There are results about classical knots whose only \emph{known} proof uses virtual knots. Abel \cite{abel} used virtual link diagrams to prove that HOMFLY-PT homology of classical links is not invariant under non-braid-like isotopies. Hence, HOMFLY-PT homology can only be defined on diagrams presented as braid closures. The possibility of a strictly classical proof, however, has not yet been eliminated. A separate technology is needed to establish independence. 
\newline
\newline
Here we have applied topos theory from the perspective of generalized spaces. Alternatively, topoi may be viewed as generalized set theories. As such, they provide a natural context in which to frame questions about independence. For example, topos theory is sufficiently powerful to prove the independence of the axiom of choice and the continuum hypothesis from the Zermelo-Fraenkel axioms of set theory (see \cite{mac_moer}, Chapter VI). Since the virtual knot space $\text{Sh}(\textbf{VK})$ is a Grothendieck topos, these techniques can be applied to solve independence problems in knot theory as well. 


\subsection*{Acknowledgments} The author would like to express his gratitude to: M. Abel, C. Caprau, J. S. Carter, D. Freund, M. Harper, J. Joseph, S. Marshall, S. Mukherjee, D. Penneys, N. Petit, P. Pongtanapaisan, A. Poudel, R. Todd, and P. Ulrickson. The Ohio State University, Marion Campus, provided funding and release time for this paper.

\bibliographystyle{plain}
\bibliography{vks_bib}

\end{document}